\documentclass[reqno,10pt]{amsart}

\usepackage{amsmath,amssymb,mathrsfs,graphicx,enumitem,verbatim}

\usepackage[pdftex=true,
	bookmarks=true,
	bookmarksopen=true,
	bookmarksnumbered=true,
	pdftoolbar=true,
	pdfmenubar=true,
	pdffitwindow=false,
	pdfstartview={FitH},
	colorlinks=true,
	linkcolor=red,
	citecolor=blue,
	breaklinks=true]{hyperref}

\allowdisplaybreaks

\numberwithin{equation}{section}

\newtheorem{thm}{Theorem}[section]
\newtheorem{prop}[thm]{Proposition}
\newtheorem{lemma}[thm]{Lemma}
\newtheorem{cor}[thm]{Corollary}
\newtheorem*{thm*}{Theorem}
\newtheorem*{prop*}{Proposition}
\newtheorem*{cor*}{Corollary}
\newtheorem*{conj*}{Conjecture}

\theoremstyle{definition}
\newtheorem{definition}[thm]{Definition}

\theoremstyle{remark}
\newtheorem{rmk}[thm]{Remark}

\newcommand{\Ran}{\operatorname{Ran}}
\newcommand{\Ker}{\operatorname{Ker}}
\newcommand{\ecup}{\overline\cup}

\newcommand{\ms}{\mathscr}
\newcommand{\sB}{\ms B}
\newcommand{\sE}{\ms E}
\newcommand{\sF}{\ms F}
\newcommand{\sP}{\ms P}

\newcommand{\la}{\langle}
\newcommand{\ra}{\rangle}
\newcommand{\pa}{\partial}
\newcommand{\codim}{\operatorname{codim}}
\newcommand{\tn}{\textnormal}
\newcommand{\ff}{\tn{ff}}
\newcommand{\mc}{\mathcal}
\newcommand{\R}{\mathbb R}
\newcommand{\RR}{\mathbb R}
\newcommand{\C}{\mathbb C}
\newcommand{\Cx}{\mathbb C}
\newcommand{\N}{\mathbb N}
\newcommand{\NN}{\mathbb N}

\newcommand{\ZZ}{\mathbb Z}

\newcommand{\Sphere}{\mathbb S}
\newcommand{\eps}{\epsilon}
\newcommand{\ep}{\epsilon}
\newcommand{\im}{\operatorname{Im}}
\newcommand{\re}{\operatorname{Re}}

\newcommand{\supp}{\operatorname{supp}}
\newcommand{\sgn}{\operatorname{sgn}}

\newcommand{\psdo}{$\Psi$DO}

\newcommand{\xra}{\xrightarrow}

\renewcommand{\Im}{\operatorname{Im}}
\renewcommand{\Re}{\operatorname{Re}}
\newcommand{\Id}{\operatorname{Id}}

\newcommand{\CI}{C^\infty}
\newcommand{\dCI}{\dot C^\infty}
\newcommand\cO{\mathcal{O}}
\newcommand{\cL}{\mathcal L}
\newcommand{\cR}{\mathcal R}
\newcommand{\cP}{\mathcal P}
\newcommand{\cQ}{\mathcal Q}
\newcommand{\cX}{\mathcal X}
\newcommand{\cY}{\mathcal Y}
\newcommand{\cI}{\mathcal I}
\newcommand{\cS}{\mathcal S}
\newcommand{\Diff}{\mathrm{Diff}}
\newcommand{\Diffb}{\Diff_\bl}
\newcommand{\Vf}{\mathcal V}
\newcommand{\Vb}{\Vf_\bl}
\newcommand{\Vsc}{\Vf_\scl}
\newcommand{\Psib}{\Psi_\bl}
\newcommand{\bl}{{\mathrm{b}}}
\newcommand{\scl}{{\mathrm{sc}}}
\newcommand{\Hb}{H_{\bl}}
\newcommand{\Hbloc}{H_{\bl,\loc}}
\newcommand{\WFb}{\WF_{\bl}}
\newcommand{\WF}{\mathrm{WF}}

\newcommand{\Tb}{{}^{\bl}T}
\newcommand{\rcTb}{{}^{\bl}\overline{T}}
\newcommand{\Tsc}{{}^{\scl}T}
\newcommand{\Sb}{{}^{\bl}S}
\newcommand{\Nb}{{}^{\bl}N}
\newcommand{\SNb}{{}^{\bl}SN}
\newcommand{\loc}{{\mathrm{loc}}}
\newcommand{\sH}{\mathsf{H}}
\newcommand{\Op}{\operatorname{Op}}

\newcommand{\frakt}{\mathfrak{t}}

\newcommand{\Ell}{\mathrm{Ell}}
\newcommand{\bdiff}{{}^{\bl}d}
\newcommand{\scdiff}{{}^{\scl}d}

\newcommand{\bnormiso}{\mathcal{H}_{\bl,\Gamma}}
\newcommand{\region}{\Omega}

\newcommand{\KG}{\mathrm{KG}}
\newcommand{\I}{\mathrm{I}}
\newcommand{\ol}{\overline}
\newcommand{\wt}{\widetilde}
\newcommand{\wh}{\widehat}
\newcommand{\ftrans}{\;\!\widehat{\ }\;\!}

\begin{document}
\title[Semilinear wave equations]{Semilinear wave equations on asymptotically de Sitter, Kerr-de Sitter and Minkowski spacetimes}

\author{Peter Hintz and Andras Vasy}
\address{Department of Mathematics, Stanford University, CA
  94305-2125, USA}
\email{phintz@math.stanford.edu}
\email{andras@math.stanford.edu}
\date{June 19, 2013. Revised: April 9, 2015.}
\thanks{The authors were supported in part by A.V.'s National Science
  Foundation grants DMS-0801226 and DMS-1068742 and P. H.\ was
  supported in part by a Gerhard Casper Stanford Graduate Fellowship and the German National Academic Foundation.}

\begin{abstract}
  In this paper we show the small data solvability of suitable semilinear wave and Klein-Gordon equations on geometric classes of spaces, which include so-called asymptotically de Sitter and Kerr-de Sitter spaces, as well as asymptotically Minkowski spaces. These spaces allow general infinities, called conformal infinity in the asymptotically de Sitter setting; the Minkowski type setting is that of non-trapping Lorentzian scattering metrics introduced by Baskin, Vasy and Wunsch.  Our results are obtained by showing the {\em global} Fredholm property, and indeed invertibility, of the underlying linear operator on suitable $L^2$-based function spaces, which also possess appropriate algebra or more complicated multiplicative properties. The linear framework is based on the b-analysis, in the sense of Melrose, introduced in this context by Vasy to describe the asymptotic behavior of solutions of linear equations. An interesting feature of the analysis is that {\em resonances}, namely poles of the inverse of the Mellin transformed b-normal operator, which are `quantum' (not purely symbolic) objects, play an important role.
\end{abstract}

\maketitle

\section{Introduction}

In this paper we consider semilinear wave equations in contexts such as asymptotically de Sitter and Kerr-de Sitter spaces, as well as asymptotically Minkowski spaces. The word `asymptotically' here does {\em not} mean that the asymptotic behavior has to be that of exact de Sitter, etc., spaces, or even a perturbation of these at infinity; much more general infinities, that nonetheless possess a similar structure as far as the underlying analysis is concerned, are allowed.  Recent progress \cite{Va12,Ba13} allows one to set up the analysis of the associated linear problem {\em globally} as a Fredholm problem, concretely using the framework of Melrose's b-pseudodifferential operators \cite{Me93} on appropriate compactifications $M$ of these spaces. (The b-analysis itself originates in Melrose's work on the propagation of singularities for the wave equation on manifolds with smooth boundary, and Melrose described a systematic framework for elliptic b-equations in \cite{Me93}. Here `b' refers to analysis based on vector fields tangent to the boundary of the space; we give some details later in the introduction and further details in \S\ref{SecStaticDeSitter-Fredholm}, where we recall the setting of \cite{Va12}.) This allows one to use the contraction mapping theorem to solve semilinear equations with small data in many cases since typically the semilinear terms can be considered perturbations of the linear problem. That is, as opposed to solving an evolution equation on time intervals of some length, possibly controlling this length in some manner, and iterating the solution using (almost) conservation laws, we solve the equation globally in one step.

As Fredholm analysis means that one has to control the linear operator
$L$ modulo compact errors, which in these settings means modulo terms
which are {\em both} smoother and more decaying, the underlying linear
analysis involves both arguments based on the principal symbol of the
wave operator and on its so-called (b-)normal operator family, which
is a holomorphic family $\widehat N(L)(\sigma)$ of operators  on $\pa
M$. In settings in which there is a $\RR^+$-action in the normal
variable, and the operator is dilation invariant, this simply means
Mellin transforming in the normal variable. Replacing the normal
variable by its logarithm, this is equivalent to a Fourier transform.

At the principal symbol level one encounters real principal type
phenomena as well as radial points of the Hamilton flow at the
boundary of the compactified underlying space $M$; these allow for the
usual (for wave equations) loss of one (b-)derivative relative to
elliptic problems. Physically, in the de Sitter and Kerr-de Sitter type settings radial points correspond to a red shift effect. In Kerr-de Sitter spaces there is an additional loss of derivatives due to trapping. On the other hand, the b-normal operator family enters via the poles $\sigma_j$ of the meromorphic inverse $\widehat N(L)(\sigma)^{-1}$; these poles, called {\em resonances}, determine the decay/growth rates of solutions of the linear problem at $\pa M$, namely $\im\sigma_j>0$ means growing while $\im\sigma_j<0$ means decaying solutions.  Translated into the nonlinear setting, taking powers of solutions of the linear equation means that growing linear solutions become even more growing, thus the non-linear problem is uncontrollable, while decaying linear solutions become even more decaying, thus the non-linear effects become negligible at infinity. Correspondingly, the location of these resonances becomes crucial for non-linear problems. We note that in addition to providing solvability of semilinear problems, our results can also be used to obtain the {\em asymptotic expansion} of the solution.

In short, we present a \emph{systematic approach} to the analysis of
semilinear wave and Klein-Gordon equations: Given an appropriate
structure of the space at infinity and given that the location of the
resonances fits well with the non-linear terms, see the discussion
below, one can solve (suitable) semilinear equations. Thus, the main
purpose of this paper is to present the first step towards a general
theory for the global study of linear and nonlinear wave-type
equations; the semilinear applications we give are meant to show how
far we can get in the nonlinear regime using relatively simple means,
and lend themselves to meaningful comparisons with existing
literature, see the discussion below. In particular, our approach
readily generalizes to the analysis of \emph{quasilinear} equations,
provided one understands the necessary (b-)analysis for non-smooth
metrics. Since the first version of the present paper, the authors described such generalizations in detail in the context of asymptotically de Sitter \cite{HintzQuasilinearDS} and asymptotically Kerr-de Sitter spaces \cite{HintzVasyQuasilinearKdS}.

We now describe our setting in more detail. We consider semilinear wave equations of the form
$$
(\Box_g-\lambda)u=f+q(u,du)
$$
on a manifold $M$ where $q$ is (typically; more general functions are also considered) a polynomial vanishing at least quadratically at $(0,0)$ (so contains no constant or linear terms, which should be included either in $f$ or in the operator on the left hand side). The derivative $du$ is measured relative to the metric structure (e.g.\ when constructing polynomials in it).  Here $g$ and $\lambda$ fit in one of the following scenarios, which we state slightly informally, with references to the precise theorems. We discuss the terminology afterwards in more detail, but the reader unfamiliar with the terms could drop the word `asymptotically' and `even' to obtain specific examples.

\begin{enumerate}
\item
A neighborhood of the backward light cone from future infinity in an
asymptotically de Sitter space. (This may be called a static region/patch
of an asymptotically de Sitter space, even when there is no time like Killing
vector field.) In order to solve the semilinear equation,
if $\lambda>0$, one can allow $q$ an arbitrary polynomial with quadratic vanishing at
the origin, or indeed a more general function. If $\lambda=0$ and $q$
depends on $du$ only, the same conclusion holds. Further, in either
case, one obtains an
expansion of the solution at infinity. See
Theorems~\ref{ThmDSQu} and \ref{ThmDSPoly}, and Corollary~\ref{cor:ndS}.
\item
Kerr-de Sitter space, including a neighborhood of the event
horizon, or more general spaces with normally hyperbolic trapping, discussed below.
In the main part of the section we assume $\lambda>0$, and allow $q=q(u)$ with
quadratic vanishing at the origin. We also obtain an expansion at
infinity. See Theorems~\ref{ThmKdSQu} and \ref{ThmKdSPoly}, and
Corollary~\ref{cor:nKdS}. However, in \S\ref{sec:KdS-derivs} we
briefly discuss
non-linearities involving
derivatives which are appropriately behaved at the trapped
set.
\item
Global {\em even} asymptotically de Sitter spaces. These are in some sense the
easiest examples as they correspond, via extension across the
conformal boundary, to working on a manifold without
boundary.
Here $\lambda=(n-1)^2/4+\sigma^2$. While the equation is unchanged if one replaces $\sigma$ by
  $-\sigma$, the process of extending across the boundary
  breaks this symmetry, and in \S\ref{SecDeSitter} we mostly consider $\im\sigma\leq 0$.
If $\Im\sigma<0$ is sufficiently small and the dimension satisfies $n\geq 6$, quadratic
vanishing of $q$ suffices; if $n\geq 4$ then cubic vanishing is
sufficient. If $q$ does not involve derivatives, $\Im\sigma\geq 0$ small
also works, and if $\Im\sigma>0,n\geq 5$, or $\im\sigma=0$, $n\geq 6$, then quadratic vanishing of $q$ is sufficient.
See Theorems~\ref{ThmGlobalDSQu}, \ref{ThmGlobalDS} and
\ref{ThmGlobalDSLowReg}. Using the results from `static'
asymptotically de Sitter spaces, quadratic vanishing of $q$ in fact
suffices for all $\lambda>0$, and indeed $\lambda\geq 0$ if $q=q(du)$,
but the decay estimates for solutions are lossy relative to the {\em decay} of the forcing. See Theorem~\ref{ThmGSdSQu}.
\item
Non-trapping Lorentzian scattering (generalized asymptotically
Minkowski) spaces, $\lambda=0$.
If $q=q(du)$, we allow $q$ with quadratic
vanishing at $0$ if $n\geq 5$; cubic if $n\geq 4$. If $q=q(u)$, we
allow $q$ with quadratic vanishing if $n\geq 6$; cubic if $n\geq
4$. Further, for $q=q(du)$ quadratic satisfying a null condition,
$n=4$ also works. See Theorems~\ref{ThmMinkQu}, \ref{ThmMink} and \ref{ThmNullMink}.
\end{enumerate}

We now recall these settings in more detail.
First, see \cite{Va07}, an asymptotically de Sitter space
is an appropriate generalization of the Riemannian conformally compact
spaces of Mazzeo and Melrose \cite{Ma87}, namely a
smooth manifold with boundary, $\wt M$, with interior
$\wt M^\circ$ equipped with a Lorentzian metric $\wt g$, which we take to
be of signature $(1,n-1)$ for the sake of definiteness, and with a boundary
defining function $\rho$, such that $\wh g=\rho^2 \wt g$ is a
smooth symmetric 2-cotensor of signature $(1,n-1)$ up to
the boundary of $\wt M$ and $\wh g(d\rho,d\rho)=1$ (thus, the
boundary defining function is timelike, and thus the boundary is
spacelike; the $=1$ statement makes the curvature asymptotically
constant),
and in addition $\pa\wt M$ has two components (each of
which may be a union of connected components) $\wt X_\pm$, with all
null-geodesics $c=c(s)$ of $\wt g$ tending to $\wt X_+$ as $s\to +\infty$ and to
$\wt X_-$ as $s\to-\infty$, or vice versa. Notice that in
  the interior of $\wt M$, the conformal factor $\rho^{-2}$ simply
  reparameterizes the null-geodesics, so equivalently one can require
  that null-geodesics of $\wh g$ reach $\wt X_\pm$ at finite
  parameter values. Analogously to
asymptotically hyperbolic spaces, where this was shown by Graham and
Lee \cite{Gr91}, on such a space one can
always introduce a product decomposition $(\pa\wt M)_z\times
[0,\delta)_\rho$ near $\pa\wt M$ (possibly changing $\rho$) such that the metric has a
warped product structure $\wh g=d\rho^2-h(\rho,z,dz)$, $\wt
g=\rho^{-2}\wh g$; the metric is called even if $h$ can be taken even
in $\rho$, i.e.\ a smooth function of $\rho^2$. We refer to Guillarmou
\cite{Gu05} for the introduction of even metrics in
the asymptotically hyperbolic context, and to \cite{Va07},
\cite{Va12} and \cite{Va13} for further discussion.

Blowing up a point $p$ at $\wt X_+$, which essentially means
introducing spherical coordinates around it, we obtain a manifold with
corners $[\wt M;p]$, with a blow-down map $\beta:[\wt M;p]\to\wt M$, which is a diffeomorphism away from the {\em front face}, which
gets mapped to $p$ by $\beta$. Just like blowing up the origin in
Minkowski space desingularizes the future (or past) light cone, this
blow-up desingularizes the backward light cone from $p$ on $\wt M$, which lifts
to a smooth submanifold transversal to the front face on $[\wt
M;p]$ which intersects the front face in a sphere $Y$. The interior of this lifted backward light cone, at least near
the front face, is a generalization of the static patch in de Sitter
space, and we refer to a neighborhood $M_\delta$, $\delta>0$, of the closure of the
interior $M_+$
of the lifted backward light cone in $[\wt M;p]$ which only intersects
the boundary of $[\wt M;p]$ in the interior of the front face (so
$M_\delta$ is a non-compact manifold with boundary, with boundary $X_\delta$, and with say boundary defining function $\tau$) as
the `static' asymptotically de Sitter problem. See
Figure~\ref{FigBlowUp}. Via a doubling process, $X_\delta$ can be
replaced by a compact manifold without boundary, $X$, and $M_\delta$
by $M=X\times[0,\tau_0)_\tau$, an approach taken in \cite{Va12} where
complex absorption was used, or indeed one can instead work in a compact
region $\Omega\subset M_\delta$ by adding artificial, spacelike,
boundaries, as we do here in
\S\ref{SecStaticDeSitter-Fredholm}. With such an $\Omega$,
the distinction between $M$ and $M_\delta$ is irrelevant, and we
simply write $M$ below.

\begin{figure}[!ht]
  \centering
  \includegraphics{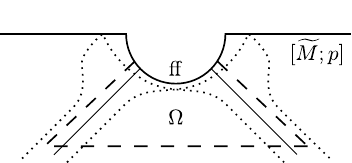}
  \caption{Setup of the `static' asymptotically de Sitter problem. Indicated are the blow-up of $\wt M$ at $p$ and the front face, the lift of the backward light cone to $[\wt M;p]$ (solid), and lifts of backward light cones from points nearby $p$ (dotted); moreover, $\Omega\subset M$ (dashed boundary) is a submanifold with corners within $M$ (which is not drawn here; see \cite{Va12} for a description of $M$ using a doubling procedure in a similar context). The role of $\Omega$ is explained in \S\ref{SecStaticDeSitter-Fredholm}.}
  \label{FigBlowUp}
\end{figure}

See \cite{Va07,Va12} for relating the `global' and the `static' problems. We note that the lift of $\wt g$ to $M$ in the static region is a Lorentzian b-metric, i.e.\ is a smooth symmetric section of signature $(1,n-1)$ of the second tensor power of the b-cotangent bundle, $\Tb^*M$. The latter is the dual of $\Tb M$, whose smooth sections are smooth vector fields on $M$ tangent to $\pa M$; sections of $\Tb^*M$ are smooth combinations of $\frac{d\tau}{\tau}$ and smooth one forms on $X$, relative to a product decomposition $X\times[0,\delta)_\tau$ near $X=\pa M$. See also \S\ref{SecStaticDeSitter-Fredholm}.

As mentioned earlier, the methods of \cite{Va12} work in a rather general b-setting,
including generalizations of `static' asymptotically de Sitter
spaces. Kerr-de Sitter space, described from this perspective in
\cite[\S6]{Va12}, can be thought of as such a
generalization. In particular, it still carries a Lorentzian b-metric, but
with a somewhat more complicated structure, of which the only
important part for us is that it has trapped rays. More concretely, it is
best to consider the bicharacteristic flow in the b-cosphere bundle
(projections of null-bicharacteristics being just the null-geodesics),
$\Sb^*M$, quotienting out by the $\RR^+$-action on the fibers of
$\Tb^*M\setminus o$. On the `static' asymptotically de Sitter space
each half of the spherical
b-conormal bundle $\SNb^*Y$ consists of (a family of) saddle points of the
null-bicharacteristic flow (these are called {\em radial sets}, the
stable/unstable directions are normal to $\SNb^*Y$ itself), with
one of the stable and unstable manifolds being the conormal bundle of
the lifted light cone (which plays the role of the {\em event horizon}
in black hole settings), and the other being the characteristic set
within the boundary $X$ (so within the boundary, the radial sets
$\SNb^*Y$, are actually sources or sinks). Then on asymptotically de Sitter spaces all
null-bicharacteristics over $\overline{M_+}\setminus X$ either leave $\Omega$ in finite time or
(if they lie on the conormal bundle of the event horizon)
tend to $\SNb^*Y$ as the parameter goes to $\pm\infty$, with each
bicharacteristic tending to $\SNb^*Y$ in at most
one direction. The main
difference for Kerr-de Sitter space is that there are
null-bicharacteristics which do not leave $\overline{M_+}\setminus X$ and do not tend to
$\SNb^*Y$. On de Sitter-Schwarzschild space (non-rotating black holes)
these future
trapped rays project to a sphere, called the photon sphere, times $[0,\delta)_\tau$; on general
Kerr-de Sitter space the trapped set deforms, but is still {\em normally
  hyperbolic}, a setting studied by Wunsch and Zworski in \cite{Wu11}
and by Dyatlov in \cite{Dy13}.

We refer to \cite[\S3]{Ba13} and to
\S\ref{SecMinkowski-Fredholm} here for a definition of asymptotically Minkowski
spaces, but roughly they are manifolds with boundary $M$ with
Lorentzian metrics $g$ on the interior $M^\circ$ conformal to a
b-metric $\wh g$ as $g=\tau^{-2}\wh g$, with $\tau$ a boundary
defining function\footnote{In \S\ref{SecMinkowski} we switch to
  $\rho$ as the boundary defining function for consistency with \cite{Ba13}.} (so these are Lorentzian {\em scattering} metrics in the sense
of Melrose \cite{Me94}, i.e.\ symmetric cotensors in the second power
of the scattering cotangent bundle, and of signature $(1,n-1)$), with a
real $\CI$ function $v$ defined on $M$ with $dv$, $d\tau$ linearly
independent at $S=\{v=0,\ \tau=0\}$, and with a specific
behavior of the metric at $S$ which reflects that of Minkowski space
on its radial compactification near the boundary of the light cone at
infinity (so $S$ is the light cone at infinity in this greater
generality). Concretely, the specific form is\footnote{More general,
  `long-range' scattering metrics also work for the purposes of this
  paper without any significant changes; the analysis of these is currently being completed by Baskin,
  Vasy and Wunsch. The difference is the presence of smooth multiples
  of $\tau\frac{d\tau^2}{\tau^2}$ in the metric near $\tau=0$, $v=0$. These do not affect
  the normal operator, but slightly change the dynamics in
  $\Sb^*M$. This, however, does not affect the function spaces to be
  used for our semilinear problem.}
$$
\tau^2 g=\wh
g=v\frac{d\tau^2}{\tau^2}-\Bigl(\frac{d\tau}{\tau}\otimes\alpha+\alpha\otimes\frac{d\tau}{\tau}\Bigr)-\wt h,
$$
where $\alpha$ is a smooth one form on $M$, equal to $\frac{1}{2}\,dv$
at $S$, $\wt h$ is a smooth
2-cotensor on $M$, which is positive definite on the annihilator of
$d\tau$ and $dv$ (which is a codimension 2 space).
The difference between the de Sitter-type and Minkowski
settings is in part this conformal factor, $\tau^{-2}$, but more
importantly, as this conformal factor again does not affect the
behavior of the null-bicharacteristics so one can consider those of
$\wh g$ on $\Sb^*M$, at the spherical conormal bundle $\SNb^*S$ of $S$ (see \S\ref{SecStaticDeSitter})
the nature of the radial points is source/sink rather than a saddle
point of the flow. (One also makes a non-trapping assumption in the
asymptotically Minkowski setting.)

Now we comment on the specific way these settings fit into the
b-framework, and the way the various restrictions described above arise.

\begin{enumerate}
\item
Asymptotically `static' de Sitter.
Due to a zero resonance for the linear problem when
$\lambda=0$, which moves to the lower half plane for $\lambda>0$, in
this setting $\lambda>0$ works in general; $\lambda=0$ works if
$q$ depends on $du$ but not on $u$. The relevant function spaces are
$L^2$-based b-Sobolev spaces (see \S\ref{SecStaticDeSitter}) on the bordification (partial
compactification) of the space, or analogous spaces plus a finite
expansion. Further, the semilinear terms involving $du$ have
coefficients corresponding to the b-structure, i.e.\ b-objects are
used to create functions from the differential forms, or equivalently
b-derivatives of $u$ are used.
\item
Kerr-de Sitter space. This is an extension of (1), i.e.\ the framework is
essentially the same, with the difference being
that there is now trapping corresponding to the `photon sphere'. This
makes first order terms in the non-linearity non-perturbative, unless
they are well-adapted to the trapping. Thus,
we assume $\lambda>0$. The
relevant function spaces are as in the asymptotically de Sitter setting.
\item
Global {\em even} asymptotically de Sitter spaces. In order to get reasonable results, one needs to
measure regularity relatively finely, using the module of vector fields
tangent to what used to be the conformal boundary in the
extension. The
relevant function spaces are thus Sobolev spaces with additional
(finite) conormal regularity. Further, $du$ has coefficients
corresponding to the 0-structure of Mazzeo and Melrose, in the same
sense the b-structure was used in (1).
The range of $\lambda$ here is limited
by the process of extension across the boundary; for non-linearities
involving $u$ only, the restriction amounts to (at least very slowly) decaying solutions for
the linear problem (without extension across the conformal boundary).

Another possibility is to view global de Sitter space as a union of static patches. Here, the b-Sobolev spaces on the static parts translate into 0-Sobolev spaces on the global space, which have weights that are shifted by a dimension-dependent amount relative to the weights of the b-spaces. This approach allows many of the non-linearities that we can deal with on static parts; however, the resulting decay estimates on $u$ are quite lossy relative to the decay of the forcing term $f$.
\item
Non-trapping Lorentzian scattering spaces (generalized asymptotically Min\-kowski spaces), $\lambda=0$. Note that if
$\lambda>0$, the type of the equation changes drastically; it
naturally fits into Melrose's scattering algebra\footnote{In many
  ways the scattering algebra is actually much better behaved than the
  b-algebra, in particular it is symbolic in the sense of
  weights/decay. Thus, with numerical modifications, our methods
  should extend directly.} rather than the
b-algebra which can be used for $\lambda=0$. While the results here are quite robust
and there are no issues with trapping, they are more involved as one
needs to keep track of regularity relative to the module of vector
fields tangent to the light cone at infinity. The relevant function
spaces are b-Sobolev spaces with additional b-conormal regularity
corresponding to the aforementioned module. Further, $du$ has coefficients
corresponding to Melrose's scattering structure. These spaces, in the
special case of Minkowski space, are related to the spaces used by
Klainerman \cite{Klainerman:Uniform}, using the infinitesimal
generators of the Lorentz group, but while Klainerman works in an
$L^\infty L^2$ setting, we remain purely in a (weighted) $L^2$ based
setting, as the latter is more amenable to the tools of microlocal analysis.
\end{enumerate}

We reiterate that while the way de Sitter, Minkowski, etc., type
spaces fit into it
differs somewhat, the underlying linear framework is that of $L^2$-based
b-analysis, on manifolds with boundary, except
that in the global view of asymptotically de Sitter spaces one can
eliminate the boundary altogether.

In order to underline the generality of the method, we emphasize that,
corresponding to cases (1) and (2), in b-settings in which one can
work on standard b-Sobolev spaces the restrictions on the
solvability of the semilinear equations are simply given by the presence of
resonances for the Mellin-transformed normal operator in
$\im\sigma\geq 0$, which would allow growing solutions to the
equation (with
  the exception of $\im\sigma=0$, in which case the non-linear iterative arguments
  produce growth unless the non-linearity has a special structure), making the non-linearity non-perturbative, and the losses at high energy estimates for this
Mellin-transformed operator and the closely related b-principal symbol
estimates when one has trapping. (It is these losses that cause the
  difference in the trapping setting between non-linearities with or
  without derivatives.) In particular, the results are necessarily
optimal in the non-trapping setting of (1), as shown even by an ODE,
see Remark~\ref{rmk:ODE}. In the trapping setting it is not clear
precisely what improvements are possible for non-linearities with
derivatives, though when there are no derivatives in the
non-linearity, we already have no restrictions on the non-linearity
and to this extent the result is optimal.

On Lorentzian scattering spaces more general function spaces are used, and it is not in principle clear whether the results are optimal, but at least comparison with the work of Klainerman and Christodoulou for perturbations of Minkowski space \cite{Ch86,Klainerman:Uniform, Klainerman:Null} gives consistent results; see the comments below. On global asymptotically de Sitter spaces, the framework of \cite{Va12} and \cite{Vasy:Microlocal-AH} is very convenient for the linear analysis, but it is not clear to what extent it gives optimal results in the non-linear setting. The reason why more precise function spaces become necessary is the following: There are two basic properties of spaces of functions on manifolds with boundaries, namely differentiability and decay. Whether one can have both at the same time for the linear analysis depends on the (Hamiltonian) dynamical nature of radial points: when defining functions of the corresponding boundaries of the compactified cotangent bundle have opposite character (stable vs.\ unstable) one can have both at the same time, otherwise not; see Propositions~\ref{prop:b-saddle} and \ref{prop:b-source-sink} for details.  For non-linear purposes, the most convenient setting, in which we are in (1), is if one can work with spaces of arbitrarily high regularity and fast decay, and corresponds to saddle points of the flow in the above sense. In (4) however, working in higher regularity spaces, which is necessary in order to be able to make sense of the non-linearity, requires using faster growing (or at least less decaying) weights, which is problematic when dealing with non-linearities (e.g., polynomials) since multiplication gives even worse growth properties then. Thus, to make the non-linear analysis work, the function spaces we use need to have more structure; it is a module regularity that is used to capture some weaker regularity in order to enable work in spaces with acceptable weights.

While all results are stated for the scalar equation, analogous
results hold in many cases for operators on natural vector bundles,
such as the d'Alembertian (or Klein-Gordon operator) on differential
forms, since the linear arguments work in general for operators with
scalar principal symbol whose subprincipal symbol satisfies
appropriate estimates at radial sets, see \cite[Remark~2.1]{Va12},
though of course for semilinear applications the presence of
resonances in the closed upper half plane has to be checked. This
already suffices to obtain the well-posedness of the semilinear
equations on asymptotically de Sitter spaces that we consider in this
paper; for this purpose one needs to know the poles of
  the resolvent of the Laplacian on forms on {\em exact} hyperbolic
  space only. On asymptotically Minkowski spaces, the absence of poles of an
asymptotically hyperbolic resolvent in a region has to be checked
in addition, see Theorem~\ref{thm:asymp-Mink-lin},
and the numerology depends crucially on the delicate balance of
weights and regularity, as alluded to above. Note that on {\em
  perturbations} of Minkowski space, this absence of poles
  follows from the appropriate behavior of the poles of the resolvent of
  the Laplacian on forms on
  {\em exact} hyperbolic space.

The degree to which these non-linear problems have been studied
differ, with the Minkowski problem (on perturbations of Minkowski
space, as opposed to our more general setting) being the most
studied. There semilinear and indeed even quasilinear equations are
well understood due to the work of Christodoulou \cite{Ch86} and
Klainerman \cite{Klainerman:Uniform, Klainerman:Null}, with their
book on the global stability of Einstein's equation
\cite{Christodoulou-Klainerman:Global} being one of the main
achievements. (We also refer to the work of Lindblad and Rodnianski
\cite{Lindblad-Rodnianski:Global-existence,
  Lindblad-Rodnianski:Global-Stability} simplifying some of the
arguments, of Bieri \cite{Bieri:Extensions, Bieri-Zipser:Extensions}
relaxing some of the decay conditions, of Wang
\cite{Wang:Thesis} obtaining asymptotic expansions, and of Lindblad \cite{Lindblad:Global} for results on a
class of quasilinear equations. H\"ormander's book
\cite{Hormander:Nonlinear} provides further references in the general
area. There are numerous works on the {\em linear} problem, and
estimates this yields for the non-linear problems, such as Strichartz
estimates; here we refer to the recent work of Metcalfe and Tataru
\cite{Metcalfe-Tataru:Global} for a parametrix construction in low
regularity, and references therein.) Here we obtain results comparable to these (when
restricted to the semilinear setting), on a larger class of manifolds,
see Remark~\ref{rmk:Christodoulou}. For non-linearities which do not
involve derivatives, slightly stronger results have been obtained, in
a slightly different setting, in \cite{Ch06}; see
Remark~\ref{rmk:Chrusciel}.

On the other hand, there is little (non-linear) work on the
asymptotically de Sitter and Kerr-de Sitter settings; indeed the only
paper the authors are aware of is that of Baskin \cite{Baskin:Strichartz} in
roughly comparable generality in terms of the setting,
though in {\em exact} de Sitter space
Yagdjian \cite{Yagdjian:Global, Yagdjian:Semilinear} has studied
a large class of semilinear equations with no derivatives. Baskin's
result is for a semilinear equation with no derivatives and a single
exponent, using his parametrix construction \cite{Baskin:Parametrix},
namely $u^p$ with\footnote{The dimension of the spacetime in Baskin's
  paper is $n+1$; we continue using our notation above.}
$p=1+\frac{4}{n-2}$, and for $\lambda>(n-1)^2/4$. In the same setting,
$p>1+\frac{4}{n-1}$ works for us, and thus Baskin's setting is in
particular included. Yagdjian works with the explicit solution
operator (derived using special functions) in exact de Sitter space,
again with no derivatives in the non-linearity. While there are some
exponents that his results cover (for $\lambda>(n-1)^2/4$, all $p>1$
work for him) that ours do not directly (but indirectly, via the
static model, we in fact obtain such results), the range
$(\frac{(n-1)^2}{4}-\frac{1}{4},\frac{(n-1)^2}{4})$ is excluded by him while
covered by our work for sufficiently large $p$. In the (asymptotically) Kerr-de Sitter
setting, to our knowledge, there has been no similar semilinear work,
however Luk \cite{LukKerrNonlinear} and Tohaneanu
\cite{TohaneanuKerrStrichartz} studied semilinear waves on Kerr
spacetimes. We recall finally that there
  is more work on the linear problem in de Sitter, de
  Sitter-Schwarzschild and Kerr-de Sitter spaces. We refer to
  \cite{Va12} for more detail; some references are Polarski
  \cite{Polarski:Hawking}, Yagdjian
and Galstian \cite{Yagdjian-Galstian:De-Sitter}, S\'a Barreto and Zworski
\cite{Sa-Barreto-Zworski:Distribution}, Bony and H\"afner
\cite{Bony-Haefner:Decay},
Vasy \cite{Va07}, Baskin \cite{Baskin:Parametrix}, Dafermos and
Rodnianski \cite{Dafermos-Rodnianski:Sch-dS}, 
Melrose, S\'a Barreto and Vasy
\cite{Melrose-SaBarreto-Vasy:Asymptotics},
Dyatlov \cite{Dy11a, Dy11b}. Also, while it received more attention,
the linear problem on Kerr space does not fit directly
into our setting; see the introduction of \cite{Va12} for an
explanantion and for further references,
\cite{Dafermos-Rodnianski:Lectures} for more background and additional
references.

While the basic ingredients of the necessary linear b-analysis were
analyzed in \cite{Va12}, the solvability framework was only discussed
in the dilation invariant setting, and in general the asymptotic
expansion results were slightly lossy in terms of derivatives in the
non-dilation invariant case. We remedy these issues in this paper,
providing a full Fredholm framework. The key technical tools are the
propagation of b-singularities at b-radial points which are
saddle points of the flow in $\Sb^*M$, see
Proposition~\ref{prop:b-saddle}, as well as the b-normally hyperbolic
versions, proved in \cite{HintzVasyNormHyp}, of the semiclassical
normally hyperbolic trapping estimates of Wunsch and Zworski \cite{Wu11}; the
rest of the Fredholm setup is discussed in
\S\ref{SecStaticDeSitter-Fredholm} in the non-trapping and
\S\ref{SecKdS-Fredholm} in the normally hyperbolic trapping setting. The analogue of
Proposition~\ref{prop:b-saddle} for sources/sinks was already proved
in \cite[\S4]{Ba13}; our Lorentzian scattering metric Fredholm
discussion, which relies on this, is in
\S\ref{SecMinkowski-Fredholm}.

We emphasize that our analysis
would be significantly less cumbersome in terms of technicalities if we were not including Cauchy
hypersurfaces and solved a globally well-behaved problem by imposing
sufficiently rapid decay at past infinity instead (it is standard to
convert a Cauchy problem into a forward solution problem). Cauchy
hypersurfaces are only necessary for us if we deal with a problem
ill-behaved in the past because complex absorption does not force
appropriate forward supports even though it does so at the level of singularities; otherwise we can work with appropriate
(weighted) Sobolev spaces. The latter is the
case with Lorentzian scattering spaces, which thus provide an ideal
example for our setting. It can also be done in the global setting of
asymptotically de Sitter spaces, as in setting (3) above, essentially by realizing
these as the boundary of the appropriate compactification of a
Lorentzian scattering space, see \cite{Va13}. In the case of Kerr-de
Sitter black holes, in the presence of dilation invariance, one has
access to a similar luxury; complex absorption does the job as in
\cite{Va12}; the key aspect is that it needs to be imposed {\em
  outside} the static region we consider. For a general Lorentzian b-metric with a normally
hyperbolic trapped set, this may not be easy to arrange, and we do
work by adding Cauchy hypersurfaces, even at the cost of the
resulting, rather
artificial in terms of PDE theory, technical complications. For
perturbations of Kerr-de Sitter space, however, it is possible to
forego the latter for well-posedness by an appropriate gluing to
complete the space with actual Kerr-de Sitter space in the past for
the purposes of functional analysis. We remark that Cauchy
hypersurfaces are somewhat ill-behaved for $L^2$ based estimates,
which we use, but
match $L^\infty L^2$ estimates quite well, which explains the large
role they play in existing hyperbolic theory, such as
\cite{Klainerman:Uniform} or \cite[Chapter~23.2]{Ho83}. We hope that
adopting this more commonly used form of `truncation' of hyperbolic
problems will aid the readability of the paper.

We also explain the role that the energy estimates (as opposed to
microlocal energy estimates) play. These mostly enter to deal with the
artificially introduced boundaries; if other methods were used to
truncate the flow, their role reduces to checking that in certain
cases, when the microlocal machinery only guarantees Fredholm
properties of the underlying linear operators, the potential finite
dimensional kernel and cokernel are indeed trivial. Asymptotically
Minkowski spaces illustrate this best, as the Hamilton flow is globally
well-behaved there; see \S\ref{SecMinkowski-Fredholm}.

The other key technical tool is the algebra property of b-Sobolev
spaces and other spaces with additional conormal regularity. These are
stated in the respective sections; the case of the standard b-Sobolev
spaces reduces to the algebra property of the standard Sobolev spaces
on $\RR^n$. Given the algebra properties, the
results are proved by applying the contraction mapping theorem to the linear operator.

In summary,
the plan of this paper is the following. In each of the sections below
we consider one of these settings, and first describe the Sobolev spaces on
which one has invertibility for the linear problems of interest, then
analyze the algebra properties of these Sobolev spaces, finally
proving the solvability of the semilinear equations by checking that
the hypotheses of the contraction mapping theorem are satisfied.

The authors are grateful to Dean Baskin, Rafe Mazzeo, Richard Melrose,
Gunther
Uhlmann, Jared Wunsch and Maciej Zworski for their interest and
support. In particular, the overall strategy reflects Melrose's vision
for solving non-linear PDE globally. The authors are also very grateful to an anonymous referee for many comments which improved the
exposition in the paper.

\section{Asymptotically de Sitter spaces: generalized static model}
\label{SecStaticDeSitter}

In this section we discuss solving semilinear wave equations on asymptotically de Sitter spaces from the `static perspective', i.e.\ in neighborhoods (in a blown-up space) of the backward light cone from a fixed point at future conformal infinity; see Figure~\ref{FigBlowUp}. The main ingredient is extending the linear theory from that of \cite{Va12} in various ways, which is the subject of \S\ref{SecStaticDeSitter-Fredholm}. In the following parts of this section we use this extension to solve semilinear equations, and to obtain their asymptotic behavior.

First, however, we recall some of the basics of b-analysis. As a general reference, we refer the reader to \cite{Me93}. Thus, let $M$ be an $n$-dimensional manifold with boundary $X$, and denote by $\Vb(M)$ the space of \emph{b-vector fields}, which consists of all vector fields on $M$ which are tangent to $X$. Elements of $\Vb(M)$ are sections of a natural vector bundle over $M$, the \emph{b-tangent bundle} $\Tb M$. Its dual, the \emph{b-cotangent bundle}, is denoted $\Tb^*M$. In local coordinates $(\tau,z)\in[0,\infty)\times\R^{n-1}$ near the boundary, the fibers of $\Tb M$ are spanned by $\tau\pa_\tau,\pa_{z_1},\ldots,\pa_{z_{n-1}}$, with $\tau\pa_\tau$ being a non-trivial b-vector field up to and including $\tau=0$ (even though it degenerates as an ordinary vector field), while the fibers of $\Tb^*M$ are spanned by $\frac{d\tau}{\tau},dz_1,\ldots,dz_{n-1}$. A \emph{b-metric} $g$ on $M$ is then simply a non-degenerate section of the second symmetric tensor power of $\Tb^*M$, i.e.\ of the form
\[
  g = g_{00}(\tau,z)\frac{d\tau^2}{\tau^2} + \sum_{i=1}^{n-1} g_{0i}(\tau,z)\Bigl(\frac{d\tau}{\tau}\otimes dz_i + dz_i\otimes\frac{d\tau}{\tau}\Bigr) + \sum_{i,j=1}^{n-1} g_{ij}(\tau,z)dz_i\otimes dz_j,
\]
$g_{ij}=g_{ji}$, with smooth coefficients $g_{k\ell}$. In terms of the coordinate $t=-\log\tau\in\R$, thus $\frac{d\tau}{\tau}=-dt$, the b-metric $g$ therefore approaches a stationary ($t$-independent in the local coordinate system) metric exponentially fast, as $\tau=e^{-t}$.

The \emph{b-conormal bundle} $\Nb^*Y$ of a boundary submanifold $Y\subset X$ of $M$ is the subbundle of $\Tb^*_Y M$ whose fiber over $p\in Y$ is the annihilator of vector fields on $M$ tangent to $Y$ and $X$. In local coordinates $(\tau,z',z'')$, where $Y$ is defined by $z'=0$ in $X$, these vector fields are smooth linear combinations of $\tau\pa_\tau$, $\pa_{z''_j}$, $z'_i\pa_{z'_j}$, $\tau\pa_{z'_k}$, whose span in $\Tb_p M$ is that of $\tau\pa_\tau$ and $\pa_{z''_j}$, and thus the fiber of the b-conormal bundle is spanned by the $dz'_j$, i.e.\ has the same dimension as the codimension of $Y$ in $X$ (and {\em not} that in $M$, corresponding to $\frac{d\tau}{\tau}$ not annihilating $\tau\pa_\tau$).

We define the \emph{b-cosphere bundle} $\Sb^*M$ to be the quotient of
$\Tb^*M\setminus o$ by the $\R^+$-action; here $o$ is the zero
section. Likewise, we define the spherical b-conormal bundle of a
boundary submanifold $Y\subset X$ as the quotient of $\Nb^*Y\setminus
o$ by the $\R^+$-action; it is a submanifold of $\Sb^*M$. A better way
to view $\Sb^*M$ is as the boundary at fiber infinity of the
fiber-radial compactification $\rcTb^*M$ of $\Tb^*M$, where the fibers
are replaced by their radial compactification, see \cite[\S2]{Va12}
and also \S\ref{SecMinkowski-Fredholm}. The b-cosphere bundle
$\Sb^*M\subset\rcTb^*M$ still contains the boundary of the compactification of the `old' boundary $\rcTb^*_XM$, see Figure~\ref{fig:Tb*M}.

\begin{figure}[!ht]
  \centering
  \includegraphics{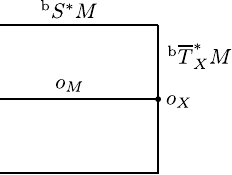}
  \caption{The radially compactified cotangent bundle $\rcTb^*M$ near $\rcTb^*_XM$; the cosphere bundle $\Sb^*M$, viewed as the boundary at fiber infinity of $\rcTb^*M$, is also shown, as well as the zero section $o_M\subset\rcTb^*M$ and the zero section over the boundary $o_X\subset\rcTb^*_XM$.}
  \label{fig:Tb*M}
\end{figure}

Next, the algebra $\Diffb(M)$ of \emph{b-differential operators} generated by $\Vb(M)$ consists of operators of the form
\[
  \cP = \sum_{|\alpha|+j\leq m} a_\alpha(\tau,z)(\tau D_\tau)^j D_z^\alpha,
\]
with $a_\alpha\in\CI(M)$, writing $D=\frac{1}{i}\pa$ as usual. (With
$t=-\log\tau$ as above, the coefficients of $\cP$ are thus constant up
to exponentially decaying remainders as $t\to\infty$.) Writing elements of $\Tb^*M$ as
\begin{equation}
\label{EqBDualCoords}
  \sigma\,\frac{d\tau}{\tau}+\sum_j\zeta_j\,dz_j,
\end{equation}
we have the principal symbol
\[
  \sigma_{\bl,m}(\cP) = \sum_{|\alpha|+j=m} a_\alpha(\tau,z) \sigma^j\zeta^\alpha,
\]
which is a homogeneous degree $m$ function in $\Tb^*M\setminus o$. Principal symbols are multiplicative, i.e.\ $\sigma_{\bl,m+m'}(\cP\circ\cP')=\sigma_{\bl,m}(\cP)\sigma_{\bl,m'}(\cP')$, and one has a connection between operator commutators and Poisson brackets, to wit
\[
  \sigma_{\bl,m+m'-1}(i[\cP,\cP']) = \sH_p p', \quad p=\sigma_{\bl,m}(\cP), p'=\sigma_{\bl,m'}(\cP'),
\]
where $\sH_p$ is the extension of the Hamilton vector field from $T^*M^\circ\setminus o$ to $\Tb^*M\setminus o$, which is thus a homogeneous degree $m-1$ vector field on $\Tb^*M\setminus o$ tangent to the boundary $\Tb^*_X M$. In local coordinates $(\tau,z)$ on $M$ near $X$, with b-dual coordinates $(\sigma,\zeta)$ as in \eqref{EqBDualCoords}, this has the form
\begin{equation}
\label{eq:b-Ham-vf}
  \sH_p=(\pa_\sigma p)(\tau\pa_\tau)-(\tau\pa_\tau p)\pa_\sigma+\sum_j \big((\pa_{\zeta_j}p)\pa_{z_j}-(\pa_{z_j}p)\pa_{\zeta_j}\big),
\end{equation}
see \cite[Equation~(3.20)]{Ba13}, where a somewhat different notation is used, given by \cite[Equation~(3.19)]{Ba13}. 

While elements of $\Diffb(M)$ commute to leading order in the symbolic sense, they do not commute in the sense of the order of decay of their coefficients. (This is in contrast to the scattering algebra, see \cite{Me94}.) The \emph{normal operator} captures the leading order part of $\cP\in\Diffb^m(M)$ in the latter sense, namely
\[
  N(\cP) = \sum_{j+|\alpha|\leq m} a_\alpha(0,z)(\tau D_\tau)^j D_z^\alpha.
\]
One can define $N(\cP)$ invariantly as an operator on the model space $M_I:=[0,\infty)_\tau\times X$ by fixing a boundary defining function of $M$, see \cite[\S3]{Va12}. Identifying a collar neighborhood of $X\subset M$ with a neighborhood of $\{0\}\times X$ in $M_I$, we then have $\cP-N(\cP)\in\tau\Diffb^m(M)$ (near $\pa M$). Since $N(\cP)$ is dilation-invariant (equivalently: translation-invariant in $t=-\log\tau$), it is naturally studied via the Mellin transform in $\tau$ (equivalently: Fourier transform in $-t$), which leads to the \emph{(Mellin transformed) normal operator family}
\[
  \wh N(\cP)(\sigma) \equiv \wh\cP(\sigma) = \sum_{j+|\alpha|\leq m} a_\alpha(0,z)\sigma^j D_z^\alpha,
\]
which is a holomorphic family of operators $\wh\cP(\sigma)\in\Diff^m(X)$.

Passing from $\Diffb(M)$ to the algebra of \emph{b-pseudodifferential
  operators} $\Psib(M)$ amounts to allowing symbols to be more general
functions than polynomials; apart from symbols being smooth functions
on $\Tb^*M$ rather than on $T^*M$ if $M$ was boundaryless, this is
entirely analogous to the way one passes from differential to
pseudodifferential operators, with the technical details being a bit
more involved. One can have a rather accurate picture of
b-pseudodifferential operators, however, by considering the following:
For $a\in\CI(\Tb^* M)$, we say $a\in S^m(\Tb^* M)$ if $a$ satisfies
  \[
    |\pa_w^\alpha\pa_\xi^\beta a(w,\xi)|\leq C_{\alpha\beta}\la\xi\ra^{m-|\beta|}\tn{ for all multiindices }\alpha,\beta
  \]
  in any coordinate chart, where $w$ are coordinates in the base and
  $\xi$ coordinates in the fiber; more precisely, in local coordinates
  $(\tau,z)$ near $X$, we take $\xi=(\sigma,\zeta)$ as above.
  We define the quantization $\Op(a)$ of $a$, acting on smooth functions $u$ supported in a coordinate chart, by
  \begin{align*}
    \Op(a)u(\tau,z)=(2\pi)^{-n}\int &e^{i(\tau-\tau')\wt\sigma+i(z-z')\zeta}\phi\left(\frac{\tau-\tau'}{\tau}\right) \\
	  &\hspace{1cm}\times a(\tau,z,\tau\wt\sigma,\zeta)u(\tau',z')\,d\tau'\,dz'\,d\wt\sigma\,d\zeta,
  \end{align*}
  where the $\tau'$-integral is over $[0,\infty)$, and
  $\phi\in\CI_c((-1/2,1/2))$ is identically $1$ near $0$. The
    cutoff $\phi$ ensures that these operators lie in the `small
    b-calculus' of Melrose, in particular that such quantizations act
    on weighted b-Sobolev spaces, defined below. For general $u$,
    define $\Op(a)u$ using a partition of unity. We write
    $\Op(a)\in\Psib^m(M)$; every element of $\Psib^m(M)$ is of the
    form $\Op(a)$ for some $a\in S^m(\Tb^*M)$ modulo the set $\Psib^{-\infty}(M)$ of
    smoothing operators. We say that $a$ is a \emph{symbol} of
    $\Op(a)$. The equivalence class of $a$ in
    $S^m(\Tb^*M)/S^{m-1}(\Tb^*M)$ is invariantly defined on $\Tb^*M$
    and is called the \emph{principal symbol} of $\Op(a)$.

If $A\in\Psib^{m_1}(M)$ and $B\in\Psib^{m_2}(M)$, then
$AB,BA\in\Psib^{m_1+m_2}(M)$, while
$[A,B]\in\Psib^{m_1+m_2-1}(M)$, and its principal symbol is
$\frac{1}{i}\sH_a b\equiv\frac{1}{i}\{a,b\}$, with $\sH_a$ as above.

Lastly, we recall the notion of \emph{b-Sobolev spaces}: Fixing a volume b-density $\nu$ on $M$, which locally is a positive multiple of $|\frac{d\tau}{\tau}\,dz|$, we define for, $s\in\N$,
\[
  \Hb^s(M) = \bigl\{ u\in L^2(M,\nu) \colon V_1\cdots V_j u\in L^2(M,\nu), V_i\in\Vb(M), 1\leq i\leq j\leq s \bigr\},
\]
which one can extend to $s\in\R$ by duality and
interpolation. \emph{Weighted b-Sobolev spaces} are denoted
$\Hb^{s,\alpha}(M)=\tau^\alpha\Hb^s(M)$, i.e.\ its elements are of the
form $\tau^\alpha u$ with $u\in\Hb^s(M)$. Any b-pseudodifferential
operator $\cP\in\Psib^m(M)$ defines a bounded linear map
$\cP\colon\Hb^{s,\alpha}(M)\to\Hb^{s-m,\alpha}(M)$ for all
$s,\alpha\in\R$. Correspondingly, there is a notion of wave front set
$\WFb^{s,\alpha}(u)\subset\Sb^*M$ for a distribution
$u\in\Hb^{-\infty,\alpha}(M)$, defined analogously to the wave front
set of distributions on $\R^n$ or closed manifolds. A point
$\varpi\in\Sb^*M$ is \emph{not} in $\WFb^{s,\alpha}(u)$ if and only if
there exists $\cP\in\Psib^0(M)$, elliptic at $\varpi$ (i.e.\ with
principal symbol non-vanishing on the ray corresponding to $\varpi$),
such that $\cP u\in\Hb^{s,\alpha}(M)$. Notice however that we
\emph{do} need to have a priori control on the weight $\alpha$ (we are
{\em assuming} $u\in\Hb^{-\infty,\alpha}(M)$), which again
reflects the lack of commutativity of $\Psib(M)$ even to leading order
in the sense of decay of coefficients at $\pa M$.

\subsection{The linear Fredholm framework}
\label{SecStaticDeSitter-Fredholm}

The goal of this section is to fully extend the results of \cite{Va12} on linear estimates for wave equations for b-metrics to non-dilation invariant settings, and to explicitly discuss Cauchy hypersurfaces since \cite{Va12} concentrated on complex absorption.  Namely, while the results of \cite{Va12} on linear estimates for wave equations for b-metrics are optimally stated when the metrics and thus the corresponding operators are dilation-invariant, i.e.\ when near $\tau=0$ the normal operator can be identified with the operator itself, see \cite[Lemma~3.1]{Va12}, the estimates for Sobolev derivatives are lossy for general b-metrics in \cite[Proposition~3.5]{Va12}, essentially because one should not treat the difference between the normal operator and the actual operator purely as a perturbation. Therefore, we first strengthen the linear results in \cite{Va12} in the non-dilation invariant setting by analyzing b-radial points which are saddle points of the Hamilton flow. This is similar to \cite[\S4]{Ba13}, where the analogous result was proved when the b-radial points are sources/sinks. This is then used to set up a Fredholm framework for the linear problem. If one is mainly interested in the dilation invariant case, one can use \cite[Lemma~3.1]{Va12} in place of Theorems~\ref{ThmPFredholm}-\ref{thm:b-main} below, either adding the boundary corresponding to $H_2$ below, or still using complex absorption as was done in \cite{Va12}.

So suppose $\cP\in\Psib^m(M)$, $M$ a manifold with boundary. (The dilation-invariant analysis of \cite[\S2]{Va12} applies to the Mellin transformed normal operator $\wh\cP(\sigma)$.) Let $p$ be the principal symbol of $\cP$, which we assume to be {\em real-valued}, and let $\sH_p$ be the Hamilton vector field of $p$. Let $\wt\rho$ denote a homogeneous degree $-1$ defining function of $\Sb^*M$. Then the rescaled Hamilton vector field
\[
  V=\wt\rho^{m-1}\sH_p
\]
is a $\CI$ vector field on $\rcTb^*M$ away from the 0-section, and it is tangent to all boundary faces. The characteristic set $\Sigma$ is the zero-set of the smooth function $\wt\rho^m p$ in $\Sb^*M$.  We refer to the flow of $V$ in $\Sigma\subset\Sb^*M$ as the Hamilton, or (null)bicharacteristic flow; its integral curves, the (null)bicharacteristics, are reparameterizations of those of the Hamilton vector field $\sH_p$, projected by the quotient map $\Tb^*M\setminus o\to\Sb^*M$.

\subsubsection{Generalized b-radial sets}

The standard propagation of singularities theorem in the characteristic set $\Sigma$ in the b-setting is that for $u\in\Hb^{-\infty,r}(M)$, within $\Sigma$, $\WFb^{s,r}(u)\setminus\WFb^{s-m+1,r}(\cP u)$ is a union of maximally extended integral curves (i.e.\ null-bicharacteristics) of $\cP$.  This is vacuous at points where $V$ vanishes (as a smooth vector field); these points are called {\em radial points}, since at such a point $\sH_p$ itself (on $\Tb^*M\setminus o$) is radial, i.e.\ is a multiple of the generator of the dilations of the fiber of the b-cotangent bundle. At a radial point $\alpha$, $V$ acts on the ideal $\cI$ of $\CI$ functions vanishing at $\alpha$, and thus on $T^*_\alpha\rcTb^*M$, which can be identified with $\cI/\cI^2$.  Since $V$ is tangent to both boundary hypersurfaces, given by $\tau=0$ and $\wt\rho=0$, $d\tau$ and $d\wt\rho$ are automatically eigenvectors of the linearization of $V$. We are interested in a generalization of the situation in which we have a smooth submanifold $L$ of $\Sb^*_X M$ consisting of radial points which is a source/sink for $V$ within $\Tb^*_X M$ but if it is a source, so in particular $d\wt\rho$ is in an unstable eigenspace, then $d\tau$ is in the (necessarily one-dimensional) stable eigenspace, and vice versa. Thus, $L$ is a saddle point of the Hamilton flow.

In view of
the bicharacteristic flow on Kerr-de Sitter space (which, unlike the
non-rotating de Sitter-Schwarzschild black holes, does not have this
precise radial point structure), it is
important to be slightly more general, as in
\cite[\S2.2]{Va12}. Thus, we assume that $dp$ does not vanish
where $p$ does, i.e.\ at $\Sigma$, and is linearly independent of
$d\tau$ at $\{\tau=0,p=0\}=\Sigma\cap\Sb^*_XM$, so $\Sigma$ is
a smooth submanifold of $\Sb^*M$ transversal to $\Sb^*_XM$. For
$L$ assume simply that $L=L_+\cup L_-$,
$L_\pm$ smooth disjoint submanifolds of $\Sb^*_X M$, given by
$\cL_\pm\cap\Sb^*_XM$ where $\cL_\pm$ are smooth disjoint
submanifolds 
of $\Sigma$ transversal to $\Sb^*_X M$ (these play the role of the two halves of the conormal bundles of
  event horizons), defined locally near $\Sb^*_X M$, with $\wt\rho^{m-1}\sH_p$ tangent to $\cL_\pm$, with a homogeneous degree zero
quadratic defining function $\rho_0$ (explained below) of $\cL$ within $\Sigma$ such
that
\begin{equation}\begin{split}\label{eq:tilde-rho-tau-pos}
\wt\rho^{m-2}\sH_p\wt\rho|_{L\pm}=&\mp\beta_0,\qquad
-\wt\rho^{m-1}\tau^{-1}\sH_p\tau|_{L_\pm}=\mp\wt\beta\beta_0,\\
&\beta_0,\wt\beta\in\CI(L_\pm),\ \beta_0,\wt\beta>0,
\end{split}\end{equation}
and, with $\beta_1>0$,
\begin{equation}\label{eq:rho_0-pos}
\mp\wt\rho^{m-1}\sH_p\rho_0-\beta_1\rho_0
\end{equation}
is $\geq 0$ modulo cubic vanishing terms at $L_\pm$. Here the phrase
`quadratic defining function $\rho_0$' means that $\rho_0$ vanishes
  quadratically at $\cL$ (and vanishes only at $\cL$), with the
  vanishing non-degenerate, in the sense that the Hessian is positive
  definite, corresponding to $\rho_0$ being a sum of squares of linear
  defining functions whose differentials span the conormal bundle of
  $\cL$ within $\Sigma$.

Under these assumptions $L_-$ is a
source and $L_+$ is a sink within $\Sb^*_XM$ in the sense that nearby
bicharacteristics {\em within $\Sb^*_XM$} all tend to $L_\pm$ as the parameter along them goes
to $\pm\infty$, but at $L_-$ there is also a stable, and at $L_+$ an
unstable, manifold, namely $\cL_-$, resp.\
$\cL_+$. Indeed, bicharacteristics in $\cL_\pm$ remain there by the
  tangency of $\wt\rho^{m-1}\sH_p$ to $\cL_\pm$; further $\tau\to
  0$ along them as the parameter goes to $\mp\infty$ by
  \eqref{eq:tilde-rho-tau-pos}, at least sufficiently close to
  $\tau=0$, since $L_\pm$ are defined in $\cL_\pm$ by $\tau=0$.

In order to simplify the statements, we assume that
$$
\wt\beta\ \text{is constant on}\ L_\pm;\ \wt\beta=\beta>0;
$$
we refer the reader to \cite[Equation~(2.5)-(2.6)]{Va12}, and the
discussion throughout that paper, where a general $\wt\beta$ is
allowed, at the cost of either $\sup\wt\beta$ or $\inf\wt\beta$
playing a role in various statements depending on signs.
Finally, we assume that $\cP-\cP^*\in\Psib^{m-2}(M)$ for
convenience (with respect to some b-metric), as this is the case for
the Klein-Gordon equation.\footnote{The natural assumption is that the principal symbol of
$\frac{1}{2i}(\cP-\cP^*)\in\Psib^{m-1}(M)$ at $L_\pm$ is
$$
\pm\wh\beta\beta_0\wt\rho^{-m+1},\qquad\wh\beta\in\CI(L_\pm).
$$
If $\wh\beta$ vanishes, Proposition~\ref{prop:b-saddle} is valid
without a change; otherwise it shifts the threshold quantity
$s-(m-1)/2-\beta r$ below in Proposition~\ref{prop:b-saddle} to
$s-(m-1)/2-\beta r+\wh\beta$ if $\wh\beta$ is constant, with
modifications as in \cite[Proof of Propositions~2.3-2.4]{Va12} otherwise.}

\begin{prop}\label{prop:b-saddle}
Suppose $\cP$ is as above.

If  $s\geq s'$, $s'-(m-1)/2>\beta r$, and if $u\in\Hb^{-\infty,r}(M)$
then $L_\pm$ (and thus a neighborhood of $L_\pm$) is disjoint from
$\WFb^{s,r}(u)$ provided $L_\pm\cap\WFb^{s-m+1,r}(\cP u)=\emptyset$, $L_\pm\cap\WFb^{s',r}(u)=\emptyset$,
and in a neighborhood of $L_\pm$, $\cL_\pm\cap\{\tau>0\}$
are disjoint from $\WFb^{s,r}(u)$.

On the other hand, if $s-(m-1)/2<\beta r$, and if $u\in\Hb^{-\infty,r}(M)$
then $L_\pm$ (and thus a neighborhood of $L_\pm$) is disjoint from
$\WFb^{s,r}(u)$ provided $L_\pm\cap\WFb^{s-m+1,r}(\cP u)=\emptyset$
and a punctured neighborhood of $L_\pm$, with $L_\pm$ removed, in
$\Sigma\cap\Sb^*_XM$ is disjoint from $\WFb^{s,r}(u)$.
\end{prop}

\begin{rmk}\label{rmk:beta-comp}
The decay order $r$ plays the role of $-\im\sigma$ in \cite{Va12} in
view of the Mellin transform in the dilation invariant setting
identifying weighted b-Sobolev spaces with weight $r$ with
semiclassical Sobolev spaces on the boundary on the line
$\im\sigma=-r$, see \cite[Equation~(3.8)-(3.9)]{Va12}. Thus,
the numerology in this proposition is a direct translation of
that in \cite[Propositions~2.3-2.4]{Va12}.
\end{rmk}

\begin{proof}
We remark first that $\wt\rho^{m-1}\sH_p\rho_0$ vanishes
quadratically on $\cL_\pm$ since $\wt\rho^{m-1}\sH_p$ is tangent to
$\cL_\pm$ and $\rho_0$ itself vanishes there
quadratically. Further, this quadratic expression is positive definite
near $\tau=0$ since it is such at $\tau=0$. Correspondingly, we can strengthen \eqref{eq:rho_0-pos}
to
\begin{equation}\label{eq:rho_0-ext-pos}
\mp\wt\rho^{m-1}\sH_p\rho_0-\frac{\beta_1}{2}\rho_0
\end{equation}
being non-negative modulo cubic terms vanishing at $\cL_\pm$ in a
neighborhood of $\tau=0$.

Notice next that, using \eqref{eq:rho_0-ext-pos} in the first case and
\eqref{eq:tilde-rho-tau-pos} in the second, and that $L_\pm$ is
defined in $\Sigma$ by $\tau=0$, $\rho_0=0$,
there exist $\delta_0>0$ and $\delta_1>0$ such that
$$
\alpha\in\Sigma,\ \rho_0(\alpha)<\delta_0,\ \tau(\alpha)<\delta_1,\ \rho_0(\alpha)\neq 0\Rightarrow (\mp\wt\rho^{m-1}\sH_p\rho_0)(\alpha)>0
$$ 
and
$$
\alpha\in\Sigma,\ \rho_0(\alpha)<\delta_0,\ \tau(\alpha)<\delta_1\Rightarrow (\pm\wt\rho^{m-1}\tau^{-1}\sH_p\tau)(\alpha)>0.
$$
Similarly to \cite[Proof of Propositions~2.3-2.4]{Va12}, which is not
in the b-setting, and \cite[Proof of Proposition~4.4]{Ba13}, which is
but concerns only sources/sinks (corresponding to Minkowski type spaces),
we consider commutants
$$
C\in\tau^{-r}\Psib^{s-(m-1)/2}(M)=\Psib^{s-(m-1)/2,-r}(M)
$$
with
principal symbol
$$
c=\phi(\rho_0)\phi_0(p_0)\phi_1(\tau)\wt\rho^{-s+(m-1)/2}\tau^{-r},\ p_0=\wt\rho^m p,
$$
where $\phi_0\in\CI_c(\RR)$ is identically $1$ near $0$,
$\phi\in\CI_c(\RR)$ is identically $1$ near $0$ with $\phi'\leq 0$ in $[0,\infty)$
and $\phi$ supported in $(-\delta_0,\delta_0)$,
while $\phi_1\in\CI_c(\RR)$ is identically $1$ near $0$ with
$\phi_1'\leq 0$ in $[0,\infty)$ and $\phi_1$ supported
in $(-\delta_1,\delta_1)$, so that
$$
\alpha\in\supp d(\phi\circ\rho_0)\cap\supp(\phi_1\circ\tau)\cap\Sigma\Rightarrow\mp(\wt\rho^{m-1}\sH_p\rho_0)(\alpha)>0,
$$
and
$$
\pm\wt\rho^{m-1}\tau^{-1}\sH_p\tau
$$
remains positive on
$\supp(\phi_1\circ\tau)\cap\supp(\phi\circ\rho_0)$.

The main contribution then comes from the
weights, which give
$$
\wt\rho^{m-1}\sH_p(\wt\rho^{-s+(m-1)/2}\tau^{-r})=\mp(-s+(m-1)/2+\beta r)\beta_0 \wt\rho^{-s+(m-1)/2}\tau^{-r},
$$
where the sign of the factor in parentheses on the right hand side
being negative, resp.\ positive, gives the first, resp.\ the second,
case of the statement of the proposition. Further, the sign of the
term in which $\phi_1(\tau)$, resp.\ $\phi(\rho_0)$, gets
differentiated, yielding $\pm\tau\wt\beta\beta_0\phi_1'(\tau)$, resp.\
$\phi'(\rho_0)\wt\rho^{m-1}\sH_p\rho_0$, is, when $s-(m-1)/2-\beta
r>0$,
the opposite, resp.\ the same, of these terms, while when
$s-(m-1)/2-\beta r<0$, it is the same, resp.\ the opposite, of these
terms. Correspondingly,
\begin{equation*}\begin{split}
\sigma_{2s}(i[\cP,C^*C])=\mp
2\Big(-\beta_0&\Big(s-\frac{m-1}{2}-\beta r\Big)\phi\phi_0\phi_1-\beta_0\wt\beta\tau \phi\phi_0\phi_1'\\
&\mp(\wt\rho^{m-1}\sH_p\rho_0)\phi'\phi_0\phi_1+m\beta_0 p_0\phi\phi'_0\phi_1\Big)\phi\phi_0\phi_1\wt\rho^{-2s}\tau^{-2r}.
\end{split}\end{equation*}
We can regularize using using $S_\ep\in \Psib^{-\delta}(M)$ for
$\ep>0$, uniformly bounded in $\Psib^0(M)$, converging to $\Id$ in
$\Psib^{\delta'}(M)$ for $\delta'>0$, with principal symbol
$(1+\ep\wt\rho^{-1})^{-\delta}$, as in \cite[Proof of
Propositions~2.3-2.4]{Va12}, where the only difference was that
the calculation was on $X=\pa M$, and thus the pseudodifferential
operators were standard ones, rather than b-pseudodifferential
operators. The a priori regularity assumption on $\WFb^{s',r}(u)$
arises as the regularizer has the opposite sign as compared to the
contribution of the weights, thus the amount of regularization one can
do is limited. The positive commutator argument then proceeds
completely analogously to \cite[Proof of
Propositions~2.3-2.4]{Va12}, except
that, as in \cite{Va12}, one has to assume a priori bounds on the term with the sign
opposite to that of $s-(m-1)/2-\beta r$, of which there is exactly one for
either sign (unlike in \cite{Va12}, in which only $s-(m-1)/2+\beta\im\sigma<0$ has
such a term), thus on
$\Sigma\cap\supp(\phi_1'\circ\tau)\cap\supp(\phi\circ\rho_0)$ when
$s-(m-1)/2-\beta r>0$ and on
$\Sigma\cap\supp(\phi_1\circ\tau)\cap\supp(\phi'\circ\rho_0)$ when
$s-(m-1)/2-\beta r<0$.

Using the openness of the complement of the wave front set we can finally choose
$\phi$ and $\phi_1$ (satisfying the support conditions, among others)
so that the a priori assumptions are satisfied, choosing $\phi_1$
first and then shrinking the support of $\phi$ in the first case, with
the choice being made in the opposite order in the second case, completing the proof
of the proposition.
\end{proof}

\subsubsection{Complex absorption}

In order to have good Fredholm properties
we either need a complete Hamilton flow, or need to `stop it' in a
manner that gives suitable estimates; one may want to do the latter to
avoid global assumptions on the flow on the ambient space. The microlocally best behaved
version is given by complex absorption; it is microlocal, works
easily with Sobolev spaces of arbitrary order, and makes the operator
elliptic in the absorbing region, giving rise to very convenient
analysis. The main downside of complex absorption is
that it does not automatically give forward mapping properties for the
support of solutions in wave equation-like settings, even though at
the level of singularities, it does have the desired forward
property. It was used extensively
in \cite{Va12} -- in the dilation invariant setting, the
bicharacteristics on $X\times(0,\infty)_\tau$ are controlled (by the
invariance) as $\tau\to\infty$ as well as when $\tau\to 0$, and thus one need not use
complex absorption there, instead decay as $\tau\to\infty$
(corresponding to growth as $\tau\to 0$ on these dilation invariant
spaces) gives the desired forward property; complex absorption was
only used to cut off the flow within $X$. Here we want to
localize in $\tau$ as well, and while complex absorption can achieve
this, it loses the forward {\em support} character of the problem. Thus, complex absorption will not be of use for us when solving semilinear forward problems later on; however, as it is
conceptually much cleaner, we discuss Fredholm properties using it first before turning to
adding artificial (spacelike) boundary hypersurfaces in the next section, which allow for the solution of forward problems but require additional technicalities.

Thus, we now consider $\cP-i\cQ\in\Psib^m(M)$, $\cQ\in\Psib^m(M)$, with real
principal symbol $q$, being the complex absorption
similarly to \cite[\S\S2.2 and 2.8]{Va12}; we assume that
$\WFb'(\cQ)\cap L=\emptyset$. Here the semiclassical
version, discussed in \cite{Va12} with further references there, is a close parallel to our b-setting; it is essentially equivalent to
the b-setting in the special case that $\cP$, $\cQ$ are
dilation-invariant, for then the Mellin transform gives rise exactly
to the semiclassical problem as the Mellin-dual parameter goes to infinity. Thus, we assume that the characteristic
set $\Sigma$ of $\cP$ has the form
$$
\Sigma=\Sigma_+\cup\Sigma_-,
$$
with each of $\Sigma_\pm$ being a union of connected components, and
$$
\mp q\geq 0\ \text{near}\ \Sigma_\pm.
$$
Recall from \cite[\S2.5]{Va12}, which in turn is a simple modification
of the semiclassical results of Nonnenmacher and Zworski \cite{Nonnenmacher-Zworski:Quantum},
and Datchev and Vasy \cite{Datchev-Vasy:Gluing-prop}, that under these sign conditions on $q$, estimates can be propagated in the backward direction
along the Hamilton flow on $\Sigma_+$ and in the forward direction for
$\Sigma_-$, or, phrased as a wave front set statement (the property of
being singular propagates in the opposite direction as the property of
being regular!),
$\WF^s(u)$ is invariant in
$(\Sigma_+\setminus\Sb^*_XM)\setminus\WF^{s-m+1}((\cP-i\cQ)u)$ under the forward Hamilton
flow, and in $(\Sigma_-\setminus\Sb^*_XM)\setminus\WF^{s-m+1}((\cP-i\cQ)u)$ under the
backward flow. (That is, the invariance is away
  from the boundary $X$; we address the behavior at the boundary in
  the rest of the paragraph.) Since this is a principal symbol argument, given in
\cite[\S2.5]{Va12} and
\cite[Lemma~5.1]{Datchev-Vasy:Gluing-prop}, its extension to the
b-setting only requires minimal changes. Namely, assuming one is away
from radial points as one may (since at these the statement is vacuous), one constructs the principal
symbol $c$ of the commutant on $\Tb^*M\setminus o$ as a $\CI$ function $c_0$
on $\Sb^*M$ with derivative of a fixed sign along the Hamilton flow in the
region where one wants to obtain the estimate (exactly the same way as
for real principal type proofs) multiplied by weights in $\tau$ and
$\wt\rho$, making the Hamilton derivative of $c_0$ large relative
to $c_0$ to control the error terms from the weights, and computes $\la u,-i[C^*C,\wt\cP]u\ra$, where $\wt\cP$ is the symmetric part
of $\cP-i\cQ$ (so has principal symbol $p$) and $\wt\cQ$ is the
antisymmetric part. This gives
$$
-2\re\la u,iC^*C(\cP-i\cQ)u\ra-2\re\la u,C^*C\wt\cQ u\ra.
$$
The issue here is that the second term on the right hand side involves
$C^*C\wt\cQ$, which is one order higher that $[C^*C,\wt\cP]$, so
while it itself has a desirable sign, one needs to be concerned about
subprincipal terms.\footnote{In fact, as the principal symbol of
  $C^*C\wt\cQ$ is real, the real part of its subprincipal symbol is
  well-defined, and is the real part of $c^2q$ where $c$ and $q$ include the real parts of their subprincipal terms, and is all that matters for this argument, so one
  could proceed symbolically.} However,
one rewrites
$$
2\re\la u,C^*C\wt\cQ u\ra=2\re\la u,C^*\wt\cQ
Cu\ra+2\re\la u,C^*[C,\wt\cQ] u\ra.
$$
Now the first term is positive modulo a controllable
error by the sharp
G{\aa}rding inequality, or if one arranges that $q$ is the square of a
symbol. This controllability claim uses the
derivative of $c$, arising in the symbol of the commutator with
$\wt\cP$, to provide the control: since $\wt\cQ$ is positive modulo an operator one order
lower, and in the term involving this operator, the principal symbol
$c$ of $C$ is not differentiated, writing $c$ as $c_0$ times a weight,
where $c_0$ is homogeneous of degree zero, taking the derivative of $c_0$ large
relative to $c_0$, as is already used to control weights, etc., controls
this error term (modulo which we have positivity) as well. On the other hand, the second can be rewritten in terms of
$[C,[C,\wt\cQ]]$, $(C^*-C)[C,\wt\cQ]$, etc., which are all
controllable as they drop {\em two} orders relative to the product
$C^*C\wt\cQ$. This gives rise to the result, namely that for $u\in\Hb^{-\infty,r}$,
$\WFb^{s,r}(u)$ is invariant in
$\Sigma_+\setminus\WF^{s-m+1,r}((\cP-i\cQ)u)$ under the forward Hamilton
flow, and in $\Sigma_-\setminus\WF^{s-m+1,r}((\cP-i\cQ)u)$ under the
backward flow. 

In analogy with \cite[Definition~2.12]{Va12}, we say that $\cP-i\cQ$ is {\em non-trapping} if all
bicharacteristics in $\Sigma$ from any point in
$\Sigma\setminus(L_+\cup L_-)$ flow to $\Ell(q)\cup L_+\cup L_-$ in both the
forward and backward directions (i.e.\ either enter $\Ell(q)$ in finite time or tend
to $L_+\cup L_-$). Notice that as $\Sigma_\pm$ are closed under the
Hamilton flow,
bicharacteristics in $\cL_\pm\setminus(L_+\cup L_-)$ necessarily enter the
elliptic set of $\cQ$ in the forward (in $\Sigma_+$), resp.\ backward  (in $\Sigma_-$)
direction. Indeed, by the non-trapping hypothesis, these
  bicharacteristics have to reach the elliptic set of $\cQ$
  as they cannot tend to $L_+$, resp.\ $L_-$: for $\cL_+$ and $\cL_-$ are
  unstable, resp.\ stable manifolds, and these bicharacteristics
  cannot enter the boundary (which is preserved by the flow), so cannot lie in the stable, resp.\
  unstable, manifolds of $L_+\cup L_-$, which are within $\Sb^*_XM$. Similarly,
  bicharacteristics in $(\Sigma\cap\Sb^*_XM)\setminus(L_+\cup L_-)$
  necessarily reach the elliptic set of $\cQ$ in the backward (in
  $\Sigma_+$), resp.\ forward (in $\Sigma_-$) direction.
Then
for $s,r$ satisfying
$$  
s-(m-1)/2>\beta r
$$  
one has an estimate
\begin{equation}\label{eq:P-b-symb-est}
\|u\|_{\Hb^{s,r}}\leq C\|(\cP- i\cQ)u\|_{\Hb^{s-m+1,r}}+C\|u\|_{\Hb^{s',r}},
\end{equation}
provided one assumes $s'<s$,
$$ 
s'-(m-1)/2>\beta r,\ u\in \Hb^{s',r}.
$$ 
Indeed, this is a simple consequence of $u\in\Hb^{s',r}$, $(\cP-
i\cQ)u\in\Hb^{s-m+1,r}$ implying $u\in\Hb^{s,r}$ via the closed graph
theorem, see \cite[Proof of Theorem~26.1.7]{Ho83} and
\cite[\S4.3]{Vasy:Microlocal-AH}. This implication in turn holds as on the
elliptic set of $\cQ$ one has the stronger statement $u\in\Hb^{s+1,r}$ under these
conditions, and then using real-principal type propagation of regularity in the {\em backward}
direction on $\Sigma_+$ and the {\em forward} direction on
$\Sigma_-$, one can propagate the microlocal membership of
$\Hb^{s,r}$ (i.e.\ the absence of the corresponding wave front set) in the
backward, resp.\ forward, direction on $\Sigma_+$, resp.\
$\Sigma_-$. Since bicharacteristics in $\cL_\pm\setminus(L_+\cup L_-)$ necessarily enter the
elliptic set of $\cQ$ in the forward, resp.\ backward
direction, and
thus one has $\Hb^{s,r}$ membership along them by what we have shown,
Proposition~\ref{prop:b-saddle} extends this membership to $L_\pm$,
and hence to a neighborhood of these,
and by our non-trapping assumption every bicharacteristic enters
either this neighborhood of $L_\pm$ or the elliptic set of $\cQ$ in
finite time in the backward, resp.\ forward, direction, so by the real
principal type propagation of singularities we have the claimed
microlocal membership everywhere.

Reversing the direction in which one propagates estimates, one also has a similar estimate for the adjoint $\cP^*+i\cQ^*$, except now one needs to have
$$  
s-(m-1)/2<\beta r
$$
in order to propagate through the saddle points in the opposite
direction, i.e.\ from within $\Sb^*_XM$ to $\cL_\pm$.
Then for $s'<s$,
\begin{equation}\label{eq:P*-b-symb-est}
\|u\|_{\Hb^{s,r}}\leq C\|(\cP^*+i\cQ^*)u\|_{\Hb^{s-m+1,r}}+C\|u\|_{\Hb^{s',r}}.
\end{equation}

The issue with these estimates is that $\Hb^{s,r}$ does not include
compactly into the error term $\Hb^{s',r}$ on
the right hand side due to the lack of additional decay. Thus, these
estimates are insufficient to show Fredholm properties, which in fact
do not hold in general.

We thus
further assume that there are no poles of the inverse of the Mellin
conjugate $(\cP-i\cQ)\ftrans(\sigma)$ of the normal operator, $N(\cP-i\cQ)$,
on the line $\Im\sigma=-r$.  Here we refer to \cite[\S3.1]{Va12} for a brief
discussion of the normal operator and the Mellin transform; this cited
section also
contains more detailed references to \cite{Me93}. Then using the
Mellin transform, which is an isomorphism between weighted b-Sobolev
spaces and semiclassical Sobolev spaces (see Equations~(3.8)-(3.9) in
\cite{Va12}), and the
estimates for $(\cP-i\cQ)\ftrans(\sigma)$ (including the
high energy, i.e.\ semiclassical, estimates,\footnote{The high energy estimates are actually\label{footnote:b-to-normal}
  implied by b-principal symbol based estimates on the normal operator
  space, $M_\infty=X\times\overline{\RR^+}$, $X=\pa M$, on spaces
  $\tau^r\Hb^s(M_\infty)$ corresponding to $\im\sigma=-r$, but we do not explicitly discuss
  this here.} all of which is discussed in detail in
\cite[\S2]{Va12} --- the high energy assumptions of
\cite[\S2]{Va12} hold by our assumptions on the b-flow at $\Sb^*_XM$, and which imply that for all but a discrete set of
$r$ the aforementioned lines do not contain such poles), we obtain that on $\RR^+_\rho\times\pa M$
\begin{equation}\label{eq:N-P-est}
\|v\|_{\Hb^{s,r}}\leq C\|N(\cP-i\cQ)v\|_{\Hb^{s-m+1,r}}
\end{equation}
when
$$  
s-(m-1)/2>\beta r.
$$  
Again, we have an analogous estimate for $N(\cP^*+i\cQ^*)$:
\begin{equation}\label{eq:N-P*-est}
\|v\|_{\Hb^{s,r}}\leq C\|N(\cP^*+i\cQ^*)v\|_{\Hb^{s-m+1,r}},
\end{equation}
provided $-r$ is not the imaginary part of a pole of the inverse of $(\cP^*+i\cQ^*)\ftrans$, and provided
$$  
s-(m-1)/2<\beta r.
$$  
As
$(\cP^*+i\cQ^*)\ftrans(\sigma)=(\wh\cP-i\wh\cQ)^*(\ol\sigma)$,
see the discussion in \cite{Va12} preceding Equation~(3.25), the requirement on $-r$ is the same as $r$ not
being the imaginary part of a pole of the inverse of $\wh\cP-i\wh\cQ$.

We apply these results by first letting $\chi\in\CI_c(M)$ be identically $1$
near $\pa M$ supported in a collar neighborhood of $\pa M$, which we
identify with $(0,\ep)_\tau\times\pa M$ of the normal operator
space. Then, {\em assuming} $s'-(m-1)/2>\beta r$,
\begin{equation}\label{eq:P-normal-op-break-up}
\begin{split}
\|u\|_{\Hb^{s',r}} &\leq \|\chi u\|_{\Hb^{s',r}}+\|(1-\chi)u\|_{\Hb^{s',r}}  \\
  &\leq C\|N(\cP-i\cQ)\chi u\|_{\Hb^{s'-m+1,r}}+\|(1-\chi)u\|_{\Hb^{s',r}}.
\end{split}
\end{equation}
Now, if $K=\supp (1-\chi)$, then$$
\|(1-\chi)u\|_{\Hb^{s',r}}\leq
C\|u\|_{H^{s'}(K)}\leq C'\|u\|_{\Hb^{s',\wt r}}\leq C''\|u\|_{\Hb^{s'+1,\wt r}}
$$
for any $\wt r$. On the other hand, $N(\cP-i\cQ)-(\cP-i\cQ)\in\tau\Psib^m([0,\ep)\times\pa M)$, so
\begin{equation*}\begin{split}
N(\cP-i\cQ)\chi u&=(\cP-i\cQ)\chi u+(N(\cP-i\cQ)-(\cP-i\cQ))\chi u\\
&=\chi (\cP-i\cQ)u+[\cP-i\cQ,\chi]u+(N(\cP-i\cQ)-(\cP-i\cQ))\chi u
\end{split}\end{equation*}
plus the fact that $[\cP-i\cQ,\chi]$ is supported in $K$ and $\|\chi
(\cP-i\cQ)u\|_{\Hb^{s'-m+1,r}}\leq\|(\cP-i\cQ)u\|_{\Hb^{s'-m+1,r}}$
show that for all $\wt r$
\begin{equation}\label{eq:P-normal-op-reduce}
\|N(\cP-i\cQ)\chi u\|_{\Hb^{s'-m+1,r}}\leq \| (\cP-i\cQ)u\|_{\Hb^{s'-m+1, r}}+C\|u\|_{\Hb^{s'+1,\wt r}}+C\|u\|_{\Hb^{s'+1,r-1}}.
\end{equation}
Combining \eqref{eq:P-b-symb-est}, \eqref{eq:P-normal-op-break-up} and
\eqref{eq:P-normal-op-reduce} we deduce that (with new constants, and
taking $s'$ sufficiently small and $\wt r=r-1$)
\begin{equation}\label{eq:P-Fred-est}
\|u\|_{\Hb^{s,r}}\leq C\|(\cP-i\cQ)u\|_{\Hb^{s-m+1,r}}+C\|u\|_{\Hb^{s'+1,r-1}},
\end{equation}
where now the inclusion $\Hb^{s,r}\to \Hb^{s'+1,r-1}$ is
compact when we choose, as we may, $s'<s-1$, requiring, however,
$s'-(m-1)/2>\beta r$.
Recall that this argument required that $s,r,s'$ satisfied the
requirements preceding \eqref{eq:P-b-symb-est}, and
that $-r$ is
not the imaginary part of any pole of $(\cP-i\cQ)\ftrans$.

Analogous estimates hold for $(\cP-i\cQ)^*$ where now we write $\wt s$,
$\wt r$ and $\wt s'$ for the Sobolev orders for the eventual application:
\begin{equation}\label{eq:P*-Fred-est}
\|u\|_{\Hb^{\wt s,\wt r}}\leq
C\|(\cP-i\cQ)^*u\|_{\Hb^{\wt s-m+1,\wt r}}+C\|u\|_{\Hb^{\wt s'+1,\wt r-1}},
\end{equation}
provided $\wt s$, $\wt r$ in place of $s$ and $r$
satisfy the requirements stated before \eqref{eq:P*-b-symb-est}, and provided
$-\wt r$ is not the imaginary part of a pole of $(\cP^*+i\cQ^*)\ftrans$
(i.e.\ $\wt r$ of $\wh\cP-i\wh\cQ$). Note that we
{\em do not} have a stronger requirement for $\wt s'$, unlike for
$s'$ above, since upper bounds for $s$ imply those for $s'\leq s$.

Via a standard functional analytic argument, see \cite[Proof of
Theorem~26.1.7]{Ho83} and also \cite[\S2.6]{Va12} in the present context, we thus obtain Fredholm
properties of $\cP-i\cQ$, in particular solvability, modulo a (possible) finite
dimensional obstruction, in $\Hb^{s,r}$ if
\begin{equation}\label{eq:error-term-req-dS}
s-(m-1)/2-1>\beta r.
\end{equation}
Concretely, we take $\wt
s=m-1-s$, $\wt r=-r$, $s'<s-1$ sufficiently close to $s-1$ so that
$s'-(m-1)/2>\beta r$ (which is possible by
\eqref{eq:error-term-req-dS}).
Thus, $s-(m-1)/2>\beta r$ means $\wt
s-(m-1)/2=(m-1)/2-s<-\beta r=\beta \wt r$,
so the space on the left hand side of
\eqref{eq:P-Fred-est} is dual to that in the first term on the right
hand side of \eqref{eq:P*-Fred-est}, and the same for the equations
interchanged, and notice that the condition on the poles of the
inverse of the Mellin transformed normal operators is the same for
both $\cP-i\cQ$ and $\cP^*+i\cQ^*$: $-r$
is not the imaginary part of a pole of $(\cP-i\cQ)\ftrans$. Let
$$
\cY^{s,r}=\Hb^{s,r}(M),\
\cX^{s,r}=\{u\in\Hb^{s,r}(M):\ (\cP-i\cQ) u\in \Hb^{s-1,r}(M)\},
$$
and note that
  $\cY^{s,r}$, $\cX^{s,r}$ are complete, in the case of $\cX^{s,r}$
  with the natural norm being
  $\|u\|^2_{\cX^{s,r}}=\|u\|^2_{\Hb^{s,r}(M)}+\|(\cP-i\cQ)
  u\|^2_{\Hb^{s-1,r}(M)}$; see Remark~\ref{RmkCompleteSpace}.
Our discussion thus far
yields:

\begin{prop}\label{prop:absorb-Fredholm}
Suppose that $\cP$ is non-trapping.
Suppose $s,r\in\RR$, $s-(m-1)/2-1>\beta r$, $-r$ is not the imaginary
part of a pole of $(\cP-i\cQ)\ftrans$
where $\cP-i\cQ$ is a priori a map
$$
\cP-i\cQ:\Hb^{s,r}(M)\to \Hb^{s-2,r}(M).
$$
Then
$$
\cP-i\cQ:\cX^{s,r}\to\cY^{s-1,r}
$$
is Fredholm.
\end{prop}

\subsubsection{Initial value problems}
As already mentioned, the main issue with this argument using complex absorption that it
does not guarantee the forward nature (in terms of supports) of the solution for a wave-like
equation, although it does guarantee the correct microlocal
structure. So now we assume that $\cP\in\Diffb^2(M)$ and that there is a Lorentzian b-metric $g$ such that
\begin{equation}\label{eq:cP-almost-Box}
\cP-\Box_g\in\Diffb^1(M),\ \cP-\cP^*\in\Diffb^0(M).
\end{equation}
Then one can run a completely analogous argument using energy type estimates by restricting the
domain we consider to be a manifold with corners, where the new
boundary hypersurfaces are spacelike with respect to $g$, i.e.\ given by level sets of
timelike functions. Such a possibility was mentioned in \cite[Remark~2.6]{Va12},
though it was not described in detail as it was not needed there,
essentially because the existence/uniqueness argument for forward
solutions was given only for dilation invariant operators. The main
difference between using complex absorption and adding boundary
hypersurfaces is that the latter limit the Sobolev regularity one can use, with the
most natural choice coming from energy estimates. However, a
posteriori one can improve the result to better Sobolev spaces using
propagation of singularities type results.

So assume now that $U\subset M$ is open, and
we have two functions $\frakt_1$ and $\frakt_2$ in $\CI(M)$,
both of which, restricted to $U$, are
timelike (in particular have non-zero differential) near their
respective 0-level sets $H_1$ and $H_2$, and
$$
\Omega=\frakt_1^{-1}([0,\infty))\cap\frakt_2^{-1}([0,\infty))\subset U.
$$
Notice that the timelike assumption forces
$d\frakt_j$ to not lie in $N^*X=N^*\pa M$ (for its image in the
b-cosphere bundle would be zero), and thus if the $H_j$ intersect $X$,
they do so transversally. We assume that the $H_j$ intersect only away
from $X$, and that they do so transversally, i.e.\ the differentials
of $\frakt_j$ are independent at the intersection. Then $\Omega$
is a manifold with corners with boundary hypersurfaces $H_1$, $H_2$
and $X$ (all intersected with $\Omega$). We however keep thinking of
$\Omega$ as a domain in $M$.
The role of the elliptic set of $\cQ$ is now
played by $\Sb^*_{H_j}M$, $j=1,2$.
The {\em non-trapping} assumption becomes
that
\begin{enumerate}
\item
all
bicharacteristics in $\Sigma_\Omega=\Sigma\cap \Sb^*_{\Omega} M$ from any point in
$\Sigma_\Omega\cap(\Sigma_+\setminus L_+)$ flow (within $\Sigma_\Omega$) to $\Sb^*_{H_1}M\cup L_+$ in the
forward direction (i.e.\ either enter $\Sb^*_{H_1}M$ in finite time or tend
to $L_+$) and to $\Sb^*_{H_2} M\cup L_+$ in the backward direction,
\item
and from any point in
$\Sigma_\Omega\cap(\Sigma_-\setminus L_-)$ the bicharacteristics flow
to $\Sb^*_{H_2} M\cup L_-$ in the forward direction and to
$\Sb^*_{H_1}M\cup L_-$ in the
backward direction;
\end{enumerate}
see Figure~\ref{FigDS}. In particular, orienting the characteristic
set by letting $\Sigma_-$ be the {\em future oriented} and $\Sigma_+$
the {\em past oriented} part, $d\frakt_1$ is future oriented, while
$d\frakt_2$ is past oriented.

\begin{figure}[!ht]
  \centering
  \includegraphics{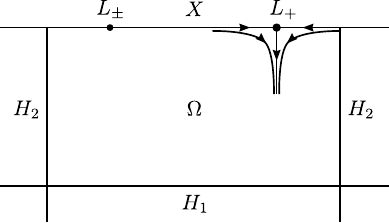}
  \caption{Setup for the discussion of the forward problem. Near the spacelike hypersurfaces $H_1$ and $H_2$, which are the replacement for the complex absorbing operator $\cQ$, we use standard (non-microlocal) energy estimates, and away from them, we use b-microlocal propagation results, including at the radial sets $L_\pm$. The bicharacteristic flow, in fact its projection to the base, is only indicated near $L_+$; near $L_-$, the directions of the flowlines are reversed.}
\label{FigDS}
\end{figure}

On a manifold with corners, such as $\Omega$, one can consider
supported and extendible distributions; see \cite[Appendix~B.2]{Ho83}
for the smooth boundary setting, with simple changes needed only for
the corners setting, which is discussed e.g.\ in \cite[\S3]{Vasy:Corners}. Here we consider $\Omega$ as a
domain in $M$, and thus its boundary face $X\cap\Omega$ is regarded as
having a different character from the $H_j\cap\Omega$, i.e.\ the
support/extendibility considerations do not arise at $X$ -- all
distributions are regarded as acting on a subspace of
$\CI$ functions on $\Omega$ vanishing at $X$ to infinite order, i.e.\ they are
automatically extendible distributions at $X$. On the other hand, at
$H_j$ we consider both extendible distributions, acting on
$\CI$ functions vanishing to infinite order at $H_j$, and supported
distributions, which act on all $\CI$ functions (as far as conditions
at $H_j$ are concerned). For example, the space of supported
distributions at $H_1$ extendible at $H_2$ (and at $X$, as we always
tacitly assume) is the dual space of the subspace of $\CI(\Omega)$
consisting of functions vanishing to infinite order at $H_2$ and $X$
(but not necessarily at $H_1$). An equivalent way of characterizing
this space of distributions is that they are restrictions of
elements of the dual of $\dCI(M)$ (consisting of $\CI$ functions on
$M$ vanishing to infinite order at $X$) with support in $\frakt_1\geq
0$ to $\CI$ functions on $\Omega$ which vanish to infinite order at $X$ and $H_2$,
thus in the terminology of \cite{Ho83}, restriction to
$\Omega\setminus (H_2\cup X)$.

The main interest is in spaces induced by the Sobolev
spaces
$\Hb^{s,r}(M)$. Notice that the Sobolev
  norm is of completely different nature at $X$ than at the $H_j$,
  namely the derivatives are based on complete, rather than
  incomplete, vector fields: $\Vb(M)$ is being restricted to $\Omega$,
  so one obtains vector fields tangent to $X$ but not to the $H_j$. As
  for supported and extendible distributions corresponding to $\Hb^{s,r}(M)$, we have, for instance,
$$
\Hb^{s,r}(M)^{\bullet,-},
$$
with the first
superscript on the right denoting whether supported ($\bullet$) or
extendible ($-$) distributions are discussed at
$H_1$, and the second the analogous property at $H_2$, consists of
restrictions of elements of $\Hb^{s,r}(M)$ with support in
$\frakt_1\geq 0$ to $\Omega\setminus (H_2\cup X)$.
Then elements of $\CI(\Omega)$ with the analogous vanishing
conditions, so in the example vanishing to infinite order at $H_1$ and
$X$, are dense in $\Hb^{s,r}(M)^{\bullet,-}$; further the dual of
$\Hb^{s,r}(M)^{\bullet,-}$ is $\Hb^{-s,-r}(M)^{-,\bullet}$ with
respect to the $L^2$ (sesquilinear) pairing.

First we work locally. For this purpose it is convenient to introduce
another timelike function $\wt\frakt_j$, not necessarily timelike, and consider 
$$
\region_{[t_0,t_1]}=\frakt_j^{-1}([t_0,\infty))\cap\wt\frakt_j^{-1}((-\infty,t_1]),
\
\region_{(t_0,t_1)}=\frakt_j^{-1}((t_0,\infty))\cap\wt\frakt_j^{-1}((-\infty,t_1)),
$$
and similarly on half-open, half-closed intervals. Thus, $\region_{[t_0,t_1]}$ becomes smaller as $t_0$ becomes larger or $t_1$ becomes smaller.

We then consider energy estimates on $\region_{[T_0,T_1]}$. In order to set up the following arguments, choose
$$
T_-<T_-'<T_0,\ T_1<T_+'<T_+,
$$
and assume that $\region_{[T_-,T_+]}$ is compact, $\region_{[T_0,T_1]}$ is non-empty, and $\frakt_j$ is timelike on $\region_{[T_-,T_+]}$. The energy estimates propagate estimates in the direction of either increasing or decreasing $\frakt_j$. With the extendible/supported character of distributions at $\wt\frakt_j=T_+$ being irrelevant for this matter in the case being considered and thus dropped from the notation (so ($-$) refers to extendibility at $\frakt_j=T_0$), consider
$$
\cP\colon\Hb^{s,r}(\region_{[T_0,T_+]})^-\to \Hb^{s-2,r}(\region_{[T_0,T_+]})^-,\qquad s,r\in\RR.
$$
The energy estimate, with backward propagation in $\frakt_j$, from $\wt\frakt_j^{-1}([T'_+,T_+])$, in this setting takes the form:

\begin{lemma}\label{lemma:energy-est-back}
  Let $r\in\RR$. There is $C>0$ such that for $u\in \Hb^{2,r}(\region_{[T_0,T_+]})^-$,
  \begin{equation}
  \label{eq:P-back-energy-est}
    \|u\|_{\Hb^{1,r}(\region_{[T_0,T_1]})^-}\leq C\Big(\|\cP u\|_{\Hb^{0,r}(\region_{[T_0,T_+]})^-}+\|u\|_{\Hb^{1,r}(\region_{[T_0,T_+]}\cap\wt\frakt_j^{-1}([T'_+,T_+]))^-}\Big)
  \end{equation}
  This also holds with $\cP$ replaced by $\cP^*$, acting on the same spaces.
\end{lemma}

\begin{rmk}\label{rmk:corner-energy-est}
  The lemma is also valid if one has several boundary hypersurfaces, i.e.\ if one replaces $\frakt_j^{-1}([t_0,\infty))$ by $\frakt_j^{-1}([t_{j,0},\infty))\cap\frakt_k^{-1}([t_{k,0},\infty))$ in the definition of $\Omega_{[t_0,t_1]}$, and/or $\wt\frakt_j^{-1}((-\infty,t_1])$ by $\wt\frakt_j^{-1}((-\infty,t_{j,1}])\cap \wt\frakt_k^{-1}((-\infty,t_{k,1}])$, i.e.\ regarding $\frakt_j$ and/or $\wt\frakt_j$ as vector valued, and propagating backwards in $\frakt_{j_0}$ for some fixed $j_0$, under the additional hypothesis that $\frakt_{j_0}$ is timelike in $\Omega_{[t_0,t_1]}$, and all $\frakt_j$, $j\neq j_0$, are timelike near their respective zero sets, with the same timelike character at $\frakt_{j_0}$. (One can also have more than two such functions.) To see this, replace $\chi(\frakt_j)$ by $\chi_{j_0}(\frakt_{j_0})\chi_k(\frakt_k)$, and analogously with $\wt\chi$ in the definition of $V$ in \eqref{eq:b-energy-est-V}, where $\chi_k$ is the characteristic function of $[t_{k,0},\infty)$, while letting $W=G(\bdiff\frakt_{j_0},\cdot)$. Then $\chi'\wt\chi\tau^\alpha A^\sharp$ is replaced by $\chi'_j\chi_k\wt\chi_j\wt\chi_k\tau^\alpha A^\sharp+\chi_j\chi'_k\wt\chi_j\wt\chi_k\tau^\alpha A^\sharp$, etc., and our additional hypothesis guarantees that the matrix $A^\sharp$ is indeed positive definite: The contribution from differentiating $\chi_{j_0}$ is positive definite by the timelike nature of $d\frakt_{j_0}$, while the contribution from differentiating $\chi_j$, $j\neq j_0$, giving $\delta$-distributions at the hypersurfaces $\frakt_j^{-1}(t_{j,0})$, is positive definite by the second part of the above additional hypothesis and can therefore be dropped as in the proof of Lemma~\ref{lemma:energy-est-back} below. Thus $\chi_{j_0}'$ can still be used to dominate $\chi_{j_0}$; and the terms in which $\wt\chi_j$ is differentiated have support where $\wt\frakt_j$ is in $(T_{+,j}',T_{+,j})$, so the control region on the right hand side of \eqref{eq:P-back-energy-est} is the union of these sets.

In our application this situation arises as we need the estimates on $\frakt_1^{-1}([T_0,T_1])\cap\frakt_2^{-1}([0,\infty))$ and $\frakt_1^{-1}([0,\infty))\cap\frakt_2^{-1}([T_0,T_1])$, with $T_0=0$, $T_1>0$ small. For instance, in the latter case $\frakt_2$ plays the role of $\frakt_j$ above, while $-\frakt_1$ and $\frakt_2$ play the role of $\wt\frakt_j$ and $\wt\frakt_k$; see Figure~\ref{FigEnergyCorner}.
\end{rmk}

\begin{figure}[!ht]
  \centering
  \includegraphics{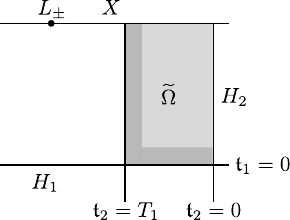}
  \caption{A domain $\wt\Omega=\frakt_2^{-1}([0,\infty))\cap\bigl((-\frakt_1)^{-1}((-\infty,0])\cap\frakt_2^{-1}((-\infty,T_1])\bigr)$ on which we will apply the energy estimate~\eqref{eq:P-back-energy-est}. The a priori control region is indicated in dark gray.}
  \label{FigEnergyCorner}
\end{figure}

\begin{proof}[Proof of Lemma~\ref{lemma:energy-est-back}.]
To see \eqref{eq:P-back-energy-est}, one proceeds as in \cite[\S3.3]{Va12} and considers
\begin{equation}
\label{eq:b-energy-est-V}
  V=-i\chi(\frakt_j)\wt\chi(\wt\frakt_j)\tau^\alpha W
\end{equation}
with $W=G(d\frakt_j,.)$ a timelike vector field and with $\chi,\wt\chi\in\CI(\RR)$, both non-negative, to be specified.  Then choosing a \emph{Riemannian} b-metric $\wt g$,
$$
-i(V^*\Box_g-\Box_g^*V)=\bdiff^*_{\wt g}C^\flat\,\bdiff,
$$
with the subscript on the adjoint on the right hand side denoting the metric with respect to which it is taken, $\bdiff:\CI(M)\to\CI(M;\Tb^*M)$ being the b-differential, and with
$$
C^\flat=\chi'\wt\chi\tau^\alpha A^\sharp+\chi\wt\chi'\tau^\alpha \wt A^\sharp+\chi\wt\chi \tau^\alpha R^\flat
$$
where $A^\sharp,\wt A^\sharp$ and $R^\flat$ are bundle endomorphisms of ${}^{\Cx}\Tb^*M$ and $A^\sharp,\wt A^\sharp$ are positive definite. Proceeding further, replacing $\Box_g$ by $\cP$, one has
\begin{equation}\begin{split}\label{eq:b-energy-commutator}
-i(V^*\cP-\cP^*V)&=\bdiff^*_{\wt g} C^\sharp\,\bdiff+(\wt E_1)^*_{\wt g}\tau^\alpha\chi\wt\chi \bdiff+\bdiff_{\wt g}^*\tau^\alpha\chi\wt\chi\wt E_2,\\
C^\sharp&=\chi'\wt\chi \tau^\alpha A^\sharp+\chi\wt\chi' \tau^\alpha \wt A^\sharp+\chi\wt\chi \tau^\alpha\wt R^\sharp
\end{split}\end{equation}
with $\wt E_j$ bundle maps from the trivial bundle over $M$ to ${}^{\Cx}\Tb^*M$, $A^\sharp,\wt A^\sharp$ as before, and $\wt R^\sharp$ a bundle endomorphism of ${}^{\Cx}\Tb^*M$, as follows by expanding
$$
-i(V^*(\cP-\Box_g)-(\cP-\Box_g)^*V)
$$
using that $\cP-\Box_g\in\Diffb^1(M)$. We regard the second term on
the right hand side of \eqref{eq:b-energy-commutator} as the one
requiring a priori control by
$\|u\|_{\Hb^{1,r}(\region_{[T_0,T_+]}\cap\wt\frakt_j^{-1}([T'_+,T_+]))^-}$;
we achieve this by making $\wt\chi$ supported in $(-\infty, T_+)$,
identically $1$ near $(-\infty,T_+']$, so $d\wt\chi$ is supported in
$(T_+',T_+)$. Now making $\chi'\geq 0$ large relative to $\chi$ on
$\supp(\chi\wt\chi)$, as in\footnote{In \cite[Equation~(3.27)]{Va12}
  the sign of $\chi'$ is opposite, as the estimate is propagated in
  the opposite direction.} \cite[Equation~(3.27)]{Va12}, allows one to
dominate all terms without derivatives of $\chi$. In order to obtain a
non-degenerate estimate up to $\frakt_j=T_0$, one cuts off $\chi$ at
$\frakt_j=T_0$ using the Heaviside function, so $\chi'$ gives a
(positive!) $\delta$-distribution there. Applying
\eqref{eq:b-energy-commutator} to $v$, pairing with $v$ and
integrating by parts, the $\delta$-distributions have the same sign as
$\chi' A^\sharp$ and can thus be dropped. Put differently, without the
sharp cutoff, one again computes the same pairing, but this time on
the domain $\Omega_{[T_0,T_+]}$, thus picking up boundary terms with
the correct sign in the integration by parts, so these terms can be
dropped. This proves the energy estimate \eqref{eq:P-back-energy-est} when one takes $\alpha=-2r$.
\end{proof}

Propagating in the forward direction,
from $\frakt_j^{-1}([T_-,T'_-])$, where now ($-$) denotes the character
of the space at $T_1$  (so ($-$) refers to extendibility at $\frakt_j=T_1$)
\begin{equation}
\|u\|_{\Hb^{1,r}(\region_{[T_0,T_1]})^-}\leq C\Big(\|\cP u\|_{\Hb^{0,r}(\region_{[T_-,T_1]})^-}+\|u\|_{\Hb^{1,r}(\region_{[T_-,T_1]}\cap\frakt_j^{-1}([T_-,T_-']))^-}\Big)
\end{equation}
In particular, for $u$ supported in
$\frakt_j\geq T_0$, the last estimate becomes,
with the first
superscript on the right denoting whether supported ($\bullet$) or
extendible ($-$) distributions are discussed at
$\frakt=T_0$, the second superscript the same at $\frakt=T_1$,
\begin{equation}\label{eq:P-forw-energy-est-bullet-orig}
\|u\|_{\Hb^{1,r}(\region_{[T_0,T_1]})^{\bullet,-}}\leq C\|\cP u\|_{\Hb^{0,r}(\region_{[T_0,T_1]})^{\bullet,-}},
\end{equation}
when
$$
\cP\colon\Hb^{s,r}(\region_{[T_0,T_1]})^{\bullet,-}\to \Hb^{s-2,r}(\region_{[T_0,T_1]})^{\bullet,-}
$$
and $u\in \Hb^{2,r}(\region_{[T_0,T_1]})^{\bullet,-}$. To
summarize, we state both this and \eqref{eq:P-back-energy-est} in terms
of these supported spaces:

\begin{cor}\label{cor:P-energy-est}
Let $r,\wt r\in\RR$.
For $u\in \Hb^{2,r}(\region_{[T_0,T_1]})^{\bullet,-}$, one has
\begin{equation}\label{eq:P-forw-energy-est-bullet}
\|u\|_{\Hb^{1,r}(\region_{[T_0,T_1]})^{\bullet,-}}\leq C\|\cP u\|_{\Hb^{0,r}(\region_{[T_0,T_1]})^{\bullet,-}},
\end{equation}
while for $v\in \Hb^{2,\wt r}(\region_{[T_0,T_1]})^{-,\bullet}$, the estimate
\begin{equation}\label{eq:P*-back-energy-est-bullet}
\|v\|_{\Hb^{1,\wt r}(\region_{[T_0,T_1]})^{-,\bullet}}\leq C\|\cP^* v\|_{\Hb^{0,\wt r}(\region_{[T_0,T_1]})^{-,\bullet}}
\end{equation}
holds.
\end{cor}

A duality argument, combined with propagation of singularities, thus gives:

\begin{lemma}\label{lemma:loc-solve-energy}
Let $s\geq 0$, $r\in\RR$. Then there is $C>0$ with the following property.

If $f\in\Hb^{s-1,r}(\region_{[T_0,T_1]})^{\bullet,-}$,
then there exists
$u\in\Hb^{s,r}(\region_{[T_0,T_1]})^{\bullet,-}$ such that $\cP
u=f$ and
$$
\|u\|_{\Hb^{s,r}(\region_{[T_0,T_1]})^{\bullet,-}}\leq C\|f\|_{\Hb^{s-1,r}(\region_{[T_0,T_1]})^{\bullet,-}}.
$$
\end{lemma}

\begin{rmk}\label{rmk:corner-solve}
As in Remark~\ref{rmk:corner-energy-est}, the lemma remains valid in more generality, namely if one replaces $\frakt_j^{-1}([t_0,\infty))$
by $\frakt_j^{-1}([t_{j,0},\infty))\cap\frakt_k^{-1}([t_{k,0},\infty))$,
and/or $\wt\frakt_j^{-1}((-\infty,t_1])$ by
$\wt\frakt_j^{-1}((-\infty,t_{j,1}])\cap
\wt\frakt_j^{-1}((-\infty,t_{k,1}])$ in the definition of
$\region_{[t_0,t_1]}$, provided that the $\frakt_j$ have linearly
independent differentials on their joint zero set, and similarly for
the $\wt\frakt_j$. The place where this linear independence is used (the 
energy estimate above does not need this) is for the continuous
Sobolev extension map, valid on manifolds with corners, see \cite[\S3]{Vasy:Corners}.
\end{rmk}

\begin{proof}
We work on the slightly bigger region $\region_{[T_-',T_+']}$, applying the energy estimates with $T_0$ replaced by $T_-'$, $T_1$ replaced by $T_+'$. First, by the supported property at $\frakt_j=T_0$, one can regard $f$ as an element of $\Hb^{s-1,r}(\region_{[T_-',T_1]})^{\bullet,-}$ with support in $\region_{[T_0,T_1]}$. Let
$$
\wt
f\in\Hb^{s-1,r}(\region_{[T_-',T_+']})^{\bullet,-}\subset
\Hb^{-1,r}(\region_{[T_-',T_+']})^{\bullet,-}
$$
be an extension
of $f$, so $\wt f$ is supported in $\region_{[T_0,T_+']}$, and
restricts to $f$; by the definition of spaces of extendible distributions as quotients of spaces of distributions on a larger space, see \cite[Appendix~B.2]{Ho83}, we can assume
\begin{equation}
\label{EqExtQuotientNorm}
  \|\wt f\|_{\Hb^{s-1,r}(\region_{[T_-',T_+']})^{\bullet,-}} \leq 2\|f\|_{\Hb^{s-1,r}(\region_{[T_-',T_1]})^{\bullet,-}}.
\end{equation}
By \eqref{eq:P-back-energy-est} applied with $\cP$
replaced by $\cP^*$, $\wt r=-r$,
\begin{equation}\label{eq:P*-back-energy-est-appl}
\|\phi\|_{\Hb^{1,\wt r}(\region_{[T_-',T_+']})^{-,\bullet}}\leq
C\|\cP^* \phi\|_{\Hb^{0,\wt r}(\region_{[T_-',T_+']})^{-,\bullet}},
\end{equation}
for $\phi\in \Hb^{2,\wt r}(\region_{[T_-',T_+']})^{-,\bullet}$.
Correspondingly, by the Hahn-Banach theorem, there exists
$$
\wt u\in (\Hb^{0,\wt
  r}(\region_{[T_-',T_+']})^{-,\bullet})^*=\Hb^{0,r}(\region_{[T_-',T_+']})^{\bullet,-}
$$
such that
$$
\la \cP\wt u,\phi\ra=\la \wt u,\cP^*\phi\ra=\la\wt
f,\phi\ra,\qquad \phi\in \Hb^{2,\wt
  r}(\region_{[T_-',T_+']})^{-,\bullet},
$$
and
\begin{equation}\label{eq:HB-est}
\|\wt u\|_{\Hb^{0,r}(\region_{[T_-',T_+']})^{\bullet,-}}\leq C
\|\wt f\|_{\Hb^{-1,r}(\region_{[T_-',T_+']})^{\bullet,-}}.
\end{equation}
One can regard $\wt
u$ as an element of
$\Hb^{0,r}(\region_{[T_-,T_+']})^{\bullet,-}$
with support in $\region_{[T_-',T_+']}$, with $\wt f$ similarly extended; then $\la\cP \wt
u,\phi\ra=\la\wt f,\phi\ra$ for $\phi\in
\dCI_c(\region_{(T_-,T_+')})$ (here the dot over $\CI$ refers to
infinite order vanishing at $X=\pa M$!), so $\cP\wt u=\wt f$ in
distributions. Since $\wt u$ vanishes on $\region_{(T_-,T_-')}$, and
$$
\wt
f\in \Hb^{s-1,r}(\region_{[T_-,T_+']})^{\bullet,-},
$$
propagation
of singularities applied on $\region_{(T_-,T_+')}$ (which has only
the boundary $\pa M=X$) gives that $\wt
u\in\Hbloc^{s,r}(\region_{(T_-,T_+')})$ (i.e.\ here we are ignoring
the two boundaries, $\frakt_j=T_-,T_+'$, not making a uniform
statement there, but we are not ignoring $\pa M=X$). In addition,
for $\chi,\wt\chi\in\CI_c (\region_{(T_-,T_+')})$,
$\wt\chi\equiv 1$ on $\supp\chi$, we have the estimate
\begin{equation}\label{eq:prop-sing-est}
\|\chi \wt u\|_{\Hb^{s,r}(\region_{[T_-,T_+']})}\leq
C\Big(\|\wt\chi\cP\wt
u\|_{\Hb^{s-1,r}(\region_{[T_-,T_+']})}+\|\wt\chi\wt u\|_{\Hb^{0,r}(\region_{[T_-,T_+']})}\Big).
\end{equation}
 In view
of the support property of $\wt u$, this gives that restricting to
$\region_{(T_-,T_1]}$, we obtain an element of
$\Hb^{s,r}(\region_{(T_-,T_1]})^-$, with support in $\region_{[T_0,T_1]}$,
i.e.\ an element of
$\Hb^{s,r}(\region_{[T_0,T_1]})^{\bullet,-}$.
The desired estimate follows from \eqref{eq:HB-est}, controlling the
second term of the right hand side of \eqref{eq:prop-sing-est}, and
\eqref{EqExtQuotientNorm} as well as using $\cP\wt u=\wt f$.
\end{proof}

At this point, $u$ given by Lemma~\ref{lemma:loc-solve-energy} is not
necessarily unique. However:

\begin{lemma}\label{lemma:unique}
Let $s,r\in\RR$. If
$u\in\Hb^{s,r}(\region_{[T_0,T_1]})^{\bullet,-}$ is such that $\cP u=0$, then $u=0$.
\end{lemma}

\begin{proof}
Propagation of singularities, as in the proof of
Lemma~\ref{lemma:loc-solve-energy}, regarding $u$ as a distribution on
$(T_-,T_1)$ with support in $[T_0,T_1)$ gives that $u\in
\Hbloc^{\infty,r}(\region_{(T_-,T_1)})$. Taking $T_0<T_1'<T_1$,
letting $u'=u|_{[T_0,T_1']}$,
\eqref{eq:P-forw-energy-est-bullet} shows that $u'=0$. Since $T'_1$ is
arbitrary, this shows $u=0$.
\end{proof}

\begin{cor}\label{cor:loc-WP-energy}
Let $s\geq 0$, $r\in\RR$. Then there is $C>0$ with the following property.

If $f\in\Hb^{s-1,r}(\region_{[T_0,T_1]})^{\bullet,-}$,
then there exists a unique
$u\in\Hb^{s,r}(\region_{[T_0,T_1]})^{\bullet,-}$ such that $\cP
u=f$.

Further, this unique $u$ satisfies
$$
\|u\|_{\Hb^{s,r}(\region_{[T_0,T_1]})^{\bullet,-}}\leq C\|f\|_{\Hb^{s-1,r}(\region_{[T_0,T_1]})^{\bullet,-}}.
$$
\end{cor}

\begin{proof}
Existence is Lemma~\ref{lemma:loc-solve-energy}, uniqueness is
linearity plus Lemma~\ref{lemma:unique}, which together with the
estimate in Lemma~\ref{lemma:loc-solve-energy} prove the corollary.
\end{proof}

\begin{cor}\label{cor:energy-est-strong}
Let $s\geq 0$, $r,\wt r\in\RR$.

For $u\in \Hb^{s,r}(\region_{[T_0,T_1]})^{\bullet,-}$ with $\cP u\in \Hb^{s-1,r}(\region_{[T_0,T_1]})^{\bullet,-}$,
\begin{equation}\label{eq:P-forw-energy-est-bullet-strong}
\|u\|_{\Hb^{s,r}(\region_{[T_0,T_1]})^{\bullet,-}}\leq C\|\cP u\|_{\Hb^{s-1,r}(\region_{[T_0,T_1]})^{\bullet,-}},
\end{equation}
while for $v\in \Hb^{s,\wt
  r}(\region_{[T_0,T_1]})^{-,\bullet}$ with $\cP^*v\in \Hb^{s-1,\wt r}(\region_{[T_0,T_1]})^{-,\bullet}$,
\begin{equation}\label{eq:P*-back-energy-est-bullet-strong}
\|v\|_{\Hb^{s,\wt r}(\region_{[T_0,T_1]})^{-,\bullet}}\leq
C\|\cP^* v\|_{\Hb^{s-1,\wt r}(\region_{[T_0,T_1]})^{-,\bullet}}.
\end{equation}
\end{cor}

\begin{rmk}
Again, this estimate remains valid for vector valued $\frakt_j$ and
$\wt\frakt_j$, see Remarks~\ref{rmk:corner-energy-est} and
\ref{rmk:corner-solve}, under the linear independence condition of the latter.
\end{rmk}

\begin{proof}
It suffices to consider \eqref{eq:P-forw-energy-est-bullet-strong}. Let $f=\cP u\in \Hb^{s-1,r}(\region_{[T_0,T_1]})^{\bullet,-}$, and let $u'\in\Hb^{s,r}(\region_{[T_0,T_1]})^{\bullet,-}$ be given by Corollary~\ref{cor:loc-WP-energy}. In view of the uniqueness statement of Corollary~\ref{cor:loc-WP-energy}, $u=u'$. Then the estimate of Corollary~\ref{cor:loc-WP-energy} proves the claim.
\end{proof}

This yields the following propagation of singularities type result:

\begin{prop}\label{prop:prop-sing-bdy}
Let $s\geq 0$, $r\in\RR$.
If $u\in \Hb^{-\infty,-\infty}(\region_{[T_0,T_1]})^{\bullet,-}$ with $\cP u\in
\Hb^{s-1,r}(\region_{[T_0,T_1]})^{\bullet,-}$, then $u\in
\Hb^{s,r}(\region_{[T_0,T_1]})^{\bullet,-}$.

If instead $u\in \Hb^{-\infty,-\infty}(\region_{[T_0,T_1]})^{-,-}$ with $\cP u\in \Hb^{s-1,r}(\region_{[T_0,T_1]})^{-,-}$ and for some $\wt T_0>T_0$, $u\in\Hb^{s,r}(\region_{[T_0,T_1]}\setminus\region_{(\wt T_0,T_1]})^{-,-}$, then $u\in\Hb^{s,r}(\region_{[T_0,T_1]})^{-,-}$.
\end{prop}

\begin{rmk}
One can `mix and match' the two parts of the proposition in the setting of
Remark~\ref{rmk:corner-energy-est}, with say a supportedness condition
at $\wt\frakt_j$, and only an extendibility assumption at
$\wt\frakt_k$, but with $\Hb^{s,r}$ membership assumption on $u$ in
$\region_{[T_0,T_1]}\setminus\wt\frakt_k^{-1}((-\infty,\wt T_1))$, $\wt T_1<T_1$, with a completely analogous argument. For instance, in the
setting of Figure~\ref{FigEnergyCorner}, one gets the regularity under
supportedness assumptions at $H_1$, just extendibility at
$\frakt_2=T_1$, but a priori regularity for $\frakt_2\in(\wt T_1,T_1)$.
\end{rmk}

\begin{proof}
Let $u'\in \Hb^{s,r}(\region_{[T_0,T_1]})^{\bullet,-}$ be the unique solution in $\Hb^{s,r}(\region_{[T_0,T_1]})^{\bullet,-}$ of $\cP u'=f$ where $f=\cP u\in \Hb^{s-1,r}(\region_{[T_0,T_1]})^{\bullet,-}$; we obtain $u'$ by applying the existence part of Corollary~\ref{cor:loc-WP-energy}. Then $u,u'\in \Hb^{-\infty,-\infty}(\region_{[T_0,T_1]})^{\bullet,-}$ and $\cP(u-u')=0$. Applying Lemma~\ref{lemma:unique}, we conclude that $u=u'$, which completes the proof of the first part.

For the second part, let $\chi\in\CI(\RR)$ be supported in $(T_0,\infty)$, identically $1$ near $[\wt T_0,\infty)$,
and consider $u'=(\chi\circ\frakt_j) u\in\Hb^{1,r}(\region_{[T_0,T_1]})^{\bullet,-}$, with the support
property arising from the vanishing of $\chi$ near $T_0$. Then $\cP u'=[\cP,(\chi\circ\frakt_j)]u+(\chi\circ\frakt_j)\cP u$. Now the first term on the right hand side is in $\Hb^{s-1,r}(\region_{[T_0,T_1]})^{\bullet,-}$ as on the support of
$d\chi$, which is in $\region_{[T_0,T_1]}\setminus\region_{(\wt T_0,T_1]}$,
$u$ is in $\Hb^{s,r}$,
and the commutator is first order, while the second term is in the
desired space since $\cP u\in \Hb^{s-1,r}(\region_{[T_0,T_1]})^{-,-}$,
and as for $u$ itself, the cutoff improves the support property. Thus,
the first part of the lemma is applicable, giving that $\chi
u\in\Hb^{s,r}(\region_{[T_0,T_1]})^{\bullet,-}$. Since
$(1-\chi)u\in\Hb^{s,r}(\region_{[T_0,T_1]})^{-,-}$ by the a priori
assumption, the conclusion follows.
\end{proof}

We take $T_0=0$ and thus consider, for $s\geq 0$,
\begin{equation}\label{eq:cP-spaces}
\cP\colon\Hb^{s,r}(\Omega)^{\bullet,-}\to\Hb^{s-2,r}(\Omega)^{\bullet,-}
\end{equation}
and
\begin{equation}\label{eq:cP*-spaces}
\cP^*\colon\Hb^{s,r}(\Omega)^{-,\bullet}\to\Hb^{s-2,r}(\Omega)^{-,\bullet}.
\end{equation}
In combination with the real
principal type propagation results and Proposition~\ref{prop:b-saddle} this yields under the non-trapping
assumptions, much as in the complex absorbing case, that\footnote{In fact, the error term on the right hand
  side can be taken to be supported in a smaller region, since at
  $H_1$ in the first case and at $H_2$ in the second, there are no
  error terms due to the energy estimates
  \eqref{eq:P-forw-energy-est-bullet}, applied with $\cP^*$ in place of $\cP$
  in the second case.}
\begin{equation}\label{eq:P-b-loc-symb-est}
\|u\|_{\Hb^{s,r}(\Omega)^{\bullet,-}}\leq C\|\cP u\|_{\Hb^{s-1,r}(\Omega)^{\bullet,-}}+C\|u\|_{\Hb^{0,r}(\Omega)^{\bullet,-}},\ \beta r<-1/2,\ s>0,
\end{equation} 
and
\begin{equation}\label{eq:P*-b-loc-symb-est}
\|u\|_{\Hb^{s,\wt r}(\Omega)^{-,\bullet}}\leq C\|\cP^* u\|_{\Hb^{s-1,\wt r}(\Omega)^{-,\bullet}}+C\|u\|_{\Hb^{0,\wt r}(\Omega)^{-,\bullet}},\ \beta \wt r>s-1/2,\ s>0.
\end{equation} 

We could proceed as in the complex absorption case to make the space on the
left hand side include compactly into the `error
term' on the right hand using the normal operators. As this imposes some
constraints, cf.\ \eqref{eq:error-term-req-dS}, which together with the requirements of the
energy estimates, namely that the Sobolev order is $\geq 0$, mean that
we would get slightly too strong restrictions on $s$, see Remark~\ref{rmk:mod-Fredholm},
we
proceed instead with a direct energy estimate. We thus assume that $\Omega$
is sufficiently small so that there is a boundary defining function
$\tau$ of $M$ with $\frac{d\tau}{\tau}$ timelike on $\Omega$, of the
same timelike character as $\frakt_2$, opposite to $\frakt_1$. As explained in
\cite[\S7]{Va12}, in this case there is $C>0$ such that for
$\im\sigma>C$, $\wh P(\sigma)$ is necessarily invertible.

The energy estimate is:

\begin{lemma}
There exists $r_0<0$ such that for $r\leq r_0$, $-\wt r\leq r_0$,
there is $C>0$ such that for
$u\in
\Hb^{2,r}(\Omega)^{\bullet,-}$, $v\in \Hb^{2,\wt
  r}(\Omega)^{-,\bullet}$, one has
\begin{equation}\begin{split}\label{eq:P-b-energy-est}
&\|u\|_{\Hb^{1,r}(\Omega)^{\bullet,-}}\leq C\|\cP 
u\|_{\Hb^{0,r}(\Omega)^{\bullet,-}},\\
&\|v\|_{\Hb^{1,\wt r}(\Omega)^{-,\bullet}}\leq C\|\cP^* 
v\|_{\Hb^{0,\wt r}(\Omega)^{-,\bullet}}.
\end{split}\end{equation} 
\end{lemma}

\begin{proof}
We run the argument of Lemma~\ref{lemma:energy-est-back} globally on $\Omega$
using a timelike
vector field (e.g.\ starting with $W=G(\frac{d\tau}{\tau},.)$) that has, as a multiplier, a sufficiently large positive
power $\alpha=-2r$ of $\tau$,
i.e.\ replacing \eqref{eq:b-energy-est-V} by
$$
V=-i\tau^\alpha W.
$$
Then the term with $\tau^\alpha$ differentiated (which in
\eqref{eq:b-energy-commutator} is included in the $\wt R^\sharp$ term), and
thus possessing a factor of $\alpha$, is used to dominate
the other, `error', terms in \eqref{eq:b-energy-commutator},
completing the proof of the lemma as in
Lemma~\ref{lemma:energy-est-back}.
\end{proof}

This can be used as in Lemma~\ref{lemma:loc-solve-energy} to provide
solvability, and using the propagation of singularities, which in this
case includes the use of Proposition~\ref{prop:b-saddle}, noting that
$s-1/2>\beta r$ is automatically satisfied, improved regularity. In particular, we obtain the following analogues of
Corollaries~\ref{cor:loc-WP-energy}-\ref{cor:energy-est-strong}.

\begin{cor}\label{cor:global-WP-energy}
There is $r_0<0$ such that for $r\leq r_0$ and for $s\geq 0$ there is $C>0$ with the following property.

If $f\in\Hb^{s-1,r}(\Omega)^{\bullet,-}$,
then there exists a unique
$u\in\Hb^{s,r}(\Omega)^{\bullet,-}$ such that $\cP
u=f$.

Further, this unique $u$ satisfies
$$
\|u\|_{\Hb^{s,r}(\Omega)^{\bullet,-}}\leq C\|f\|_{\Hb^{s-1,r}(\Omega)^{\bullet,-}}.
$$
\end{cor}

\begin{cor}\label{cor:global-energy-est-strong}
There is $r_0<0$ such that if $r<r_0$, $-\wt r<r_0$ and $s\geq 0$
then there is $C>0$ such that the following holds.

For $u\in \Hb^{s,r}(\Omega)^{\bullet,-}$ with $\cP u\in
\Hb^{s-1,r}(\Omega)^{\bullet,-}$, one has
\begin{equation}\label{eq:global-P-forw-energy-est-bullet-strong}
\|u\|_{\Hb^{s,r}(\Omega)^{\bullet,-}}\leq C\|\cP u\|_{\Hb^{s-1,r}(\Omega)^{\bullet,-}},
\end{equation}
while for $v\in \Hb^{s,\wt
  r}(\Omega)^{-,\bullet}$ with $\cP^*v\in \Hb^{s-1,\wt
  r}(\Omega)^{-,\bullet}$, one has
\begin{equation}\label{eq:global-P*-back-energy-est-bullet-strong}
\|v\|_{\Hb^{s,\wt r}(\Omega)^{-,\bullet}}\leq
C\|\cP^* v\|_{\Hb^{s-1,\wt r}(\Omega)^{-,\bullet}}.
\end{equation}
\end{cor}

We restate Corollary~\ref{cor:global-WP-energy} as an invertibility statement.

\begin{thm}
\label{ThmPFredholm}
There is $r_0<0$ with the following property.
Suppose $s\geq 0$, $r\leq r_0$, and let
$$
\cY^{s,r}=\Hb^{s,r}(\Omega)^{\bullet,-},\
\cX^{s,r}=\{u\in\Hb^{s,r}(\Omega)^{\bullet,-}:\ \cP u\in \Hb^{s-1,r}(\Omega)^{\bullet,-}\},
$$
where $\cP$ is a priori a map
$\cP\colon\Hb^{s,r}(\Omega)^{\bullet,-}\to \Hb^{s-2,r}(\Omega)^{\bullet,-}.$
Then
$$
\cP\colon\cX^{s,r}\to\cY^{s-1,r}
$$
is a continuous, invertible map, with continuous inverse.
\end{thm}

\begin{rmk}
\label{RmkCompleteSpace}
  Note that $\cY^{s,r}$, $\cX^{s,r}$ are complete, in the case of $\cX^{s,r}$ with the natural norm being $\|u\|^2_{\cX^{s,r}}=\|u\|^2_{\Hb^{s,r}(\Omega)^{\bullet,-}}+\|\cP u\|^2_{\Hb^{s-1,r}(\Omega)^{\bullet,-}}$, as follows by the continuity of $\cP$ as a map $\Hb^{s,r}(\Omega)^{\bullet,-}\to \Hb^{s-2,r}(\Omega)^{\bullet,-}$ and the completeness of the b-Sobolev spaces $\Hb^{s,r}(\Omega)^{\bullet,-}$.
\end{rmk}

\begin{rmk}\label{rmk:mod-Fredholm}
Using normal operators as in the discussion leading to
Proposition~\ref{prop:absorb-Fredholm}, one would get the following
statement: Suppose $s>1$, $s-3/2>\beta r$. Then with $\cX^{s,r}$,
$\cY^{s,r}$ as above, $\cP\colon\cX^{s,r}\to\cY^{s,r}$ is Fredholm. Here
the main loss is that one needs to assume $s>1$; this is done since in
the argument one needs to take $s'$ with $s'+1<s$ in order to
transition the normal operator estimates from $N(\cP)u$ to $\cP u$ and
still have a compact inclusion, but the normal
operator estimates need $s'\geq 0$ as, due to the boundary $H_2$, they
are again based on energy estimates. Using the direct global energy
estimate eliminates this loss, which is an artifact of combining local
energy estimates with the b-theory. In particular, in the complex
absorption setting, this problem does not arise, but on the other
hand, one need not have the forward support property of the solution.
\end{rmk}

The results of \cite{Va12} then are immediately applicable to obtain
an expansion of the solutions; the main point of the following theorem
being the elimination of the losses in differentiability in
\cite[Proposition~3.5]{Va12} due to Proposition~\ref{prop:b-saddle}.

\begin{thm}(Strengthened version of \cite[Proposition~3.5]{Va12}.)\label{thm:b-main}
Let $M$ be a manifold with a non-trapping b-metric $g$ as above, with boundary $X$
and let $\tau$ be a boundary defining function, $\cP$ as in \eqref{eq:cP-almost-Box}.
Suppose the domain $\Omega$ is as defined above, and $\frac{d\tau}{\tau}$
timelike.

Let $\sigma_j$ be the poles of $\widehat{\cP}^{-1}$, and let $\ell$ be such that $\im\sigma_j+\ell\notin\NN$ for all $j$. Let $\phi\in\CI(\RR)$ be such that $\supp\phi\subset(0,\infty)$, and $\phi\circ\frakt_1\equiv 1$ near $X\cap\Omega$.  Then for $s>3/2+\beta\ell$, there are $m_{jl}\in\NN$ such that solutions of $\cP u=f$ with $f\in \Hb^{s-1,\ell}(\Omega)^{\bullet,-}$, and with $u\in \Hb^{s_0,r_0}(\Omega)^{\bullet,-}$, $s\geq s_0\geq 1$, $s_0-1/2>\beta r_0$ satisfy that for some $a_{jl\kappa}\in\CI(X\cap\Omega)$,
\begin{equation}\label{eq:b-main-exp}
u'=u-\sum_j\sum_{l\in\NN}\sum_{\kappa\leq m_{jl}} \tau^{i\sigma_j+l}(\log
\tau)^\kappa (\phi\circ\frakt_1) a_{jl\kappa}
\in \Hb^{s,\ell}(\Omega)^{\bullet,-},
\end{equation}
where the sum is understood to be over a finite set with
$-\im\sigma_j+l<\ell$.

Here the (semi)norms of both $a_{jl\kappa}$ in $\CI(X\cap\Omega)$ and
$u'$ in $\Hb^{s,\ell}(\Omega)^{\bullet,-}$
are bounded by a constant times that of $f$ in
$\Hb^{s-1,\ell}(\Omega)^{\bullet,-}$.

The analogous result also holds if $f$ possesses an expansion modulo
$\Hb^{s-1,\ell}(\Omega)^{\bullet,-}$, namely
$$
f=f'+\sum_j\sum_{\kappa\leq m'_j}\tau^{\alpha_j}(\log
\tau)^\kappa (\phi\circ\frakt_1) a_{j\kappa},
$$
with $f'\in \Hb^{s-1,\ell}(\Omega)^{\bullet,-}$ and
$a_{j\kappa}\in\CI(X\cap\Omega)$, where terms corresponding to the
expansion of the $f$ are added to \eqref{eq:b-main-exp} in the sense of
the extended union of index sets \cite[\S5.18]{Me93}, recalled
below in Definition~\ref{def:index-set}.
\end{thm}

\begin{rmk}
Here the factor $\phi\circ\frakt_1$ is added to cut off the expansion
away from $H_1$, thus assuring that $u'$ is in the indicated space (a
supported distribution).

Also, the sum over $l$ is generated by the lack of dilation invariance
of $\cP$. If we take $\ell$ such that $-\im\sigma_j>\ell-1$ for all
$j$ then all the terms in the expansion arise directly from the
resonances, thus $l=0$ and $m_{j0}+1$ is the order of the pole of
$\widehat{\cP}^{-1}$ at $\sigma_j$, with the $a_{j0\kappa}$ being
resonant states.
\end{rmk}

\begin{proof}
First assume that $-\im\sigma_j>\ell$ for every $j$; thus there are no
terms subtracted from $u$ in \eqref{eq:b-main-exp}.
We proceed as in \cite[Proposition~3.5]{Va12}, but use the propagation
of singularities, in particular Propositions~\ref{prop:b-saddle} and \ref{prop:prop-sing-bdy}, to
eliminate the losses. First, by the propagation of singularities,
using $s_0-1/2>\beta r_0$ and $s\geq s_0$, $s\geq 0$,
$$
u\in\Hb^{s,r_0}(\Omega)^{\bullet,-}.
$$
Thus, as $\cP-N(\cP)\in\tau\Diffb^2(M)$,
\begin{equation}\label{eq:normal-op-perturb-exp-ds}
N(\cP)u=f-\wt f,\ \wt f=(\cP-N(\cP))u\in\Hb^{s-2,r_0+1}(\Omega)^{\bullet,-}
\end{equation}

We apply \cite[Lemma~3.1]{Va12} (using $s\geq s_0\geq 1$), which is the lossless version of
\cite[Proposition~3.5]{Va12} in the dilation invariant case.
Note that in \cite{Va12}, Lemma~3.1 is stated on the normal
operator space $M_\infty$,
which does not have a boundary face corresponding to $H_2$, i.e.\
$S_2\times(0,\infty)$, with complex absorption being used instead. However,
given the analysis on $X\cap\Omega$ discussed above, all the
arguments go through essentially unchanged: this is a Mellin transform
and contour deformation argument.

One thus
obtains \eqref{eq:b-main-exp} with $\ell$ replaced by $\ell'=\min(\ell,r_0+1)$ except that
$u=u'\in\Hb^{s-1,\ell'}(\Omega)^{\bullet,-}$ corresponding to the $\wt
f$ term in $N(\cP)u$ rather than
$u=u'\in\Hb^{s,\ell'}(\Omega)^{\bullet,-}$ as desired. However, using
$\cP u=f\in\Hb^{s-1,\ell'}(\Omega)^{\bullet,-}$, we deduce by the
propagation of singularities, using $s-1>\beta\ell'+1/2$, $s\geq 0$, that
$u=u'\in\Hb^{s,\ell'}(\Omega)^{\bullet,-}$. If $\ell\leq r_0+1$, we
have proved \eqref{eq:b-main-exp}. Otherwise we iterate, replacing
$r_0$ by $r_0+1$. We thus reach the conclusion, \eqref{eq:b-main-exp},
in finitely many steps.

If there are $j$ such that $-\im\sigma_j<\ell$, then in the first
step, when using \cite[Lemma~3.1]{Va12}, we obtain the partial
expansion $u_1$ corresponding to $\ell'=\min(\ell,r_0+1)$ in place of
$\ell$; here we may need to decrease $\ell'$ by an arbitrarily small
amount to make sure that $\ell'$ is not $-\im\sigma_j$ for any
$j$. Further, the terms of the partial expansion are annihilated by
$N(\cP)$, so $u'$ satisfies
$$
\cP u'=\cP u-N(\cP)u_1-(\cP-N(\cP))u_1\in \Hb^{s-1,\ell'}(\Omega)^{\bullet,-}
$$
as $(\cP-N(\cP))u_1\in \Hb^{\infty,r_0+1}(\Omega)^{\bullet,-}$ in fact
due to the conormality of $u_1$ and  $\cP-N(\cP)\in\tau\Diffb^2(M)$.
Correspondingly, the propagation of singularities result is applicable
as above to conclude that $u'\in\Hb^{s,\ell'}(\Omega)^{\bullet,-}$. If
$\ell\leq r_0+1$ we are done. Otherwise
we have better information on $\wt f$ in the next
step, namely 
$$
\wt f=(\cP-N(\cP))u=(\cP-N(\cP))u'+(\cP-N(\cP))u_1
$$
with the first term in $\Hb^{s-2,r_0+1}(\Omega)^{\bullet,-}$ (same as
in the case first considered above, without relevant resonances),
while the expansion of $u_1$ shows that $(\cP-N(\cP))u_1$ has a
similar expansion, but with an extra power of $\tau$ (i.e.\
$\tau^{i\sigma_j}$ is shifted to $\tau^{i\sigma_j+1}$). We can now apply
\cite[Lemma~3.1]{Va12} again; in the case of the terms arising from
the partial expansion, $u_1$, there are now new terms corresponding to
shifting the powers $\tau^{i\sigma_j}$ to $\tau^{i\sigma_j+1}$, as
stated in the referred Lemma, and possibly causing logarithmic terms
if $\sigma_j-i$ is also a pole of $\widehat{\cP}^{-1}$. Iterating in
the same manner proves the theorem when $f\in
\Hb^{s-1,\ell}(\Omega)^{\bullet,-}$. When $f$ has an expansion modulo
$\Hb^{s-1,\ell}(\Omega)^{\bullet,-}$, the same argument works;
\cite[Lemma~3.1]{Va12} gives the terms with the extended union, which
then further generate additional terms due to $\cP-N(\cP)$, just as
the resonance terms did.
\end{proof}

There is one problem with this theorem for the purposes of semilinear equations: the resonant terms with $\Im\sigma_j\geq 0$ which give rise to unbounded, or at most just bounded, terms in the expansion which become larger when one takes powers of these, or when one iteratively applies $\cP^{-1}$ (with the latter being the only issue if $\Im\sigma_j=0$ and the pole is simple).

Concretely, we now consider an asymptotically de Sitter space $(\wt M,\wt g)$. We then blow up
a point $p$ at the future boundary $\wt X_+$, as discussed in the
introduction, to obtain the analogue of the static model of de Sitter
space $M=[\wt M;p]$ with the pulled back metric $g$, which is a
b-metric near the front face (but away from the side face); let $\cP=\Box_g-\lambda$. If $\wt M$ is actual de Sitter space, then
$M$ is the actual static model; otherwise the metric of the
asymptotically de Sitter space, frozen at $p$, induces a de Sitter
metric, $g_0$, which is well defined at the front face of the blow
up $M$ (but away from its side faces) as a b-metric. In particular,
the resonances in the `static region' of any asymptotically de Sitter
space are the same as in the static model of actual de Sitter space.

On actual de Sitter space, the poles of $\widehat\cP^{-1}$ are those on
the hyperbolic space in the interior of the light cone equipped by a
potential, as described in \cite[Lemma~7.11]{Va07}, or indeed in
\cite[Proposition~4.2]{Va12} where essentially the present notation is
used.\footnote{In \cite[Lemma~7.11]{Va07} $-\sigma^2$ plays the same
role as $\sigma^2$ here or in \cite[Proposition~4.2]{Va12}.} As
shown in Corollary~7.18 of \cite{Va07}, converted to our notation, the
only possible poles are at
\begin{equation}\label{eq:dS-poles}
i\wh s_\pm(\lambda)-i\NN,\qquad\wh s_\pm(\lambda)=-\frac{n-1}{2}\pm\sqrt{\frac{(n-1)^2}{4}-\lambda}.
\end{equation}
In particular, when $\lambda=m^2$, $m>0$, then we conclude:

\begin{lemma}
\label{LemmaPerturbDS}
For $m>0$, $\cP=\Box_g-m^2$, $g$ induced by an asymptotically de
Sitter metric as above, all poles of $\widehat\cP^{-1}$ have strictly negative imaginary part.
\end{lemma}

In other words, for small mass $m>0$, there are no resonances $\sigma$
of the Klein-Gordon operator with $\Im\sigma\geq-\eps_0$ for some
$\eps_0>0$. Therefore, the expansion of $u$ as in
\eqref{eq:b-main-exp} no longer has a constant term. Let us fix
such $m>0$ and $\eps_0>0$, which ensures that for
$0<\eps<\eps_0$, the only term in the asymptotic expansion
\eqref{eq:b-main-exp}, when $s>1/2+\eps$ and
$f\in\Hb^{s-1,\eps}(\Omega)^{\bullet,-}$, is the `remainder' term
$u'\in\Hb^{s,\eps}(\Omega)^{\bullet,-}$. Here we use that $\beta=1$
in de Sitter space, hence on an asymptotically de Sitter space, see \cite[\S4.4]{Va12}, in particular the
second displayed equation after Equation~(4.16) there which computes
$\beta$ in accordance with Remark~\ref{rmk:beta-comp}.

Being interested in finding {\em forward solutions} to (non-linear) wave
equations on {\em asymptotically de Sitter spaces}, we now define the forward solution operator
\begin{equation}
\label{EqFwdSolDS}
S_\KG\colon\Hb^{s-1,\eps}(\Omega)^{\bullet,-}\to\Hb^{s,\eps}(\Omega)^{\bullet,-}
\end{equation}
using Theorems~\ref{ThmPFredholm} and \ref{thm:b-main}.

\begin{rmk}
If $\wt M\subset M$ is an asymptotically de Sitter space with global time function $t$, $\tau=e^{-t}$ is the defining function for future infinity, and the domain $\Omega$ is such that $\Omega\cap\wt M=\{\tau<\tau_0\}$, then $S_\KG$ in fact restricts to a forward solution operator on $\wt M$ itself; indeed, if $E\colon\Hb^{s-1,\eps}(\{\tau<\tau_0\})\to\Hb^{s-1,\eps}(\Omega)^{\bullet,-}$ is an extension operator, then the forward solution operator on $\{\tau<\tau_0\}$ is given by extending $f\in\Hb^{s-1,\eps}(\{\tau<\tau_0\})$ using $E$, finding the forward solution on $\Omega$ using $S_\KG$, and restricting back to $\{\tau<\tau_0\}$. The result is independent of the extension operator, as is easily seen from standard energy estimates; see in particular \cite[Proposition~3.9]{Va12}.
\end{rmk}

\subsection{A class of semilinear equations}
\label{SubsecStaticDSSemi}

Let us fix $m>0$ and $\eps_0>0$ as above for statements about
semilinear equations involving the Klein-Gordon operator; for
equations involving the wave operator only, let $-\eps_0$ be equal to
the largest imaginary part of all non-zero resonances of $\Box_g$.
In Theorem~\ref{ThmDSQu} and further in the subsequent sections bundles like $\Tb^*\Omega$
  refer to $\Tb^*_\Omega M$; the boundaries $H_j$ of $\Omega$ are
  regarded as artificial, and do not affect the cotangent bundle or
  the corresponding vector fields.

\begin{thm}
\label{ThmDSQu}
Let $0\leq\eps<\eps_0$ and $s>3/2+\eps$. Moreover, let $q\colon\Hb^{s,\eps}(\Omega)^{\bullet,-}\times\Hb^{s-1,\eps}(\Omega;\Tb^*\Omega)^{\bullet,-}\to\Hb^{s-1,\eps}(\Omega)^{\bullet,-}$ be a continuous function with $q(0,0)=0$ such that there exists a continuous non-decreasing function $L\colon\R_{\geq 0}\to\R$ satisfying
\[
\|q(u,\bdiff u)-q(v,\bdiff v)\|\leq L(R)\|u-v\|,\quad \|u\|,\|v\|\leq R,
\] 
where we use the norms corresponding to the map $q$. Then there is a constant $C_L>0$ so that the following holds: If $L(0)<C_L$, then for small $R>0$, there exists $C>0$ such that for all $f\in\Hb^{s-1,\eps}(\Omega)^{\bullet,-}$ with norm $\leq C$, the equation
\begin{equation}
\label{EqThmPDE}
(\Box_g-m^2)u=f+q(u,\bdiff u)
\end{equation}
has a unique solution $u\in\Hb^{s,\eps}(\Omega)^{\bullet,-}$, with norm $\leq R$, that depends continuously on $f$.

More generally, suppose
\[
q\colon\Hb^{s,\eps}(\Omega)^{\bullet,-}\times\Hb^{s-1,\eps}(\Omega;\Tb^*\Omega)^{\bullet,-}\times\Hb^{s-1,\eps}(\Omega)^{\bullet,-}\to\Hb^{s-1,\eps}(\Omega)^{\bullet,-}
\]
satisfies $q(0,0,0)=0$ and
\[
\|q(u,\bdiff u,w)-q(u',\bdiff u',w')\|\leq L(R)(\|u-u'\|+\|w-w'\|)
\]
provided $\|u\|+\|w\|,\|u'\|+\|w'\|\leq R$, where we use the norms corresponding to the map $q$, for a continuous non-decreasing function $L\colon\R_{\geq 0}\to\R$. Then there is a constant $C_L>0$ so that the following holds: If $L(0)<C_L$, then for small $R>0$, there exists $C>0$ such that for all $f\in\Hb^{s-1,\eps}(\Omega)^{\bullet,-}$ with norm $\leq C$, the equation
\begin{equation}
\label{EqThmPDEBox}
(\Box_g-m^2)u=f+q(u,\bdiff u,\Box_g u)
\end{equation}
has a unique solution $u\in\Hb^{s,\eps}(\Omega)^{\bullet,-}$, with $\|u\|_{\Hb^{s,\eps}}+\|\Box_g u\|_{\Hb^{s-1,\eps}}\leq R$, that depends continuously on $f$.

Further, if $\eps>0$ and the non-linearity is of the form $q(\bdiff u)$, with
\[
  q\colon\Hb^{s-1,\eps}(\Omega;\Tb^*\Omega)^{\bullet,-}\to\Hb^{s-1,\eps}(\Omega)^{\bullet,-}
\]
having a small Lipschitz constant near $0$, then for small $R>0$, there exists $C>0$ such that for all $f\in\Hb^{s-1,\eps}(\Omega)^{\bullet,-}$ with $\|f\|\leq C$, the equation
\[
\Box_g u=f+q(\bdiff u)
\]
has a unique solution $u$ with
$u-(\phi\circ\frakt_1)c=u'\in\Hb^{s,\eps}(\Omega)^{\bullet,-}$, where
$c\in\C$, that depends continuously on $f$, i.e.\ $c\in\C$ and
$u'\in\Hb^{s,\eps}(\Omega)^{\bullet,-}$ depend continuously on $f$. Here, $\phi\in\CI(\R)$ with support in $(0,\infty)$ and $\frakt_1$ are as in Theorem~\ref{thm:b-main}. In fact, the statement even holds for non-linearities $q(u,\bdiff u)$ provided
\[
  q\colon(\C(\phi\circ\frakt_1)\oplus\Hb^{s,\eps}(\Omega))\times\Hb^{s-1,\eps}(\Omega;\Tb^*\Omega)^{\bullet,-}\to\Hb^{s-1,\eps}(\Omega)^{\bullet,-}
\]
has a small Lipschitz constant near $0$.
\end{thm}
\begin{proof}
To prove the first part, let $S_\KG$ be the forward solution operator for $\Box_g-m^2$ as in \eqref{EqFwdSolDS}. We want to apply the Banach fixed point theorem to the operator $T_\KG\colon\Hb^{s,\eps}(\Omega)^{\bullet,-}\to\Hb^{s,\eps}(\Omega)^{\bullet,-}$, $T_\KG u=S_\KG(f+q(u,\bdiff u))$.

Let $C_L=\|S_\KG\|^{-1}$, then we have the estimate
\begin{equation}
\label{EqContractionQu}
\|T_\KG u-T_\KG v\|\leq\|S_\KG\|L(R')\|u-v\|\leq C_0\|u-v\|
\end{equation}
for $\|u\|,\|v\|\leq R$ and a constant $C_0<1$, granted that $L(R)\leq C_0\|S_\KG\|^{-1}$, which holds for small $R>0$ by assumption on $L$. Then, $T_\KG$ maps the $R$-ball in $\Hb^{s,\eps}(\Omega)^{\bullet,-}$ into itself if $\|S_\KG\|(\|f\|+L(R)R)\leq R$, i.e.\ if $\|f\|\leq R(\|S_\KG\|^{-1}-L(R))$. Put
\[
  C=R(\|S_\KG\|^{-1}-L(R)).
\]
Then the existence of a unique solution $u\in\Hb^{s,\eps}(\Omega)^{\bullet,-}$ with norm $\leq R$ to the PDE~\eqref{EqThmPDE} with $\|f\|_{\Hb^{s-1,\eps}}\leq C$ follows from the Banach fixed theorem.

To prove the continuous dependence of $u$ on $f$, suppose we are given $u_j\in\Hb^{s,\eps}(\Omega)^{\bullet,-}$, $j=1,2$, with norms $\leq R$, $f_j\in\Hb^{s-1,\eps}(\Omega)^{\bullet,-}$ with norms $\leq C$, such that
\[
(\Box_g-m^2)u_j=f_j+q(u_j,\bdiff u_j),\quad j=1,2.
\]
Then
\[
  (\Box_g-m^2)(u_1-u_2)=f_1-f_2+q(u_1,\bdiff u_1)-q(u_2,\bdiff u_2),
\]
hence
\[
  \|u_1-u_2\|\leq\|S_\KG\|(\|f_1-f_2\|+L(R)\|u_1-u_2\|),
\]
which in turn gives
\[
  \|u_1-u_2\|\leq\frac{\|f_1-f_2\|}{1-C_0}.
\]
This completes the proof of the first part.

For the more general statement, we use that one can think of $\Box_g$ in the non-linearity as a first order operator. Concretely, we work on the coisotropic space
\[
\cX=\{u\in\Hb^{s,\eps}(\Omega)^{\bullet,-}\colon\Box_g u\in\Hb^{s-1,\eps}(\Omega)^{\bullet,-}\}
\]
with norm
\[
\|u\|_{\cX}=\|u\|_{\Hb^{s,\eps}(\Omega)^{\bullet,-}}+\|\Box_g u\|_{\Hb^{s-1,\eps}(\Omega)^{\bullet,-}}.
\]
This is a Banach space: Indeed, if $(u_k)$ is a Cauchy sequence in $\cX$, then $u_k\to u$ in $\Hb^{s,\eps}(\Omega)^{\bullet,-}$, and $\Box_g u_k\to v$ in $\Hb^{s-1,\eps}(\Omega)^{\bullet,-}$; in particular, $\Box_g u_k\to \Box_g u$ and $\Box_g u_k\to v$ in $\tau^\eps\Hb^{s-2}(\Omega)^{\bullet,-}$, thus $\Box_g u=v\in\Hb^{s-1,\eps}(\Omega)^{\bullet,-}$, which was to be shown. We then define $T_\KG\colon\cX\to\cX$ by $T_\KG u=S_\KG(f+q(u,\bdiff u,\Box_g u))$ and obtain the estimate
\begin{align*}
\|T_\KG u-T_\KG v\|_{\cX}&=\|T_\KG u-T_\KG v\|_{\Hb^{s,\eps}}+\|q(u,\bdiff u,\Box_g u)-q(v,\bdiff v,\Box_g v)\|_{\Hb^{s-1,\eps}} \\
&\leq(\|S_\KG\|+1)L(R)(\|u-v\|_{\Hb^{s,\eps}}+\|\Box_g u-\Box_g v\|_{\Hb^{s-1,\eps}}) \\
&=(\|S_\KG\|+1)L(R)\|u-v\|_{\cX}\leq C_0\|u-v\|_{\cX}
\end{align*}
for $u,v\in\cX$ with norms $\leq R$, with $C_0<1$ if $R>0$ is small enough, provided we require $L(0)<C_L:=(\|S_\KG\|+1)^{-1}$. Then, for $u\in\cX$ with norm $\leq R$,
  \[
    \|T_\KG u\|_{\cX}\leq(\|S_\KG\|+1)(\|f\|_{\Hb^{s-1,\eps}}+L(R)R)\leq R
  \]
  if $\|f\|\leq C$, $C>0$ small. Thus, $T_\KG$ is a contraction on $\cX$, and we obtain the solvability of equation \eqref{EqThmPDEBox}. The continuous dependence of the solution on the forcing term $f$ is proved as above.

  For the third part, we use the forward solution operator
  $S\colon\Hb^{s-1,\eps}(\Omega)^{\bullet,-}\to\cY:=\C\oplus\Hb^{s,\eps}(\Omega)^{\bullet,-}$
  for $\Box_g$; note that $\cY$ is a Banach space with norm $\|(c,u')\|_{\cY}=|c|+\|u'\|_{\Hb^{s,\eps}(\Omega)^{\bullet,-}}$. (See \S\ref{SubsecDSPoly} for related and more general statements.) We will apply the Banach fixed point theorem to the operator $T\colon\cY\to\cY$, $Tu=S(f+q(u,\bdiff u))$: We again have an estimate like \eqref{EqContractionQu}, since $\bdiff u\in\Hb^{s-1,\eps}(\Omega;\Tb^*\Omega)^{\bullet,-}$ for $u\in\cY$, and for small $R>0$, $T$ maps the $R$-ball around $0$ in $\cY$ into itself if the norm of $f$ in $\Hb^{s-1,\eps}(\Omega)^{\bullet,-}$ is small, as above. The continuous dependence of the solution on the forcing term is proved as above.
\end{proof}

The following basic statement ensures that there are interesting non-linearities $q$ that satisfy the requirements of the theorem; see also \S\ref{SubsecDSPoly}.
\begin{lemma}
\label{LemmaHsAlgebra}
  Let $s>n/2$, then $\Hb^s(\R^n_+)$ is an algebra. In particular, $\Hb^s(N)$ is an algebra on any compact $n$-dimensional manifold $N$ with boundary which is equipped with a b-metric.
\end{lemma}
\begin{proof}
  The first statement is the special case $k=0$ of
  Lemma~\ref{LemmaYDerivatives} after a logarithmic change of
  coordinates, which gives an isomorphism $\Hb^s(\R^n_+)\cong H^s(\R^n)$; the lemma is well-known in this case, see e.g.\ \cite[Chapter~13.3]{Tay}. The second statement follows by localization and from the coordinate invariance of $\Hb^s$.
\end{proof}
More and related statements will be given in \S\ref{SubsecDeSitterAlgebra}.

\begin{rmk}
  The algebra property of $\Hb^s(N)$ for $s>\dim(N)/2$ is a special
  case of the fact that for any $F\in\CI(\R)$, for real valued $u$, or
  $F\in\CI(\Cx)$, for complex valued $u$, with $F(0)=0$, the
composition map $\Hb^s(N)\ni u\mapsto F\circ u\in\Hb^s(N)$ is
well-defined and continuous, see for example
\cite[Chapter 13.10]{Tay}. In the real valued $u$ case, if $F(0)\neq 0$, then writing
$F(t)=F(0)+tF_1(t)$ shows that $F\circ u\in\Cx+\Hb^s(N)$. If $r>0$,
then $\Hb^{s,r}(N)\subset \Hb^s(N)$ shows that $F_1(u)\in\Hb^s(N)$, thus $F\circ u=F(0)+uF_1(u)\in\Cx+\Hb^{s,r}(N)$; and if $F$ vanishes to order $k$ at $0$
then $F(t)=t^k F_k(t)$, so $F\circ u=u^k (F_k\circ u)$, and the
multiplicative properties of $\Hb^{s,r}(N)$ show that $F\circ u\in
\Hb^{s,kr}(N)$. The argument is analogous for complex valued $u$,
indeed for $\RR^L$-valued $u$, using Taylor's theorem on $F$ at the origin.
\end{rmk}

As a corollary we have

\begin{cor}\label{cor:ndS}
If $s>n/2$, the hypotheses of Theorem~\ref{ThmDSQu}
hold for non-linearities $q(u)=cu^p$, $p\geq 2$ integer, $c\in\Cx$, as
well as $q(u)=q_0 u^p$, $q_0\in\Hb^s(M)$.

If $s-1>n/2$, the hypotheses of Theorem~\ref{ThmDSQu}
hold for non-linearities $q$
\begin{equation}
\label{EqDSPolynomial}
q(u,\bdiff u)=\sum_{2\leq j+|\alpha|\leq d} q_{j\alpha} u^j \prod_{k\leq|\alpha|}X_{\alpha,k}u,
\end{equation}
$q_{j,\alpha}\in\Cx+\Hb^s(M)$, $X_{\alpha,k}\in\Vb(M)$.

Thus, in either case, for $m>0$, $0\leq\eps<\eps_0$, $s>3/2+\eps$ and for small $R>0$, there exists $C>0$ such that for all $f\in\Hb^{s-1,\eps}(\Omega)^{\bullet,-}$ with norm $\leq C$, the equation
  \begin{equation}
  \label{EqThmPDE-cor}
    (\Box_g-m^2)u=f+q(u,\bdiff u)
  \end{equation}
  has a unique solution $u\in\Hb^{s,\eps}(\Omega)^{\bullet,-}$, with
  norm $\leq R$, that depends continuously on $f$.

The analogous conclusion also holds for $\Box_g u=f+q(u,\bdiff u)$ provided $\eps>0$ and
\begin{equation}
\label{EqDSPolynomial2}
q(u,\bdiff u)=\sum_{2\leq j+|\alpha|\leq d,|\alpha|\geq 1} q_{j\alpha} u^j \prod_{k\leq|\alpha|}X_{\alpha,k}u,
\end{equation}
with the solution being in $\Cx (\phi\circ\frakt_1)\oplus\Hb^{s,\eps}(\Omega)^{\bullet,-}$, $\phi\circ\frakt_1$ identically $1$ near $X\cap\Omega$, vanishing near $H_1$.
\end{cor}

\begin{rmk}
  For such polynomial non-linearities, the Lipschitz constant $L(R)$ in the statement of Theorem~\ref{ThmDSQu} satisfies $L(0)=0$.
\end{rmk}

\begin{rmk}
  In this paper, we do not prove that one obtains smooth (i.e.\ conormal) solutions if the forcing term is smooth (conormal); see \cite{HintzQuasilinearDS} for such a result in the quasilinear setting.
\end{rmk}

Since in Theorem~\ref{ThmDSQu} we allow $q$ to depend on $\Box_g u$, we can in particular solve certain quasilinear equations if $s>\max(1/2+\eps,n/2+1)$: Suppose for example that $q'\colon\Hb^{s,\eps}(\Omega)^{\bullet,-}\to\Hb^{s-1}(\Omega)^{\bullet,-}$ is continuous with $\|q'(u)-q'(v)\|\leq L'(R)\|u-v\|$ for $u,v\in\Hb^{s,\eps}(\Omega)^{\bullet,-}$ with norms $\leq R$, where $L'\colon\R_{\geq 0}\to\R$ is locally bounded, then we can solve the equation
\[
  (1+q'(u))(\Box_g-m^2)u=f\in\Hb^{s-1,\eps}(\Omega)^{\bullet,-}
\]
provided the norm of $f$ is small. Indeed, put $q(u,w)=-q'(u)(w-m^2u)$, then $q(u,\Box_g u)=-q'(u)(\Box_g-m^2)u$ and the PDE becomes
\[
  (\Box_g-m^2)u=f+q(u,\Box_g u),
\]
which is solvable by Theorem~\ref{ThmDSQu}, since, with $\|\cdot\|=\|\cdot\|_{\Hb^{s-1,\eps}}$, for $u,u'\in\Hb^{s,\eps}(\Omega)^{\bullet,-}$, $w,w'\in\Hb^{s-1,\eps}(\Omega)^{\bullet,-}$ with $\|u\|+\|w\|,\|u'\|+\|w'\|\leq R$, we have
\begin{align*}
  \|q&(u,w)-q(u',w')\| \\
 &\leq \|q'(u)-q'(u')\|\|w-m^2u\|+\|q'(u')\|\|w-w'-m^2(u-u')\| \\
 &\leq L'(R)((1+m^2)R+m^2R)\|u-u'\|+L'(R)R\|w-w'\| \\
 &\leq L(R)(\|u-u'\|+\|w-w'\|)
\end{align*}
with $L(R)\to 0$ as $R\to 0$.
  
By a similar argument, one can also allow $q'$ to depend on $\bdiff u$ and $\Box_g u$.

\begin{rmk}\label{rmk:ODE}
  Recalling the discussion following Theorem~\ref{thm:b-main}, let us emphasize the importance of $\widehat P(\sigma)^{-1}$ having no poles in the closed upper half plane by looking at the explicit example of the operator $\cP=\pa_x$ in $1$ dimension. In terms of $\tau=e^{-x}$, we have $\cP=-\tau\pa_\tau$, thus $\widehat P(\sigma)=-i\sigma$, considered as an operator on the boundary (which is a single point) at $+\infty$ of the radial compactification of $\R$; hence $\widehat P(\sigma)^{-1}$ has a simple pole at $\sigma=0$, corresponding to constants being annihilated by $\cP$. Now suppose we want to find a forward solution of $u'=u^2+f$, where $f\in\CI_c(\R)$. In the first step of the iterative procedure described above, we will obtain a constant term; the next step gives a term that is linear in $x$ ($x$ being the antiderivative of $1$), i.e.\ in $\log\tau$, then we get quadratic terms and so on, therefore the iteration does not converge (for general $f$), which is of course to be expected, since solutions to $u'=u^2+f$ in general blow up in finite time. On the other hand, if $\cP=\pa_x+1$, then $\widehat P(\sigma)^{-1}=(1-i\sigma)^{-1}$, which has a simple pole at $\sigma=-i$, which means that forward solutions $u$ of $u'+u=u^2+f$ with $f$ as above can be constructed iteratively, and the first term of the expansion of $u$ at $+\infty$ is $c\tau^{i(-i)}=ce^{-x}$, $c\in\Cx$.
\end{rmk}

\subsection{Semilinear equations with polynomial non-linearity}
\label{SubsecDSPoly}

With polynomial non-linearities as in \eqref{EqDSPolynomial}, we can use the second part of Theorem~\ref{thm:b-main} to obtain an asymptotic expansion for the solution; see Remark~\ref{RemAPosterioriExp} and, in a slightly different setting, \S\ref{SubsecKdSSemi} for details on this. Here, we instead define a space that encodes asymptotic expansions directly in such a way that we can run a fixed point argument directly.

To describe the exponents appearing in the expansion, we use index sets as introduced by Melrose, see \cite{Me93}.

\begin{definition}
\label{def:index-set}\ 
  \begin{enumerate}
  \item An \emph{index set} is a discrete subset $\sE$ of $\C\times\N_0$ satisfying the conditions
    \begin{enumerate}
      \item $(z,k)\in\sE$ $\Rightarrow$ $(z,j)\in\sE$ for $0\leq j\leq k$, and
      \item If $(z_j,k_j)\in\sE$, $|z_j|+k_j\to\infty$ $\Rightarrow$ $\Re z_j\to\infty$.
    \end{enumerate}
  \item For any index set $\sE$, define
    \[
	  w_{\sE}(z)=\begin{cases}
	    \max\{k\in\N_0\colon(z,k)\in\sE\},&(z,0)\in\sE \\
		-\infty &\tn{otherwise}.
	   \end{cases}
	\]
  \item For two index sets $\sE,\sE'$, define their extended union by
    \[
	  \sE\ecup\sE'=\sE\cup\sE'\cup\{(z,l+l'+1)\colon (z,l)\in\sE,(z,l')\in\sE'\}
	\]
	and their product by $\sE\sE'=\{(z+z',l+l')\colon (z,l)\in\sE,(z',l')\in\sE'\}$. We shall write $\sE^k$ for the $k$-fold product of $\sE$ with itself.
  \item A \emph{positive index set} is an index set $\sE$ with the property that $\Re z>0$ for all $z\in\C$ with $(z,0)\in\sE$.
  \end{enumerate}
\end{definition}

\begin{rmk}
  To ensure that the class of polyhomogeneous conormal distributions with a given index set $\sE$ is invariantly defined, Melrose \cite{Me93} in addition requires that $(z,k)\in\sE$ implies $(z+j,k)\in\sE$ for all $j\in\N_0$. In particular, this is a natural condition in non dilation-invariant settings as in Theorem~\ref{thm:b-main}. A convenient way to enforce this condition in all relevant situations is to enlarge the index set corresponding to the poles of the inverse of the normal operator accordingly; see the statement of Theorem~\ref{ThmDSPoly}.

  Observe though that this condition is not needed in the dilation-invariant cases of the solvability statements below.
\end{rmk}

Since we want to capture the asymptotic behavior of solutions near $X\cap\Omega$, we fix a cutoff $\phi\in\CI(\R)$ with support in $(0,\infty)$ such that $\phi\circ\frakt_1\equiv 1$ near $X\cap\Omega$ (we already used such a cutoff in Theorem~\ref{thm:b-main}), and make the following definition.

\begin{definition}
  Let $\sE$ be an index set, and let $s,r\in\R$. For $\eps>0$ with the property that there is no $(z,0)\in\sE$ with $\Re z=\eps$, define the space $\cX_{\sE}^{s,r,\eps}$ to consist of all tempered distributions $v$ on $M$ with support in $\bar\Omega$ such that
  \begin{equation}
  \label{EqXDefExpansion}  v'=v-\sum_{\genfrac{}{}{0pt}{}{(z,k)\in\sE}{\Re z<\eps}}\tau^z(\log\tau)^k (\phi\circ\frakt_1)v_{z,k}\in\Hb^{s,\eps}(\Omega)^{\bullet,-}
  \end{equation}
  for certain $v_{z,k}\in H^r(X\cap\Omega)$.
\end{definition}

Observe that the terms $v_{z,k}$ in the expansion \eqref{EqXDefExpansion} are uniquely determined by $v$, since $\eps>\Re z$ for all $z\in\C$ for which $(z,0)$ appears in the sum \eqref{EqXDefExpansion}; then also $v'$ are uniquely determined by $v$. Therefore, we can use the isomorphism
\[
  \cX_{\sE}^{s,r,\eps}\cong\Bigl(\bigoplus_{\genfrac{}{}{0pt}{}{(z,k)\in\sE}{\Re z<\eps}}H^r(X\cap\Omega)\Bigr)\oplus\Hb^{s,\eps}(\Omega)^{\bullet,-}
\]
to give $\cX_{\sE}^{s,r,\eps}$ the structure of a Banach space.

\begin{lemma}
\label{LemmaLimIndex}
  Let $\sP,\sF$ be positive index sets, and let $\eps>0$. Define $\sE'_0=\sP\ecup\sF$ and recursively $\sE'_{N+1}=\sP\ecup\bigl(\sF\cup\bigcup_{k\geq 2}(\sE'_N)^k\bigr)$; put $\sE_N=\{(z,k)\in\sE'_N\colon 0<\Re z\leq\eps\}$. Then there exists $N_0\in\N$ such that $\sE_N=\sE_{N_0}$ for all $N\geq N_0$; moreover, the limiting index set $\sE_\infty(\sP,\sF,\eps):=\sE_{N_0}$ is finite.
\end{lemma}
\begin{proof}
Writing $\pi_1\colon\C\times\N_0\to\C$ for the projection, one has
\[
  \pi_1\sE_1=\Bigl\{z\colon 0<\Re z\leq\eps, z=\sum_{j=1}^k z_j\colon k\geq 1,z_j\in\pi_1\sE_0\Bigr\},
\]
and it is then clear that $\pi_1\sE_N=\pi_1\sE_1$ for all $N\geq 1$. Since $\sE_0$ is a positive index set, there exists $\delta>0$ such that $\Re z\geq\delta$ for all $z\in\sE_0$; hence $\pi_1\sE_\infty=\pi_1\sE_1$ is finite.

To finish the proof, we need to show that for all $z\in\C$, the number $w_{\sE_N}(z)$ stabilizes. Defining $p(z)=w_{\sP}(z)+1$ for $z\in\pi_1\sP$ and $p(z)=0$ otherwise, we have a recursion relation
\begin{equation}
\label{EqIndexSetRecursion}
  w_{\sE_N}(z)=p(z)+\max\biggl\{w_{\sF}(z),\max_{\genfrac{}{}{0pt}{}{z=z_1+\cdots+z_k}{k\geq 2,z_j\in\pi_1\sE_\infty}}\Bigl\{\sum_{j=1}^k w_{\sE_{N-1}}(z_j)\Bigr\}\biggr\},\quad N\geq 1.
\end{equation}
For each $z_j$ appearing in the sum, we have $\Im z_j\leq\Im z-\delta$. Thus, we can use \eqref{EqIndexSetRecursion} with $z$ replaced by such $z_j$ and $N$ replaced by $N-1$ to express $w_{\sE_N}(z)$ in terms of a finite number of $p(z_\alpha)$ and $w_{\sF}(z_\alpha)$, $\Im z_\alpha\leq\Im z$, and a finite number of $w_{\sE_{N-2}}(z_\beta)$, $z_\beta\leq\Im z-2\delta$. Continuing in this way, after $N_0=\lfloor{(\Im z)/\delta\rfloor}+1$ steps we have expressed $w_{\sE_N(z)}$ in terms of a finite number of $p(z_\gamma)$ and $w_{\sF}(z_\gamma)$, $\Im z_\gamma\leq\Im z$, only, and this expression is independent of $N$ as long as $N\geq N_0$.
\end{proof}
  
\begin{definition}
  Let $\sP,\sF$ be positive index sets, and let $\eps>0$ be such that there is no $(z,0)\in\sE_\infty(\sP,\sF,\eps)$ with $\Re z=\eps$, with $\sE_\infty(\sP,\sF,\eps)$ as defined in the statement of Lemma~\ref{LemmaLimIndex}. Then for $s,r\in\R$, define the Banach spaces
  \begin{align*}
    \cX_{\sP,\sF}^{s,r,\eps}&:=\cX_{\sE_\infty(\sP,\sF,\eps)}^{s,r,\eps}, \\
	{}^0\cX_{\sP,\sF}^{s,r,\eps}&:=\cX_{\sE_\infty(\sP,\sF,\eps)\cup\{(0,0)\}}^{s,r,\eps}.
  \end{align*}
\end{definition}

Note that the spaces ${}^{(0)}\cX_{\sP,\sF}^{s,s,\eps}$ are Banach
algebras for $s>n/2$ in the sense that there is a constant $C>0$ such that $\|uv\|\leq C\|u\|\|v\|$ for all $u,v\in{}^{(0)}\cX_{\sP,\sF}^{s,s,\eps}$. Moreover, $\cX_{\sP,\sF}^{s,s,\eps}$ interacts well with the forward solution operator $S_\KG$ of $\Box_g-m^2$ in the sense that $u\in\cX_{\sP,\sF}^{s,s,\eps}$, $k\geq 2$, with $\sP$ being related to the poles of $\widehat\cP(\sigma)^{-1}$, where $\cP=\Box_g-m^2$, as will be made precise in the statement of Theorem~\ref{ThmDSPoly} below, implies $S_\KG(u^k)\in\cX_{\sP,\sF}^{s,s,\eps}$.

We can now state the result giving an asymptotic expansion of the solution of $(\Box_g-m^2)u=f+q(u,\bdiff u)$ for polynomial non-linearities $q$.

\begin{thm}
\label{ThmDSPoly}
 Let $\eps>0,s>\max(3/2+\eps,n/2+1)$, and $q$ as in
 \eqref{EqDSPolynomial}. Moreover, if $\sigma_j\in\C$ are the poles of
 the inverse family $\widehat\cP(\sigma)^{-1}$, where
 $\cP=\Box_g-m^2$, and $m_j+1$ is the order of the pole of
 $\widehat\cP(\sigma)^{-1}$ at $\sigma_j$, let
 $\sP=\{(i\sigma_j+k,\ell)\colon 0\leq\ell\leq m_j,k\in\N_0\}$. Assume
 that $\eps\neq\Re(i\sigma_j)$ for all $j$, and that moreover $m>0$,
 which implies that $\sP$ is a positive index set; see Lemma~\ref{LemmaPerturbDS}. Finally, let $\sF$ be a positive index set.
 
 Then for small enough $R>0$, there exists $C>0$ such that for all $f\in\cX_{\sF}^{s-1,s-1,\eps}$ with norm $\leq C$, the equation
  \[
    (\Box_g-m^2)u=f+q(u,\bdiff u)
  \]
  has a unique solution $u\in\cX_{\sP,\sF}^{s,s,\eps}$, with norm $\leq R$, that depends continuously on $f$; in particular, $u$ has an asymptotic expansion with remainder term in $\Hb^{s,\eps}(\Omega)^{\bullet,-}$.

  Further, if the polynomial non-linearity is of the form $q(\bdiff u)$, then for small $R>0$, there exists $C>0$ such that for all $f\in\cX_{\sF}^{s-1,s-1,\eps}$ with norm $\leq C$, the equation
  \[
    \Box_g u=f+q(\bdiff u)
  \]
  has a unique solution $u\in{}^0\cX_{\sP,\sF}^{s,s,\eps}$, with norm $\leq R$, that depends continuously on $f$.
\end{thm}
\begin{proof}
 By Theorem~\ref{thm:b-main} and the definition of the space $\cX=\cX_{\sP,\sF}^{s,s,\eps}$, we have a forward solution operator $S_\KG\colon\cX\to\cX$ of $\Box_g-m^2$. Thus, we can apply the Banach fixed point theorem to the operator $T\colon\cX\to\cX$, $Tu=S_\KG(f+q(u,\bdiff u))$, where we note that $q\colon\cX\to\cX$, which follows from the definition of $\cX$ and the fact that $q$ is a polynomial only involving terms of the form $u^j\prod_{k\leq|\alpha|}X_{\alpha,k}u$ for $j+|\alpha|\geq 2$. This condition on $q$ also ensures that $T$ is a contraction on a sufficiently small ball in $\cX_+$.

 For the second part, writing ${}^0\cX={}^0\cX_{\sP,\sF}^{s,s,\eps}$, we have a forward solution operator $S\colon\cX\to{}^0\cX$. But $q(\bdiff u)\colon{}^0\cX\to\cX$, since $\bdiff$ annihilates constants, and we can thus finish the proof as above.

 The continuous dependence of the solution on the right hand side is proved as in the proof of Theorem~\ref{ThmDSQu}.
\end{proof}

Note that $\eps>0$ is (almost) unrestricted here, and thus we can get arbitrarily many terms in the asymptotic expansion if we work with arbitrarily high Sobolev spaces.

The condition that the polynomial $q(u,\bdiff u)$ does not involve a linear term is very important as it prevents logarithmic terms from stacking up in the iterative process used to solve the equation. Also, adding a term $\nu u$ to $q(u,\bdiff u)$ effectively changes the Klein-Gordon parameter from $-m^2$ to $\nu-m^2$, which will change the location of the poles of $\widehat P(\sigma)^{-1}$; in the worst case, if $\nu>m^2$, this would even cause a pole to move to $\Im\sigma>0$, corresponding to a resonant state that blows up exponentially in time. Lastly, let us remark that the form \eqref{EqDSPolynomial2} of the non-linearity is not sufficient to obtain an expansion beyond leading order, since in the iterative procedure, logarithmic terms would stack up in the next-to-leading order term of the expansion.

\begin{rmk}
\label{RemAPosterioriExp}
  Instead of working with the spaces ${}^{(0)}\cX_{\sP,\sF}^{s,s,\eps}$, which have the expansion built in, one could alternatively first prove the existence of a solution $u$ in a (slightly) decaying b-Sobolev space, which then allows one to regard the polynomial non-linearity as a perturbation of the linear operator $\Box_g-m^2$; then an iterative application of the dilation-invariant result \cite[Lemma~3.1]{Va12} gives an expansion of the solution to the non-linear equation. We will follow this idea in the discussion of polynomial non-linearities on asymptotically Kerr-de Sitter spaces in the next section.
\end{rmk}


\section{Kerr-de Sitter space}
\label{SecKdS}

In this section we analyze semilinear waves on Kerr-de Sitter space,
and more generally on spaces with normally hyperbolic trapping,
discussed below. The effect of the latter is a loss of derivatives for
the linear estimates in
general, but we show that at least derivatives with principal symbol
vanishing on the trapped set are well-behaved. We then use these
results to solve semilinear equations in the rest of the section.

\subsection{Linear Fredholm theory}\label{SecKdS-Fredholm}

The linear theorem in the case of normally hyperbolic trapping
for dilation-invariant operators $\cP=\Box_g-\lambda$ is the following:

\begin{thm}(See \cite[Theorem~1.4]{Va12}.)\label{thm:KdS-main}
Let $M$ be a manifold with a b-metric $g$ as above,
with boundary $X$, and let
$\tau$ be the boundary defining function, $\cP$ as in \eqref{eq:cP-almost-Box}.
If $g$ has normally hyperbolic
trapping, $\frakt_1,\Omega$ are as above, $\phi\in\CI(\R)$ as in Theorem~\ref{thm:b-main},
then there exist $C'>0$, $\varkappa>0$, $\beta\in\R$ such that
for $0\leq\ell<C'$ and $s>1/2+\beta\ell$, $s\geq 0$,
solutions $u\in\Hb^{-\infty,-\infty}(\Omega)^{\bullet,-}$ of $(\Box_g-\lambda)u=f$
with $f\in\Hb^{s-1+\varkappa,\ell}(\Omega)^{\bullet,-}$
satisfy that for some $a_{j\kappa}\in\CI(\Omega\cap X)$ (which are
the resonant states) and $\sigma_j\in\Cx$ (which are the resonances),
\begin{equation}
\label{EqAsympExpansion}
  u'=u-\sum_j\sum_{\kappa\leq m_j} \tau^{i\sigma_j}(\log \tau)^\kappa (\phi\circ\frakt_1)a_{j\kappa} \in\Hb^{s,\ell}(\Omega)^{\bullet,-}.
\end{equation}
Here the (semi)norms of both $a_{j\kappa}$ in $\CI(\Omega\cap X)$ and $u'$ in $\Hb^{s,\ell}(\Omega)^{\bullet,-}$ are bounded by a constant times that of $f$ in $\Hb^{s-1+\varkappa,\ell}(\Omega)^{\bullet,-}$.
The same conclusion holds for sufficiently small perturbations of the
metric as a symmetric bilinear form on
$\Tb M$ provided the trapping is normally hyperbolic.
\end{thm}

In order to state the analogue of Theorems~\ref{ThmPFredholm}-\ref{thm:b-main} when one has {\em normally hyperbolic trapping} at $\Gamma\subset\Sb^*_XM$, we will employ non-trapping estimates in certain so-called normally isotropic functions spaces, established in \cite{HintzVasyNormHyp}. To put our problem into the context of \cite{HintzVasyNormHyp}, we need some notation in addition to that in \S\ref{SecStaticDeSitter}: In the setting of \S\ref{SecStaticDeSitter}, as leading up to Theorem~\ref{ThmPFredholm}, see the discussion above Figure~\ref{FigDS}, we define
\begin{enumerate}
  \item the {\em forward trapped set in $\Sigma_+$} as the set of points in $\Sigma_\Omega\cap(\Sigma_+\setminus L_+)$ the bicharacteristics through which do not flow (within $\Sigma_\Omega$) to $\Sb^*_{H_1}M\cup L_+$ in the forward direction (i.e.\ they do not reach $\Sb^*_{H_1}M$ in finite time and they do not tend to $L_+$),
  \item the  {\em backward trapped set in $\Sigma_+$} as the set of points in $\Sigma_\Omega\cap(\Sigma_+\setminus L_+)$ the bicharacteristics through which do not flow  to $\Sb^*_{H_2} M\cup L_+$ in the backward direction,
  \item the {\em forward trapped set in $\Sigma_-$} as the set of points in $\Sigma_\Omega\cap(\Sigma_-\setminus L_-)$ the bicharacteristics through which do not flow to $\Sb^*_{H_2} M\cup L_-$ in the forward direction,
  \item the {\em backward trapped set in $\Sigma_-$} as the set of points in $\Sigma_\Omega\cap(\Sigma_-\setminus L_-)$ the bicharacteristics through which do not flow to $\Sb^*_{H_1}M\cup L_-$ in the backward direction.
\end{enumerate}
The {\em forward trapped set} $\Gamma_-$ is the union of the forward trapped sets in $\Sigma_\pm$, and analogously for the {\em backward trapped set} $\Gamma_+$. The {\em trapped set} $\Gamma$ is the intersection of the forward and backward trapped sets. We say that $\cP$ is {\em normally hyperbolically trapping}, or has normally hyperbolic trapping, if $\Gamma\subset\Sb^*_XM$ is b-normally hyperbolic in the sense discussed in \cite[\S3.2]{HintzVasyNormHyp}.

Following \cite{HintzVasyNormHyp}, we introduce replacements for the b-Sobolev spaces used in \S\ref{SecStaticDeSitter} which are called {\em normally isotropic at $\Gamma$}; these spaces $\bnormiso^s$, see also \eqref{eq:b-norm-isotropic-KdS}, and dual spaces $\bnormiso^{*,-s}$ are just the standard b-Sobolev spaces $\Hb^s(M)$, resp.\ $\Hb^{-s}(M)$, microlocally away from $\Gamma$.

Concretely, suppose $\Gamma$ is locally (in a neighborhood $U_0$ of
$\Gamma$) defined by $\tau=0$, $\phi_+=\phi_-=0$, $\wh p=0$ in $\Sb^*M$,
with $d\tau,d\phi_+,d\phi_-,d\wh p$, $\wh p=\wt\rho^m p$, linearly
independent at $\Gamma$. Here one should
  think of $\phi_-$ as being a defining function of
  $\Gamma_\pm\cap\Sigma_\pm$ (with the either the top or the bottom
  choice of sign in both $\pm$) within $\Sb^* M$, and $\phi_+$ of
  $\Gamma_\pm\cap\Sigma_\mp$ within $\Sb^*_X M$. Then taking any $Q_\pm\in\Psib^0(M)$ with principal symbol $\phi_\pm$, $\wh P\in\Psib^0(M)$ with principal symbol $\wh p$, and $Q_0\in\Psib^0(M)$ elliptic on $U_0^c$ with $\WFb'(Q_0)\cap\Gamma=\emptyset$, we define the (global) b-normally isotropic spaces at $\Gamma$ of order $s$, $\bnormiso^s=\bnormiso^s(M)$, by the norm
\begin{equation}
\label{eq:b-norm-isotropic-KdS}
\|u\|_{\bnormiso^s}^2=\|Q_0 u\|^2_{\Hb^s}+\|Q_+u\|^2_{\Hb^s}+\|Q_-u\|^2_{\Hb^s}+\|\tau^{1/2} u\|^2_{\Hb^s}+\|\wh P u\|^2_{\Hb^{s}}+\|u\|^2_{\Hb^{s-1/2}},
\end{equation}
and let $\bnormiso^{*,-s}$ be the dual space relative to $L^2$ which is
\[
Q_0\Hb^{-s}+Q_+\Hb^{-s}+Q_-\Hb^{-s}+\tau^{1/2}\Hb^{-s}+\wh P\Hb^{-s}+\Hb^{-s+1/2}.
\]
In particular,
\begin{equation}\begin{split}\label{eq:b-Sobolev-normiso}
&\Hb^s(M)\subset\bnormiso^s(M)\subset\Hb^{s-1/2}(M)\cap\Hb^{s,-1/2}(M),\\
&\Hb^{s+1/2}(M)+\Hb^{s,1/2}(M)\subset\bnormiso^{*,s}(M)\subset\Hb^s(M).
\end{split}\end{equation}
Microlocally away from $\Gamma$, $\bnormiso^s(M)$ is indeed just the standard $\Hb^s$ space while $\bnormiso^{*,-s}$ is $\Hb^{-s}$ since at least one of $Q_0$, $Q_\pm$, $\tau$, $\wh P$ is elliptic; the space is independent of the choice of $Q_0$ satisfying the criteria since at least one of $Q_\pm$, $\tau$, $\wh P$ is elliptic on $U_0\setminus\Gamma$. Moreover, every operator in $\Psib^k(M)$ defines a continuous map $\bnormiso^s(M)\to\bnormiso^{s-k}(M)$ as for $A\in\Psib^k(M)$, $Q_+Au=AQ_+u+[Q_+,A]u$ and $[Q_+,A]\in\Psib^{k-1}(M)$; the analogous statement also holds for the dual spaces.

The non-trapping estimates then are:

\begin{prop}\label{prop:b-normiso-propagation}
(See \cite[Theorem~3]{HintzVasyNormHyp}.) With $\cP,\bnormiso^s,\bnormiso^{*,s}$ as above, for any neighborhood $U$ of $\Gamma$ and for any $N$ there exist $B_0\in\Psib^0(M)$ elliptic at $\Gamma$ and $B_1,B_2\in\Psib^0(M)$ with $\WFb'(B_j)\subset U$, $j=0,1,2$, $\WFb'(B_2)\cap\Gamma_+=\emptyset$ and $C>0$ such that
\begin{equation}\label{eq:b-normiso}
\|B_0 u\|_{\bnormiso^s}\leq \|B_1\cP u\|_{\bnormiso^{*,s-m+1}}+\|B_2u\|_{\Hb^s}+C\|u\|_{\Hb^{-N}},
\end{equation} 
i.e.\ if all the functions on the right hand side are in the indicated spaces: $B_1\cP u\in \bnormiso^{*,s-m+1}$, etc., then $B_0 u\in\bnormiso^s$, and the inequality holds.

The same conclusion also holds if we assume $\WFb'(B_2)\cap\Gamma_-=\emptyset$ instead of $\WFb'(B_2)\cap\Gamma_+=\emptyset$.

Finally, if $r<0$, then, with $\WFb'(B_2)\cap\Gamma_+=\emptyset$,
\eqref{eq:b-normiso} becomes
\begin{equation}\label{eq:b-normiso-offspect}
\|B_0 u\|_{\Hb^{s,r}}\leq \|B_1\cP u\|_{\Hb^{s-m+1,r}}+\|B_2u\|_{\Hb^{s,r}}+C\|u\|_{\Hb^{-N,r}},
\end{equation}
while if $r>0$, then, with $\WFb'(B_2)\cap\Gamma_-=\emptyset$,
\begin{equation}\label{eq:b-normiso-offspect-reverse}
\|B_0 u\|_{\Hb^{s,r}}\leq \|B_1\cP u\|_{\Hb^{s-m+1,r}}+\|B_2u\|_{\Hb^{s,r}}+C\|u\|_{\Hb^{-N,r}},
\end{equation}
\end{prop}

\begin{rmk}
Note that the weighted versions \eqref{eq:b-normiso-offspect}-\eqref{eq:b-normiso-offspect-reverse} use {\em standard} weighted b-Sobolev spaces.
\end{rmk}

Next, if $\Omega\subset M$, as in \S\ref{SecStaticDeSitter}, is such that $\Sb^*_{H_j}{\Omega}\cap\Gamma=\emptyset$, $j=1,2$, then spaces such as
$$
\bnormiso^{*,s}(\Omega)^{\bullet,-}
$$
are not only well-defined, but are standard $\Hb^s$-spaces near the $H_j$. The inclusions analogous to \eqref{eq:b-Sobolev-normiso} also hold for the corresponding spaces over $\Omega$.

Notice that elements of $\Psib^p(M)$ only map $\bnormiso^s(M)$ to $\bnormiso^{*,s-p-1}(M)$, with the issues being at $\Gamma$ corresponding to \eqref{eq:b-Sobolev-normiso} (thus there is no distinction between the behavior on the $\Omega$ vs.\ the $M$-based spaces). However, if $A\in\Psib^p(M)$ has principal symbol vanishing on $\Gamma$ then
\begin{equation}\label{eq:bnormiso-mapping}
A:\bnormiso^s(M)\to\Hb^{s-p}(M),\ A:\Hb^s(M)\to\bnormiso^{*,s-p}(M),
\end{equation}
as $A$ can be expressed as $A_+Q_++A_- Q_-+A_\pa\tau+\wh A\wh P+A_0Q_0+R$, $A_\pm,A_0,A_\pa,\wh A\in\Psib^0(M)$, $R\in\Psib^{-1}(M)$, with the second mapping property following by duality as $\Psib^p(M)$ is closed under adjoints, and the principal symbol of the adjoint vanishes wherever that of the original operator does. Correspondingly, if $A_j\in\Psib^{m_j}(M)$, $j=1,2$, have principal symbol vanishing at $\Gamma$ then $A_1A_2 u:\bnormiso^s(M)\to\bnormiso^{*,s-m_1-m_2}(M)$.

We consider $\cP$ as a map
$$
\cP\colon\bnormiso^{s}(\Omega)^{\bullet,-}\to \bnormiso^{s-2}(\Omega)^{\bullet,-},
$$
and let
$$
\cY^{s}_{\Gamma}=\bnormiso^{*,s}(\Omega)^{\bullet,-},\
\cX^{s}_{\Gamma}=\{u\in\bnormiso^s(\Omega)^{\bullet,-}:\ \cP u\in \cY^{s-1}_{\Gamma}\}.
$$
While $\cX^s_\Gamma$ is complete,\footnote{
Also, elements of $\CI(\Omega)$ vanishing to infinite order at
$H_1$ and $X\cap\Omega$ are dense in $\cX^s_\Gamma$. Indeed, in view of
\cite[Lemma~A.3]{Melrose-Vasy-Wunsch:Corners} the only possible
issue is at $\Gamma$, thus the distinction between $\Omega$ and
$M$ may be dropped.
To complete the argument, one proceeds as in the quoted lemma,
using the ellipticity of $\sigma$ at
$\Gamma$,
letting $\Lambda_n \in\Psib^{-\infty}(M)$, $n\in\NN$, be a quantization
of $\phi(\sigma/n)a$, $a\in\CI(\Sb^*M)$ supported in a neighborhood of
$\Gamma$, identically $1$ near $\Gamma$, $\phi\in\CI_c(\RR)$, noting that
$[\Lambda_n,\cP]\in\Psib^{-\infty}(M)$ is uniformly bounded in
$\Psib^0(M)+\tau\Psib^1(M)$ in view of \eqref{eq:b-Ham-vf}, and thus for $u\in\cX^s_\Gamma$,
$\cP\Lambda_n u=\Lambda_n\cP u+[\cP,\Lambda_n]u\to \cP u$ in
$\bnormiso^{*,s-1}$ since $[\cP,\Lambda_n]$ is uniformly bounded
$\Hb^{s-1/2}\cap\Hb^{s,-1/2}\to\Hb^{s-1/2}\cap\Hb^{s-1,1/2}$, and thus
$\bnormiso^s\to\bnormiso^{*,s-1}$ by \eqref{eq:b-Sobolev-normiso}.}
it is a slightly
exotic space, unlike $\cX^s$ in
Theorem~\ref{ThmPFredholm} which is a coisotropic space depending on
$\Sigma$ (and thus the principal symbol of $\cP$) only, since elements of
$\Psib^p(M)$ only map $\bnormiso^s(M)$ to $\bnormiso^{*,s-p-1}(M)$ as
remarked earlier.
Correspondingly, $\cX^{s}_\Gamma$ actually depends on $\cP$
modulo $\Psib^0(M)$ plus first order pseudodifferential operators of the
form $A_1A_2$, $A_1\in\Psib^0(M)$, $A_2\in\Psib^1(M)$, both with
principal symbol vanishing at $\Gamma$ -- here the operators should
have Schwartz kernels supported away from the $H_j$; near $H_j$ (but
away from $\Gamma$), one
should say $\cP$ matters modulo $\Diffb^1(M)$, i.e.\ only the
principal symbol of $\cP$ matters.

We then have:

\begin{thm}
\label{ThmPFredholm-norm-hyp}
Suppose $s\geq 3/2$, and that the inverse of the Mellin transformed normal operator
$\widehat{\cP}(\sigma)^{-1}$ has no poles with $\im\sigma\geq 0$.
Then
$$
\cP\colon\cX^{s}_\Gamma\to\cY^{s-1}_\Gamma
$$
is invertible,
giving the forward solution operator.
\end{thm}

\begin{proof}
First, with $r<-1/2$, thus with dual spaces having weight $\wt r>1/2$, Theorem~\ref{ThmPFredholm} holds without changes as Proposition~\ref{prop:b-normiso-propagation} gives non-trapping estimates in this case on the standard b-Sobolev spaces.  In particular, if $r\ll 0$, $\Ker\cP$ is trivial even on $\Hb^{s-1/2,r}(\Omega)^{\bullet,-}$, hence certainly on its subspace $\bnormiso^{s}(\Omega)^{\bullet,-}$. Similarly, $\Ker\cP^*$ is trivial on $\Hb^{s,\wt r}(\Omega)^{-,\bullet}$, $\wt r\gg 0$, and thus with $r<-1/2$, for $f\in\Hb^{-1,r}(\Omega)^{\bullet,-}$ there exists $u\in\Hb^{0,r}(\Omega)^{\bullet,-}$ with $\cP u=f$. Further, making use of the non-trapping estimates in Proposition~\ref{prop:b-normiso-propagation}, if $r<0$ and $f\in\Hb^{s-1,r}(\Omega)^{\bullet,-}$, then the argument of Theorem~\ref{thm:b-main} improves this statement to $u\in\Hb^{s,r}(\Omega)^{\bullet,-}$.

In particular, if $f\in\bnormiso^{*,s-1}(\Omega)^{\bullet,-}\subset \Hb^{s-1,0}(\Omega)^{\bullet,-}$, then
$u\in\Hb^{s,r}(\Omega)^{\bullet,-}$ for $r<0$. This can be
improved using the argument of Theorem~\ref{thm:b-main}.
Indeed, with $-1\leq r<0$ arbitrary,
$\cP-N(\cP)\in\tau\Diffb^2(M)$ implies as in \eqref{eq:normal-op-perturb-exp-ds} that
\begin{equation}\label{eq:normal-op-perturb-exp}
N(\cP)u=f-\wt f,\ \wt f=(\cP-N(\cP))u\in\Hb^{s-2,r+1}(\Omega)^{\bullet,-}.
\end{equation}
But $f\in \bnormiso^{*,s-1}(\Omega)^{\bullet,-}\subset\Hb^{s-1,0}(\Omega)^{\bullet,-}$, hence the right hand side is in $\Hb^{s-2,0}(\Omega)^{\bullet,-}$; thus the dilation-invariant result, \cite[Lemma~3.1]{Va12}, gives
$u\in\Hb^{s-1,0}(\Omega)^{\bullet,-}$. This can then be improved
further since in view of
$\cP u=f\in\bnormiso^{*,s-1}(\Omega)^{\bullet,-}$, propagation of
singularities, most crucially
Proposition~\ref{prop:b-normiso-propagation}, yields
$u\in\bnormiso^{s}(\Omega)^{\bullet,-}$.
This completes the proof of the theorem.
\end{proof}

This result shows the importance of controlling the resonances in
$\im\sigma\geq 0$.
For the wave operator on exact Kerr-de Sitter space, Dyatlov's
analysis \cite{Dy11a,Dy11b} shows that the zero resonance of $\Box_g$ is the only
one in $\Im\sigma\geq 0$, the residue at $0$ having constant functions
as its range. For the Klein-Gordon operator $\Box_g-m^2$, the
statement is even better from our perspective as there are no
resonances in $\im\sigma\geq 0$ for $m>0$ small. This is pointed out
in \cite{Dy11a}; we give a direct proof based on perturbation theory.

\begin{lemma}
\label{LemmaPerturb}
  Let $\cP=\Box_g$ on exact Kerr-de Sitter space. Then for small $m>0$, all poles of $(\widehat\cP(\sigma)-m^2)^{-1}$ have strictly negative imaginary part.
\end{lemma}
\begin{proof}
  By perturbation theory, the inverse family of $\widehat\cP(\sigma)-\lambda$ has a simple pole at $\sigma(\lambda)$ coming with a single resonant state $\phi(\lambda)$ and a dual state $\psi(\lambda)$, with analytic dependence on $\lambda$, where $\sigma(0)=0,\phi(0)\equiv 1,\psi(0)=1_{\{\mu>0\}}$, where we use the notation of \cite[\S6]{Va12}. Differentiating $\widehat\cP(\sigma(\lambda))\phi(\lambda)=\lambda\phi(\lambda)$ with respect to $\lambda$ and evaluating at $\lambda=0$ gives
  \[
    \sigma'(0)\widehat\cP'(0)\phi(0)+\widehat\cP(0)\phi'(0)=\phi(0).
  \]
  Pairing this with $\psi(0)$, which is orthogonal to $\Ran\widehat\cP(0)$, yields
  \[
    \sigma'(0)=\frac{\la\psi(0),\phi(0)\ra}{\la\psi(0),\widehat\cP'(0)\phi(0)\ra},
  \]
  Since $\phi(0)=1$ and $\psi(0)=1_{\{\mu>0\}}$, this implies
  \begin{equation}
  \label{EqMovingResonance}\sgn\Im\sigma'(0)=-\sgn\Im\la\psi(0),\widehat\cP'(0)\phi(0)\ra.
  \end{equation}
  To find the latter quantity, we note that the only terms in the
  general form of the d'Alembertian that could possibly yield a
  non-zero contribution here are terms involving $\tau D_\tau$ and
  either $D_r$, $D_\phi$ or $D_\theta$. Concretely, using the explicit
  form of the dual metric $G$, see~Equation (6.1) in \cite{Va12}, in the new coordinates $t=\wt t+h(r),\phi=\wt\phi+P(r),\tau=e^{-t}$, with $h(r),P(r)$ as in \cite[Equation~(6.5)]{Va12},
  \begin{align*}
    G=-\rho^{-2}\biggl(&\wt\mu(\pa_r-h'(r)\tau\pa_\tau+P'(r)\pa_\phi)^2+\frac{(1+\gamma)^2}{\varkappa\sin^2\theta}(-a\sin^2\theta \tau\pa_\tau+\pa_\phi)^2+\varkappa\pa_\theta^2 \\
	&-\frac{(1+\gamma)^2}{\wt\mu}(-(r^2+a^2)\tau\pa_\tau+a\pa_\phi)^2\biggr),
  \end{align*}
  and its determinant $|\det G|^{1/2}=(1+\gamma)^2\rho^{-2}(\sin\theta)^{-1}$, we see that the only non-zero contribution to the right hand side of \eqref{EqMovingResonance} comes from the term
  \begin{align*}
    (1+\gamma)^2&\rho^{-2}(\sin\theta)^{-1}D_r\bigl((1+\gamma)^{-2}\rho^2\sin\theta \rho^{-2}\wt\mu h'(r)\bigr)\tau D_\tau \\
	  &=-i\rho^{-2}\pa_r(\wt\mu h'(r))\tau D_\tau
  \end{align*}
  of the d'Alembertian. Mellin transforming this amounts to replacing $\tau D_\tau$ by $\sigma$; then differentiating the result with respect to $\sigma$ gives
  \begin{align}
    \la&\psi(0),\widehat\cP'(0)\phi(0)\ra=-i\int_{\wt\mu>0}\rho^{-2}\pa_r(\wt\mu h'(r))\,d\mathrm{vol} \nonumber\\
	  &=-i\int_0^\pi\!\!\int_0^{2\pi}\!\!\int_{r_-}^{r_+}(1+\gamma)^{-2}\sin\theta\,\pa_r(\wt\mu h'(r))\,dr\,d\phi\,d\theta \nonumber\\
	\label{EqMovingResonance2} &=-\frac{4\pi i}{(1+\gamma)^2}[(\wt\mu h'(r))|_{r_+}-(\wt\mu h'(r))|_{r_-}].
  \end{align}
  Since the singular part of $h'(r)$ at $r_\pm$ (which are the roots of $\wt\mu$) is $h'(r)=\mp\frac{1+\gamma}{\wt\mu}(r^2+a^2)$, the right hand side of \eqref{EqMovingResonance2} is positive up to a factor of $i$; thus $\Im\sigma'(0)<0$ as claimed.
\end{proof}

In other words, for small mass $m>0$, there are no resonances $\sigma$
of the Klein-Gordon operator with $\Im\sigma\geq-\eps_0$ for some
$\eps_0>0$. Therefore, the expansion of $u$ as in
\eqref{EqAsympExpansion} no longer has a constant
term. Correspondingly, for $\eps\in\RR$, $\eps\leq\eps_0$,
Theorem~\ref{thm:KdS-main}
gives the forward solution operator
\begin{equation}\label{eq:varkappa-KG}
  S_{\KG,\I}\colon \Hb^{s-1+\varkappa,\eps}(\Omega)^{\bullet,-}\to \Hb^{s,\eps}(\Omega)^{\bullet,-}
\end{equation}
in the dilation-invariant case.

Further, Theorem~\ref{ThmPFredholm-norm-hyp} is applicable and gives the forward solution operator
\begin{equation}\label{eq:non-trapping-KG}
  S_\KG\colon \bnormiso^{*,s-1}(\Omega)^{\bullet,-}\to\bnormiso^s(\Omega)^{\bullet,-}
\end{equation}
on the normally isotropic spaces.

For the semilinear application, for non-linearities without derivatives, it is important that the loss of
derivatives $\varkappa$ in the space $\Hb^{s-1+\varkappa,\eps}$ is $\leq 1$. This is not explicitly specified in
the paper of Wunsch and Zworski \cite{Wu11}, though their proof
directly (see especially the part before Section 4.4 of \cite{Wu11}) gives that, for small $\eps>0$, $\varkappa$ can be taken
proportional to $\eps$, and there is $\eps'_0>0$ such
that $\varkappa\in(0,1]$ for $\eps<\eps'_0$. We reduce $\eps_0>0$
above if needed so that $\eps_0\leq\eps_0'$; then
\eqref{eq:varkappa-KG} holds with $\varkappa=c\eps\in(0,1]$ if
$\eps<\eps_0$, where $c>0$.

In fact, one does not need to go through the proof of \cite{Wu11}, for the
Phragm\'en-Lindel\"of theorem allows one to obtain the same conclusion
from their final result:

\begin{lemma}
\label{LemmaPhragmen}
  Suppose $h\colon U\to E$ is a holomorphic function on the half strip $U=\{z\in\C \colon 0\leq\Im z\leq c,\Re z\geq 1\}$ which is continuous on $\ol U$, with values in a Banach space $E$, and suppose moreover that there are constants $A,C>0$ such that
  \begin{gather*}
    \|h(z)\|\leq C|z|^{k_1}, \quad\Im z=0, \\
	\|h(z)\|\leq C|z|^{k_2}, \quad\Im z=c, \\
	\|h(z)\|\leq C\exp(A|z|), \quad z\in\ol U.
  \end{gather*}
  Then there is a constant $C'>0$ such that
  \[
    \|h(z)\|\leq C'|z|^{k_1\left(1-\frac{\Im z}{c}\right)+k_2\frac{\Im z}{c}}
  \]
  for all $z\in\ol U$.
\end{lemma}
\begin{proof}
  Consider the function $f(z)=z^{k_1-i\frac{k_2-k_1}{c}z}$, which is holomorphic on a neighborhood of $\ol U$. Writing $z\in\ol U$ as $z=x+iy$ with $x,y\in\R$, one has
  \begin{align*}
    |f(z)|&=|z|^{k_1}\exp\left(\Im\left(\frac{k_2-k_1}{c}z\log z\right)\right) \\
	  &=|z|^{k_1}|z|^{\frac{k_2-k_1}{c}\Im z}\exp\left(\frac{k_2-k_1}{c}x\arctan(y/x)\right).
  \end{align*}
  Noting that $|x\arctan(y/x)|=y|(x/y)\arctan(y/x)|$ is bounded by $c$ for all $x+iy\in\ol U$, we conclude that
  \[
    e^{-|k_2-k_1|}|z|^{k_1\left(1-\frac{\Im z}{c}\right)+k_2\frac{\Im z}{c}}\leq|f(z)|\leq e^{|k_2-k_1|}|z|^{k_1\left(1-\frac{\Im z}{c}\right)+k_2\frac{\Im z}{c}}.
  \]
  Therefore, $f(z)^{-1}h(z)$ is bounded by a constant $C'$ on $\partial\ol U$, and satisfies an exponential bound for $z\in U$. By the Phragm\'en-Lindel\"of theorem, $\|f(z)^{-1}h(z)\|_E\leq C'$, and the claim follows.\qedhere
\end{proof}

Since for any $\delta>0$, we can bound $|\log z|\leq C_\delta|z|^\delta$ for $|\Re z|\geq 1$, we obtain that the inverse family $R(\sigma)=\widehat\cP(\sigma)^{-1}$ of the normal operator of $\Box_g$ on (asymptotically) Kerr-de Sitter spaces as in \cite{Va12}, here in the setting of artificial boundaries as opposed to complex absorption, satisfies a bound
 \begin{equation}
   \|R(\sigma)\|_{|\sigma|^{-(s-1)}H^{s-1}_{|\sigma|^{-1}}(X\cap\Omega)\to|\sigma|^{-s}H^s_{|\sigma|^{-1}}(X\cap\Omega)}\leq C_\delta|\sigma|^{-1+\varkappa'+\delta}
  \label{eq-estimate-near-r}
 \end{equation}
 for any $\delta>0$, $\Im\sigma\geq-c\varkappa'$ and $|\Re\sigma|$ large. Therefore, as mentioned above, by the proof of Theorem~\ref{thm:KdS-main}, i.e.\ \cite[Theorem~1.4]{Va12}, in particular using \cite[Lemma~3.1]{Va12}, we can assume $\varkappa\in(0,1]$ in the dilation-invariant result, Theorem~\ref{thm:KdS-main}, if we take $C'>0$ small enough, i.e.\ if we do not go too far into the lower half plane $\Im\sigma<0$, which amounts to only taking terms in the expansion \eqref{EqAsympExpansion} which decay to at most some fixed order, which we may assume to be less than $-\Im\sigma_j$ for all resonances $\sigma_j$.

\subsection{A class of semilinear equations; equations with polynomial non-linearity}
\label{SubsecKdSSemi}

In the following semilinear applications, let us fix $\varkappa\in(0,1]$ and $\eps_0$ as explained before Lemma~\ref{LemmaPhragmen}, so that we have the forward solution operator $S_{\KG,\I}$ as in \eqref{eq:varkappa-KG}.

We then have statements paralleling Theorems~\ref{ThmDSQu}, \ref{ThmDSPoly} and Corollary~\ref{cor:ndS}, namely Theorems~\ref{ThmKdSQu}, \ref{ThmKdSPoly} and Corollary~\ref{cor:nKdS}, respectively.

\begin{thm}
\label{ThmKdSQu}
 Suppose $(M,g)$ is dilation-invariant. Let $-\infty<\eps<\eps_0,s>1/2+\beta\eps$, $s\geq 1$, and let $q\colon\Hb^{s,\eps}(\Omega)^{\bullet,-}\to \Hb^{s-1+\varkappa,\eps}(\Omega)^{\bullet,-}$ be a continuous function with $q(0)=0$ such that there exists a continuous non-decreasing function $L\colon\R_{\geq 0}\to\R$ satisfying
  \[
    \|q(u)-q(v)\|\leq L(R)\|u-v\|,\quad \|u\|,\|v\|\leq R.
  \]
  Then there is a constant $C_L>0$ so that the following holds: If $L(0)<C_L$, then for small $R>0$, there exists $C>0$ such that for all $f\in\Hb^{s-1+\varkappa,\eps}(\Omega)^{\bullet,-}$ with norm $\leq C$, the equation
  \[
    (\Box_g-m^2)u=f+q(u)
  \]
  has a unique solution $u\in\Hb^{s,\eps}(\Omega)^{\bullet,-}$, with norm $\leq R$, that depends continuously on $f$.

  More generally, suppose
  \[
    q\colon\Hb^{s,\eps}(\Omega)^{\bullet,-}\times\Hb^{s-1+\varkappa,\eps}(\Omega)^{\bullet,-}\to\Hb^{s-1+\varkappa,\eps}(\Omega)^{\bullet,-}
  \]
  satisfies $q(0,0)=0$ and
  \[
    \|q(u,w)-q(u',w')\|\leq L(R)(\|u-u'\|+\|w-w'\|)
  \]
  provided $\|u\|+\|w\|,\|u'\|+\|w'\|\leq R$, where we use the norms corresponding to the map $q$, for a continuous non-decreasing function $L\colon\R_{\geq 0}\to\R$. Then there is a constant $C_L>0$ so that the following holds: If $L(0)<C_L$, then for small $R>0$, there exists $C>0$ such that for all $f\in\Hb^{s-1+\varkappa,\eps}(\Omega)^{\bullet,-}$ with norm $\leq C$, the equation
  \[
    (\Box_g-m^2)u=f+q(u,\Box_g u)
  \]
  has a unique solution $u\in\Hb^{s,\eps}(\Omega)^{\bullet,-}$, with $\|u\|_{\Hb^{s,\eps}}+\|\Box_g u\|_{\Hb^{s-1+\varkappa,\eps}}\leq R$, that depends continuously on $f$.
\end{thm}
\begin{proof}
  We use the proof of the first part of Theorem~\ref{ThmDSQu}, where in the current setting the solution operator $S_{\KG,\I}$ maps $\Hb^{s-1+\varkappa,\eps}(\Omega)^{\bullet,-}\to\Hb^{s,\eps}(\Omega)^{\bullet,-}$, and the contraction map is $T\colon\Hb^{s,\eps}(\Omega)^{\bullet,-}\to\Hb^{s,\eps}(\Omega)^{\bullet,-}$, $Tu=S_{\KG,\I}(f+q(u))$.

  For the general statement, we follow the proof of the second part of Theorem~\ref{ThmDSQu}, where we now instead use the Banach space
  \[
    \cX=\{u\in\Hb^{s,\eps}(\Omega)^{\bullet,-}\colon \Box_g u\in\Hb^{s-1+\varkappa,\eps}(\Omega)^{\bullet,-}\}
  \]
  with norm
  \[
    \|u\|_{\cX}=\|u\|_{\Hb^{s,\eps}}+\|\Box_g u\|_{\tau^\eps \Hb^{s-1+\varkappa}}.
  \]
  which is a Banach space by the same argument as in the proof of Theorem~\ref{ThmDSQu}.
\end{proof}

We have a weaker statement in the general, non dilation-invariant case, where we work in unweighted spaces.

\begin{thm}
\label{ThmKdSQuNoninv}
  Let $s\geq 1$, and suppose $q\colon\Hb^s(\Omega)^{\bullet,-}\to\Hb^s(\Omega)^{\bullet,-}$ is a continuous function with $q(0)=0$ such that there exists a continuous non-decreasing function $L\colon\R_{\geq 0}\to\R$ satisfying
  \[
    \|q(u)-q(v)\|\leq L(R)\|u-v\|,\quad \|u\|,\|v\|\leq R.
  \]
  Then there is a constant $C_L>0$ so that the following holds: If $L(0)<C_L$, then for small $R>0$, there exists $C>0$ such that for all $f\in\Hb^s(\Omega)^{\bullet,-}$ with norm $\leq C$, the equation
  \[
    (\Box_g-m^2)u=f+q(u)
  \]
  has a unique solution $u\in\Hb^s(\Omega)^{\bullet,-}$, with norm $\leq R$, that depends continuously on $f$.

  An analogous statement holds for non-linearities $q=q(u,\Box_g u)$ which are continuous maps $q\colon\Hb^s(\Omega)^{\bullet,-}\times\Hb^s(\Omega)^{\bullet,-}\to\Hb^s(\Omega)^{\bullet,-}$, vanish at $(0,0)$ and have a small Lipschitz constant near $0$.
\end{thm}
\begin{proof}
  Since
  \[
    S_\KG\colon \Hb^s(\Omega)^{\bullet,-}\subset\bnormiso^{*,s-1/2}(\Omega)^{\bullet,-}\to\bnormiso^{s+1/2}(\Omega)^{\bullet,-}\subset\Hb^s(\Omega)^{\bullet,-},
  \]
  by \eqref{eq:b-Sobolev-normiso} and \eqref{eq:non-trapping-KG}, this follows again from the Banach fixed point theorem.
\end{proof}

\begin{rmk}
  The proof of Theorem~\ref{ThmPFredholm-norm-hyp} shows that equations on function spaces with negative weights (i.e.\ growing near infinity) behave as nicely as equations on the static part of asymptotically de Sitter spaces, discussed in \S\ref{SecStaticDeSitter}. However, naturally occurring non-linearities (e.g., polynomials) will not be continuous non-linear operators on such growing spaces.
\end{rmk}

\begin{cor}\label{cor:nKdS}
If $s>n/2$, the hypotheses of Theorem~\ref{ThmKdSQuNoninv}
hold for non-linearities $q(u)=cu^p$, $p\geq 2$ integer, $c\in\Cx$, as
well as $q(u)=q_0 u^p$, $q_0\in\Hb^s(M)$.

Thus for small $m>0$ and $R>0$, there exists $C>0$ such that for all $f\in\Hb^s(\Omega)^{\bullet,-}$ with norm $\leq C$, the equation
  \[
    (\Box_g-m^2)u=f+q(u)
  \]
  has a unique solution $u\in\Hb^s(\Omega)^{\bullet,-}$, with norm $\leq R$, that depends continuously on $f$.
\end{cor}

If $f$ satisfies stronger decay assumptions, then $u$ does as well. More precisely, denoting the inverse family of the normal operator of the Klein-Gordon operator with (small) mass $m$ by $R_m(\sigma)=(\widehat\cP(\sigma)-m^2)^{-1}$, which has poles only in $\Im\sigma<0$ (cf.\ Lemma~\ref{LemmaPerturb} and \cite{Dy11a,Va12}), and moreover defining the spaces $\cX_{\sF}^{s,r,\eps}$ and $\cX_{\sP,\sF}^{s,r,\eps}$ analogously to the corresponding spaces in \S\ref{SubsecDSPoly}, we have the following result:

\begin{thm}
\label{ThmKdSPoly}
 Fix $0<\eps<\min\{C',1/2\}$ and let $s\gg s'\geq\max(1/2+\beta\eps,n/2,1+\varkappa)$. (A concrete bound for $s$ will be given in the course of the proof, see equation~\ref{EqSBound}.) Let
 \[
   q(u)=\sum_{p=2}^d q_pu^p,\quad q_p\in\Hb^s(M).
 \]
 Moreover, if $\sigma_j\in\C$ are the poles of the inverse family $R_m(\sigma)$, and $m_j+1$ is the order of the pole of $R_m(\sigma)$ at $\sigma_j$, let $\sP=\{(i\sigma_j+k,\ell)\colon 0\leq\ell\leq m_j,k\in\N_0\}$. Assume that $\eps\neq\Re(i\sigma_j)$ for all $j$, and that $m>0$ is so small that $\sP$ is a positive index set. Finally, let $\sF$ be a positive index set.
 
 Then for small enough $R>0$, there exists $C>0$ such that for all $f\in\cX_{\sF}^{s,s,\eps}$ with norm $\leq C$, the equation
  \begin{equation}
  \label{EqKdSPoly}
    (\Box_g-m^2)u=f+q(u)
  \end{equation}
  has a unique solution $u\in\cX_{\sP,\sF}^{s',s',\eps}$, with norm $\leq R$, that depends continuously on $f$; in particular, $u$ has an asymptotic expansion with remainder in $\Hb^{s',\eps}(\Omega)^{\bullet,-}$.
\end{thm}
\begin{proof}
  Let us write $\cP=\Box_g-m^2$. Let $\delta<1/2$ be such that $0<2\delta<\Re z$ for all $(z,0)\in\sF$, then $f\in\Hb^{s,2\delta}(\Omega)^{\bullet,-}$. Now, for $u\in\Hb^{s,\delta}(\Omega)^{\bullet,-}$, consider $Tu:=S_\KG(f+q(u))$. First of all, $f+q(u)\in\Hb^{s,2\delta}(\Omega)^{\bullet,-}\subset\Hb^s(\Omega)^{\bullet,-}$, thus the proof of Theorem~\ref{ThmPFredholm-norm-hyp} shows that we have $Tu\in\Hb^{s+1,r}(\Omega)^{\bullet,-}$, $r<0$ arbitrary. Therefore,
  \[
    N(\cP)u=f+q(u)+(N(\cP)-\cP)u\in\Hb^{s,2\delta}(\Omega)^{\bullet,-}+\Hb^{s-1,r+1}(\Omega)^{\bullet,-}\subset\Hb^{s-1,2\delta}(\Omega)^{\bullet,-},
  \]
  and thus if $\delta>0$ is sufficiently small, namely, $\delta<\inf\{-\Im\sigma_j\}/2$, Theorem~\ref{thm:KdS-main} implies $u\in\Hb^{s-\varkappa,2\delta}(\Omega)^{\bullet,-}$. Since we can choose $\varkappa=c\delta$ for some constant $c>0$, we obtain
  \[
    Tu\in\bigcap_{r>0}\Hb^{s+1,r}(\Omega)^{\bullet,-}\cap\Hb^{s-c\delta,2\delta}(\Omega)^{\bullet,-}\subset\bigcap_{r'>0}\Hb^{s,2\delta-2c\delta^2/(1+c\delta)-r'}(\Omega)^{\bullet,-}
  \]
  by interpolation. In particular, choosing $\delta>0$ even smaller if necessary, we obtain $Tu\in\Hb^{s,\delta}(\Omega)^{\bullet,-}$. Applying the Banach fixed point theorem to the map $T$ thus gives a solution $u\in\Hb^{s,\delta}(\Omega)^{\bullet,-}$ to the equation~\eqref{EqKdSPoly}.

  For this solution $u$, we obtain
  \[
    N(\cP)u=\cP u+(N(\cP)-\cP)u\in\Hb^{s,2\delta}+\Hb^{s-2,\delta+1}\subset\Hb^{s-2,2\delta}
  \]
  since $q$ only has quadratic and higher terms. Hence Theorem~\ref{thm:KdS-main} implies that $u=u_1+u'$, where $u_1$ is an expansion with terms coming from poles of $\widehat\cP^{-1}$ whose decay order lies between $\delta$ and $2\delta$, and $u'\in\Hb^{s-1-\varkappa,2\delta}(\Omega)^{\bullet,-}$. This in turn implies that $f+q(u)$ has an expansion with remainder term in $\Hb^{s-1-\varkappa,\min\{4\delta,\eps\}}(\Omega)^{\bullet,-}$, thus
  \[
    N(\cP)u\in\Hb^{s-3-\varkappa,\min\{4\delta,\eps\}}(\Omega)^{\bullet,-}\tn{ plus an expansion},
  \]
  and we proceed iteratively, until, after $k$ more steps, we have $4\cdot 2^k\delta\geq\eps$, and then $u$ has an expansion with remainder term $\Hb^{s-3-2k-\varkappa,\eps}(\Omega)^{\bullet,-}$ provided we can apply Theorem~\ref{thm:KdS-main} in the iterative procedure, i.e.\ provided $s-3-2k-\varkappa=:s'>\max(1/2+\beta\eps,n/2,1+\varkappa)$. This is satisfied if
  \begin{equation}
  \label{EqSBound}
    s>\max(1/2+\beta\eps,n/2,1+\varkappa)+2\lceil\log_2(\eps/\delta)\rceil+\varkappa-1.\qedhere
  \end{equation}
\end{proof}

\subsection{Semilinear equations with derivatives in the non-linearities}\label{sec:KdS-derivs}
Theorem~\ref{ThmPFredholm-norm-hyp} allows one to solve even
semilinear equations with derivatives in some cases. For instance,
in the case of de Sitter-Schwarzschild space,
within $\Sigma\cap\Sb^*_XM$, $\Gamma$ is given by $r=r_c$,
$\sigma_1(D_r)=0$, where $r_c=\frac{3}{2}r_s$ is the radius of the photon
sphere, see e.g.\ \cite[\S6.4]{Va12}. Thus,
non-linear terms such as $(r-r_c)(\pa_r u)^2$ are allowed for
$s>\frac{n}{2}+1$ since
$\pa_r:\bnormiso^s(M)\to\Hb^{s-1}(M)$, with the latter space being an
algebra, while multiplication by $r-r_c$ maps this space to
$\bnormiso^{*,s-1}$ by \eqref{eq:bnormiso-mapping}. Thus, a
straightforward modification of Theorem~\ref{ThmKdSQuNoninv},
applying the fixed point theorem on the normally isotropic spaces
directly, gives well-posedness.

\section{Asymptotically de Sitter spaces: global approach}
\label{SecDeSitter}

We can approach the problem of solving non-linear wave equations on
global asymptotically de Sitter spaces in two ways: Either, we proceed
as in the previous two sections, first showing invertibility of the
linear operator on suitable spaces and then applying the contraction
mapping principle to solve the non-linear problem; or we use the
solvability results from \S\ref{SecStaticDeSitter} for backward
light cones from points at future conformal infinity and glue the
solutions on all these `static' parts together to obtain a global
solution. The first approach, which we will follow in
\S\S\ref{SubsecGlobalDSLinear}-\ref{SubsecSemiPolyDS}, has the
disadvantage that the conditions on the non-linearity that guarantee
the existence of solutions are quite restrictive, however in case the
conditions are met, one has good decay estimates for solutions. The
second approach on the other hand, detailed in
\S\ref{SubsecGlobalFromStatic}, allows many of the
non-linearities, suitably reinterpreted, that work on `static parts'
of asymptotically de Sitter spaces (i.e.\ backward light cones), but the decay estimates for solutions are quite weak relative to the decay of the forcing term because of the gluing process.

\subsection{The linear framework}
\label{SubsecGlobalDSLinear}

Let $g$ be the metric on an $n$-dimensional asymptotically de Sitter space $X$ with global time function $t$ \cite{Va07}. Then, following \cite[Section 4]{Va12}, the operator\footnote{$P_\sigma$ in our notation corresponds to $P_{\bar\sigma}^*$ in \cite{Va12}, the latter operator being the one for which one solves the forward problem.}
\begin{equation}
\label{EqPSigma}
  P_\sigma=\mu^{-1/2}\mu^{i\sigma/2-(n+1)/4}\biggl(\Box_g-\Bigl(\frac{n-1}{2}\Bigr)^2-\sigma^2\biggr)\mu^{-i\sigma/2+(n+1)/4}\mu^{-1/2}
\end{equation}
extends non-degenerately to an operator on a closed manifold $\wt X$
which contains the compactification $\overline X$ of the
asymptotically de Sitter space as a submanifold with boundary $Y$,
where $Y=Y_-\cup Y_+$ has two connected components, which we call the
boundary of $X$ at past, resp.\ future, infinity. The expression
`non-degenerately' here means that near $Y_\pm$, $P_\sigma$ fits into
the framework of \cite{Va12}. Here, $\mu=0$ is the defining function of $Y$, and $\mu>0$ is the interior of the asymptotically de Sitter space. Moreover, null-bicharacteristics of $P_\sigma$ tend to $Y_\pm$ as $t\to\pm\infty$.

Following Vasy \cite{Va13}, let us in fact assume that $\wt
X=\overline{C_-}\cup\overline X\cup\overline{C_+}$ is the union of the
compactifications of asymptotically de Sitter space $X$ and two
asymptotically hyperbolic caps $C_\pm$; one might need to take two
copies of $X$ to construct $\wt X$ as explained in \cite{Va13}. For
the purposes of the next statement we recall that variable order
Sobolev spaces $H^s(\wt X)$ were discussed in \cite[Section 1, Appendix]{Ba13}. Then $P_\sigma$ is the restriction to $X$ of an operator $\wt P_\sigma\in\Diff^2(\wt X)$, which is Fredholm as a map
\[
  \wt P_\sigma\colon \wt\cX^s\to \wt\cY^{s-1},\quad \wt\cX^s=\{u\in H^s\colon\wt P_\sigma u\in H^{s-1}\},\quad \wt\cY^{s-1}=H^{s-1},
\]
where $s\in\CI(S^*\wt X)$, monotone along the bicharacteristic flow,
is such that $s|_{N^*Y_-}>1/2-\Im\sigma$, $s|_{N^*Y_+}<1/2-\Im\sigma$,
and $s$ is constant near $S^*Y_\pm$. Note that the choice of signs here
  is opposite to the one in \cite{Va13}, since here we are going to
  construct the forward solution operator on $X$. 

Restricting our attention to $X$, we define the space $H^s(X)^{\bullet,-}$ to be the completion in $H^s(X)$ of the space of $\CI$ functions that vanish to infinite order at $Y_-$; thus the superscripts indicate that distributions in $H^s(X)^{\bullet,-}$ are supported distributions near $Y_-$ and extendible distributions near $Y_+$. Then, define the spaces
\[
  \cX^s=\{u\in H^s(X)^{\bullet,-}\colon P_\sigma u\in H^{s-1}(X)^{\bullet,-}\},\quad \cY^{s-1}=H^{s-1}(X)^{\bullet,-}.
\]

\begin{thm}
\label{ThmGlobalDS-lin}
  Fix $\sigma\in\C$ and $s\in\CI(S^*\overline X)$ as above. Then $P_\sigma\colon\cX^s\to\cY^{s-1}$ is invertible, and $P_\sigma^{-1}\colon H^{s-1}(X)^{\bullet,-}\to H^s(X)^{\bullet,-}$ is the forward solution operator of $P_\sigma$.
\end{thm}
\begin{proof}
  First, let us assume $\re\sigma\gg 0$ so semiclassical/large
  parameter estimates are applicable to $\wt P_\sigma$, and let $T_0\in\R$ be such
  that $s$ is constant in $\{t\leq T_0\}$. Then for any $T_1\leq T_0$,
  we can paste together microlocal energy estimates for $\wt P_\sigma$
  near $\overline{C_-}$ and standard energy estimates for the wave
  equation in $\{t\leq T_1\}$ away from $Y_-$ as in the derivation of
  Equation~(3.29) of \cite{Va12}, and thereby obtain
  \begin{equation}
  \label{EqGlobalDSE}
    \|u\|_{H^1(\{t\leq T_1\})}\lesssim \|\wt P_\sigma u\|_{H^0(\{t\leq T_1\})};
  \end{equation}
  thus, for $f\in\CI(\wt X)$, $\supp f\subset\{t\geq T_1\}$ implies $\supp\wt P_\sigma^{-1}f\subset\{t\geq T_1\}$. Choosing $\phi\in\CI_c(X)$ with support in $\{t\geq T_1\}$ and $\psi\in\CI(\wt X)$ with support in $\{t\leq T_1\}$, we therefore obtain $\psi\wt P_\sigma^{-1}\phi=0$. Since $\wt P_\sigma^{-1}$ is meromorphic, this continues to hold for all $\sigma\in\C$ such that $\Im\sigma>1/2-s$. Since $T_1\leq T_0$ is arbitrary, this, together with standard energy estimates on the asymptotically de Sitter space $X$, proves that $P_\sigma^{-1}$ propagates supports forward, provided $P_\sigma$ is invertible. Moreover, elements of $\ker\wt P_\sigma$ are supported in $\overline{C_+}$.
  
  The invertibility of $P_\sigma$ is a consequence of \cite[Lemma~8.3]{Ba13}, also see Footnote 15 there: Let $E\colon H^{s-1}(X)^{\bullet,-}\to H^{s-1}(\wt X)$ be a continuous extension operator that extends by $0$ in $\overline{C_-}$ and $R\colon H^s(\wt X)\to H^s(X)^{-,-}$ the restriction, then $R\circ\wt P_\sigma^{-1}\circ E$ does not have poles; and since
  \[
    \bigcup_{T_1\leq T_0} H^s(\{t>T_1\})^{\bullet,-} \subset H^s(X)^{\bullet,-}
  \]
  (where ($\bullet$) denotes supported distributions at $\{t=T_1\}$, resp.\ $Y_-$) is dense, $R\circ\wt P_\sigma^{-1}\circ E$ in fact maps into $H^s(X)^{\bullet,-}$, thus $P_\sigma^{-1}=R\circ\wt P_\sigma^{-1}\circ E$ indeed exists and has the claimed properties.
\end{proof}

In our quest for finding forward solutions of semilinear equations, we
restrict ourselves to a submanifold with boundary
$\Omega\subset\overline X$ containing and localized near future
infinity, so that we can work in fixed order Sobolev spaces; moreover,
it will be useful to measure the conormal regularity of solutions to
the linear equation at the conormal bundle of the boundary of $X$ at
future infinity more precisely.
So let $H^{s,k}(\wt X,Y_+)$ be the subspace of $H^s(\wt X)$ with $k$-fold regularity with respect to the $\Psi^0(\wt X)$-module $\mc M$ of first order \psdo{}s with principal symbol vanishing on $N^*Y_+$.
A result of Haber and Vasy,
\cite[Theorem~6.3]{Haber-Vasy:Radial}, with $s_0=1/2-\im\sigma$ in our
case, shows that $f\in H^{s-1,k}(\wt
X,Y_+)$, $\wt P_\sigma u=f$, $u$ a distribution, in fact imply that
$u\in H^{s,k}(\wt X,Y_+)$.
So if we let $H^{s,k}(\Omega)^{\bullet,-}$ denote the space of all
$u\in H^s(X)^{\bullet,-}$ which are restrictions to $\Omega$ of
functions in $H^{s,k}(\wt X,Y_+)$, supported in
$\Omega\cup\overline{C_+}$, the argument of
Theorem~\ref{ThmGlobalDS-lin} shows that we have a forward solution
operator
$S_\sigma\colon H^{s-1,k}(\Omega)^{\bullet,-}\to H^{s,k}(\Omega)^{\bullet,-}$, provided
\begin{equation}
\label{EqGlobalDSForwardCond}
  s<1/2-\Im\sigma.
\end{equation}

\subsubsection{The backward problem}
\label{SubsubsecDSBackward}

Another problem that we will briefly consider below is the backward problem, i.e.\ where one solves the equation on $X$ backward from $Y_+$, which is the same, up to relabelling, as solving the equation forward from $Y_-$. Thus, we have a backward solution operator $S_\sigma^-\colon H^{s-1,k}(\Omega)^{-,\bullet}\to H^{s,k}(\Omega)^{-,\bullet}$ (where $\Omega$ is chosen as above so that we can use constant order Sobolev spaces), provided $s>1/2-\Im\sigma$. Similarly to the above, ($-$) denotes extendible distributions at $\pa\Omega\cap X^\circ$ and ($\bullet$) supported distributions at $Y_+$; the module regularity is measured at $Y_+$.

\subsection{Algebra properties of $H^{s,k}(\Omega)^{\bullet,-}$}
\label{SubsecDeSitterAlgebra}


Let us call a polynomially bounded measurable function
$w\colon\R^n\to(0,\infty)$ a \emph{weight function}. For such a weight
function $w$, we define
\[
  H^{(w)}(\R^n)=\{u\in S'(\R^n)\colon w\wh u\in L^2(\R^n)\}.
\]
The following lemma is similar in spirit to, but different from, Strichartz'
result on Sobolev algebras \cite{Strichartz:Sobolev}; it is the basis
for the multiplicative properties of the more delicate spaces
considered below.

\begin{lemma}
\label{LemmaProductsWeights}
  Let $w_1,w_2,w$ be weight functions such that one of the quantities
  \begin{equation}
  \label{EqWeightConditionGeneral}
    \begin{split}
      M_+&:=\sup_{\xi\in\R^n} \int\left(\frac{w(\xi)}{w_1(\eta)w_2(\xi-\eta)}\right)^2\,d\eta \\
 	  M_-&:=\sup_{\eta\in\R^n} \int\left(\frac{w(\xi)}{w_1(\eta)w_2(\xi-\eta)}\right)^2\,d\xi \\
    \end{split}
  \end{equation}
  is finite. Then $H^{(w_1)}(\R^n)\cdot H^{(w_2)}(\R^n)\subset H^{(w)}(\R^n)$.
\end{lemma}
\begin{proof}
  For $u,v\in S(\R^n)$, we use Cauchy-Schwarz to estimate
  \begin{align*}
    \|u&v\|_{H^{(w)}}^2=\int w(\xi)^2|\widehat{uv}(\xi)|^2\,d\xi \\
	 &=\int w(\xi)^2 \left(\int w_1(\eta)|\wh u(\eta)| w_2(\xi-\eta)|\wh v(\xi-\eta)| w_1(\eta)^{-1}w_2(\xi-\eta)^{-1}\,d\eta\right)^2\,d\xi \\
	 &\leq \int\left(\int\left(\frac{w(\xi)}{w_1(\eta)w_2(\xi-\eta)}\right)^2\,d\eta\right) \\
	 &\hspace{12ex}\times\left(\int w_1(\eta)^2|\wh u(\eta)|^2 w_2(\xi-\eta)^2|\wh v(\xi-\eta)|^2\,d\eta\right) d\xi \\
	 &\leq M_+\|u\|_{H^{(w_1)}}^2\|v\|_{H^{(w_2)}}^2
  \end{align*}
  as well as
  \begin{align*}
    \|u&v\|_{H^{(w)}}^2\leq\int\left(\int w_2(\xi-\eta)^2|\wh v(\xi-\eta)|^2\,d\eta\right) \\
	  &\hspace{12ex}\times\left(\int\left(\frac{w(\xi)}{w_1(\eta)w_2(\xi-\eta)}\right)^2 w_1(\eta)^2|\wh u(\eta)|^2\,d\eta\right)d\xi \\
	  &= \|v\|_{H^{(w_2)}}^2 \int w_1(\eta)^2|\wh u(\eta)|^2 \left(\int\left(\frac{w(\xi)}{w_1(\eta)w_2(\xi-\eta)}\right)^2\,d\xi\right)d\eta \\
	  &\leq M_-\|u\|_{H^{(w_1)}}^2\|v\|_{H^{(w_2)}}^2.
  \end{align*}
  Since $S(\R^n)$ is dense in $H^{(w_1)}(\R^n)$ and $H^{(w_2)}(\R^n)$, the lemma follows.
\end{proof}

In particular, if
\begin{equation}
\label{EqWeightCondition}\left\|\frac{w(\xi)}{w(\eta)w(\xi-\eta)}\right\|_{L^\infty_\xi L^2_\eta}<\infty,
\end{equation}
then $H^{(w)}$ is an algebra.

For example, the weight function $w(\xi)=\la\xi\ra^s$ for $s>n/2$
satisfies \eqref{EqWeightCondition} as we will check below, which implies that $H^s(\R^n)$ is
an algebra for $s>n/2$; this is the special case $k=0$ of
Lemma~\ref{LemmaYDerivatives} below, and is well-known, see e.g.\
\cite[Chapter~13.3]{Tay}.
Also, product-type weight functions $w_d(\xi)=\la\xi'\ra^s\la\xi''\ra^k$ (where $\xi=(\xi',\xi'')\in\R^{d+(n-d)}$) for $s>d/2,k>(n-d)/2$ satisfy \eqref{EqWeightCondition}.

The following lemma, together with the triangle inequality $\la\xi\ra^\alpha\lesssim\la\eta\ra^\alpha+\la\xi-\eta\ra^\alpha$ for $\alpha\geq 0$, will often be used to check conditions like \eqref{EqWeightConditionGeneral}.
\begin{lemma}
\label{LemmaFracEst}
  Suppose $\alpha,\beta\geq 0$ are such that $\alpha+\beta>n$. Then
  \[
    \int_{\R^n}\frac{d\eta}{\la\eta\ra^\alpha\la\xi-\eta\ra^\beta}\in L^\infty(\R^n_\xi).
  \]
\end{lemma}
\begin{proof}
  Splitting the domain of integration into the two regions $\{\la\eta\ra<\la\xi-\eta\ra\}$ and $\{\la\eta\ra\geq\la\xi-\eta\ra\}$, we obtain the bound
  \[
    \int_{\R^n}\frac{d\eta}{\la\eta\ra^\alpha\la\xi-\eta\ra^\beta}\leq 2\int_{\R^n}\frac{d\eta}{\la\eta\ra^{\alpha+\beta}},
  \]
  which is finite in view of $\alpha+\beta>n$.
\end{proof}

Another important consequence of Lemma~\ref{LemmaProductsWeights} is that $H^{s'}(\R^n)$ is an $H^s(\R^n)$-module provided $|s'|\leq s$, $s>n/2$, which follows for $s'\geq 0$ from $M_+<\infty$, and for $s'<0$ either by duality or from $M_-<\infty$ (with $M_\pm$ as in the statement of the lemma, with the corresponding weight functions).

\begin{lemma}
\label{LemmaYDerivatives}
  Write $x\in\R^n$ as $x=(x',x'')\in\R^{d+(n-d)}$. For $s\in\R,k\in\N_0$, let
  \[
    \mc Y^{s,k}_d(\R^n)=\{u\in H^s(\R^n)\colon D_{x''}^k u\in H^s(\R^n)\}.
  \]
  Then for $s>d/2,s+k>n/2$, $\mc Y^{s,k}_d(\R^n)$ is an algebra.
\end{lemma}
\begin{proof}
Using the Leibniz rule, we see that it suffices to show: If $u,v\in\mc
Y^{s,k}_d$, then $D_{x''}^\alpha u D_{x''}^\beta v\in H^s$, provided
$|\alpha|+|\beta|\leq k$. Since $D_{x''}^\alpha u\in\mc
Y^{s,k-|\alpha|}_d$ and $D_{x''}^\beta v\in\mc Y^{s,k-|\beta|}_d$,
this amounts to showing that
\begin{equation}\label{eq:Ysa-mult-prop}
\mc Y^{s,a}_d\cdot \mc Y^{s,b}_d\subset H^s\ \text{if}\ a+b\geq k.
\end{equation}
Using the characterization $\mc Y^{s,a}_d=H^{(w)}$ for $w(\xi)=\la\xi\ra^s\la\xi''\ra^k$, Lemma~\ref{LemmaProductsWeights} in turn reduces this to the estimate
\begin{align*}
  \int&\frac{\la\xi\ra^{2s}}{\la\eta\ra^{2s}\la\eta''\ra^{2a}\la\xi-\eta\ra^{2s}\la\xi''-\eta''\ra^{2b}}\,d\eta \\
   &\lesssim \int\frac{d\eta}{\la\eta''\ra^{2a}\la\xi-\eta\ra^{2s}\la\xi''-\eta''\ra^{2b}}+\int\frac{d\eta}{\la\eta\ra^{2s}\la\eta''\ra^{2a}\la\xi''-\eta''\ra^{2b}} \\
   &\leq \int\frac{d\eta'}{\la\xi'-\eta'\ra^{2s'}} \int\frac{d\eta''}{\la\eta''\ra^{2a}\la\xi''-\eta''\ra^{2b+2(s-s')}} \\
   &\quad +\int\frac{d\eta'}{\la\eta'\ra^{2s'}}\int\frac{d\eta''}{\la\eta''\ra^{2a+2(s-s')}\la\xi''-\eta''\ra^{2b}},
\end{align*}
where we choose $d/2<s'<s$ such that $a+b+s-s'>(n-d)/2$, which holds if $k+s>(n-d)/2+s'$, which is possible by our assumptions on $s$ and $k$. The integrals are uniformly bounded in $\xi$: For the $\eta'$-integrals, this follows from $s'>d/2$; for the $\eta''$-integrals, we use Lemma~\ref{LemmaFracEst}.
\end{proof}

We shall now use this (non-invariant) result to prove algebra properties of spaces with iterated module regularity: Consider a compact manifold without boundary $X$ and a submanifold $Y$. Let $\mc M\supset\Psi^0(X)$ be the $\Psi^0(X)$-module of first order \psdo{}s whose principal symbol vanishes on $N^*Y$. For $s\in\R,k\in\N_0$, define
\[
  H^{s,k}(X,Y)=\{u\in H^s(X)\colon \mc M^k u\in H^s(X)\}.
\]

\begin{prop}
\label{PropHsk}
 Suppose $\dim(X)=n$ and $\codim(Y)=d$. Assume that $s>d/2$ and $s+k>n/2$. Then $H^{s,k}(X,Y)$ is an algebra.
\end{prop}
\begin{proof}
 Away from $Y$, $H^{s,k}(X,Y)$ is just $H^{s+k}(X)$, which is an algebra since $s+k>\dim(X)/2$. Thus, since the statement is local, we may assume that we have a product decomposition near $Y$, namely $X=\R_{x'}^d\times\R_{x''}^{n-d}$, $Y=\{x'=0\}$, and that we are given arbitrary $u,v\in H^{s,k}(X,Y)$ with compact support close to $(0,0)$ for which we have to show $uv\in H^{s,k}(X,Y)$. Notice that for $f\in H^s(X)$ with such small support, $f\in H^{s,k}(X,Y)$ is equivalent to $\mc M'^k f\in H^s(X)$, where $\mc M'$ is the $C^\infty(M)$-module of differential operators generated by $\mathrm{Id},\partial_{x''_i},x'_j\partial_{x'_k}$, where $1\leq i\leq n-d,1\leq j,k\leq d$.

 Thus the proposition follows from the following statement: For $s,k$ as in the statement of the proposition,
 \[
   H^{s,k}(\R^n,\R^{n-d}):=\{u\in H^s(\R^n)\colon (x')^{\wt\alpha}D_{x'}^\alpha D_{x''}^\beta u\in H^s(\R^n), |\wt\alpha|=|\alpha|,|\alpha|+|\beta|\leq k\}
 \]
 is an algebra. Using the Leibniz rule, we thus have to show that
 \begin{equation}
 \label{EqProduct}((x')^{\wt\alpha}D_{x'}^\alpha D_{x''}^\beta u)((x')^{\wt\gamma}D_{x'}^\gamma D_{x''}^{\delta} v)\in H^s,
 \end{equation}
 provided
 $|\wt\alpha|=|\alpha|,|\wt\gamma|=|\gamma|,|\alpha|+|\beta|+|\gamma|+|\delta|\leq
 k$. Since the two factors in \eqref{EqProduct} lie in
 $H^{s,k-|\alpha|-|\beta|}$ and $H^{s,k-|\gamma|-|\delta|}$,
 respectively, this amounts to showing that $H^{s,a}\cdot
 H^{s,b}\subset H^s$ for $a+b\geq k$. This however is easy to see,
 since $H^{s,c}\subset\mc Y^{s,c}_d$ for all $c\in\N_0$, and $\mc
 Y^{s,a}_d\cdot\mc Y^{s,b}_d\subset H^s$ was proved in
 \eqref{eq:Ysa-mult-prop}.
\end{proof}

In order to be able to obtain sharper results for particular non-linear equations in \S\ref{subsec:semi-dS}, we will now prove further results in the case $\codim(Y)=1$, which we will assume to hold from now on; also, we fix $n=\dim(X)$.

\begin{prop}
\label{PropHskModule}
  Assume that $s>1/2$ and $k>(n-1)/2$. Then $H^{s,k}(X,Y)\cdot H^{s-1,k}(X,Y)\subset H^{s-1,k}(X,Y)$.
\end{prop}
\begin{proof}
  Using the Leibniz rule, this follows from $\cY_1^{s,a}\cdot\cY_1^{s-1,b}\subset H^{s-1}$ for $a+b\geq k$. This, as before, can be reduced to the local statement on $\R^n=\R_{x_1}\times\R_{x'}^{n-1}$ with $Y=\{x_1=0\}$. We write $\xi=(\xi_1,\xi')\in\R^{1+(n-1)}$ and $\eta=(\eta_1,\eta')\in\R^{1+(n-1)}$. By Lemma~\ref{LemmaProductsWeights}, the case $s\geq 1$ follows from the estimate
  \begin{align*}
    \int&\frac{\la\xi\ra^{2(s-1)}}{\la\eta\ra^{2s}\la\eta'\ra^{2a} \la\xi-\eta\ra^{2(s-1)}\la\xi'-\eta'\ra^{2b}}\,d\eta \\
	  &\lesssim \int\frac{d\eta}{\la\eta\ra^2 \la\eta'\ra^{2a} \la\xi-\eta\ra^{2(s-1)}\la\xi'-\eta'\ra^{2b}}+\int\frac{d\eta}{\la\eta\ra^{2s}\la\eta'\ra^{2a}\la\xi'-\eta'\ra^{2b}} \\
	  &\leq 2\int\frac{d\eta_1}{\la\eta_1\ra^{2s}} \int\frac{d\eta'}{\la\eta'\ra^{2a}\la\xi'-\eta'\ra^{2b}} \in L^\infty_\xi
  \end{align*}
  by Lemma~\ref{LemmaFracEst}.

  If $1/2<s\leq 1$, then $\xi_1$ and $\xi'$ play different
  roles. Indeed, the background regularity to be proved is $H^{s-1}$,
  $s-1\leq 0$, thus the continuity of multiplication in the conormal
  direction to $Y$ is proved by `duality' (i.e.\ using
  Lemma~\ref{LemmaProductsWeights} with $M_-<\infty$), whereas the
  continuity in the tangential (to $Y$) directions, where both factors
  have $k>(n-1)/2$ derivatives, is proved directly (i.e.\ using
  Lemma~\ref{LemmaProductsWeights} with $M_+<\infty$). So let $u\in\cY_1^{s,a}$, $v\in\cY_1^{s-1,b}$, and put
  \[
    u_0(\xi)=\la\xi\ra^s\la\xi'\ra^a u(\xi)\in L^2(\R^n),\quad v_0(\xi)=\la\xi\ra^{s-1}\la\xi'\ra^b v(\xi)\in L^2(\R^n).
  \]
  Then
  \[
    \la\xi\ra^{s-1}\widehat{uv}(\xi)=\int\frac{\la\eta\ra^{1-s}}{\la\xi\ra^{1-s}\la\eta'\ra^b\la\xi-\eta\ra^s\la\xi'-\eta'\ra^a} u_0(\xi-\eta)v_0(\eta)\,d\eta,
  \]
  hence by Cauchy-Schwarz and Lemma~\ref{LemmaFracEst}
  \begin{align*}
    \int&\la\xi\ra^{2(s-1)}|\widehat{uv}(\xi)|^2\,d\xi \\
	  &\leq \int\left(\int\frac{d\eta'}{\la\eta'\ra^{2b}\la\xi'-\eta'\ra^{2a}}\right)\left(\int \left| \int\frac{\la\eta\ra^{1-s}}{\la\xi\ra^{1-s}\la\xi-\eta\ra^s} u_0(\xi-\eta)v_0(\eta)\,d\eta_1\right|^2\,d\eta'\right) d\xi \\
	  &\lesssim \iint \left(\int |u_0(\xi-\eta)|^2\,d\eta_1\right) \left(\int\frac{\la\eta\ra^{2(1-s)}}{\la\xi\ra^{2(1-s)}\la\xi-\eta\ra^{2s}}|v_0(\eta)|^2\,d\eta_1\right) \,d\eta'\,d\xi \\
	  &\lesssim \iint \|u_0(\cdot,\xi'-\eta')\|_{L^2}^2 |v_0(\eta)|^2 \\
	  &\hspace{12ex}\times \left(\int\frac{1}{\la\xi-\eta\ra^{2s}}+\frac{1}{\la\xi\ra^{2(1-s)}\la\xi-\eta\ra^{2(2s-1)}}\,d\xi_1\right) \,d\xi'\,d\eta \\
	  &\lesssim \|u\|_{\cY_1^{s,a}}^2\|v\|_{\cY_1^{s-1,b}}^2,
  \end{align*}
  since $1/2<s\leq 1$, thus $1-s\geq 0$ and $2s-1>0$, and the $\xi_1$-integral is thus bounded from above by
  \[
    \int\frac{1}{\la\xi_1-\eta_1\ra^{2s}}+\frac{1}{\la\xi_1\ra^{2(1-s)}\la\xi_1-\eta_1\ra^{2(2s-1)}}\,d\xi_1 \in L^\infty_{\eta_1}.
  \]
  The proof is complete.
\end{proof}

For semilinear equations whose non-linearity does not involve any derivatives, one can afford to lose derivatives in multiplication statements. We give two useful results in this context, the first being a consequence of Proposition~\ref{PropHskModule}.

\begin{cor}
\label{CorMuMult}
  Let $\mu\in\CI(X)$ be a defining function for $Y$,
  i.e.\ $\mu|_Y\equiv 0$, $d\mu\neq 0$ on $Y$, and $\mu$ vanishes on $Y$ only. Suppose $s>1/2$ and $\ell\in\C$ are such that $\Re\ell+3/2>s$. Then multiplication by $\mu_+^\ell$ defines a continuous map $H^{s,k}(X,Y)\to H^{s-1,k}(X,Y)$ for all $k\in\N_0$.
\end{cor}
\begin{proof}
  By the Leibniz rule, it suffices to prove the statement for
  $k=0$. We have $\mu_+^\ell\in H^{\Re\ell+1/2-\eps;\infty}(X,Y)$ for
  all $\eps>0$: Indeed, the Fourier transform of
    $\chi(x)x_+^\ell$ on $\R$, with $\chi\in\CI_c(\R)$, is bounded by
    a constant multiple of $\la\xi\ra^{-\Re\ell-1}$, which is an element of $\la\xi\ra^{-r}L^2_\xi$ if and only if $r-\Re\ell-1<-1/2$, i.e.\ if $\Re\ell+1/2>r$. Hence, the corollary follows from Proposition~\ref{PropHskModule}, since one has $\Re\ell+1/2-\eps\geq s-1$ for some $\eps>0$ provided $\Re\ell+3/2>s$.
\end{proof}

\begin{prop}
\label{PropHskLowRegMult}
  Let $0\leq s',s_1,s_2<1/2$ be such that $s'<s_1+s_2-1/2$, and let $k>(n-1)/2$. Then $H^{s_1,k}(X,Y)\cdot H^{s_2,k}(X,Y)\subset H^{s',k}(X,Y)$.
\end{prop}
\begin{proof}
 Using the Leibniz rule, this reduces to the statement that $\cY^{s_1,a}_1\cdot\cY^{s_2,b}_1\subset H^{s'}$ if $a+b\geq k$. Splitting variables $\xi=(\xi_1,\xi')$, $\eta=(\eta_1,\eta')$, Lemma~\ref{LemmaProductsWeights} in turn reduces this to the observation that
 \begin{align*}
   \int&\frac{\la\xi\ra^{2s'}}{\la\eta\ra^{2s_1}\la\eta'\ra^{2a}\la\xi-\eta\ra^{2s_2}\la\xi'-\eta'\ra^{2b}}\,d\eta \\
     &\lesssim \left(\int\frac{d\eta_1}{\la\eta_1\ra^{2(s_1-s')}\la\xi_1-\eta_1\ra^{2s_2}}+\int\frac{d\eta_1}{\la\eta_1\ra^{2s_1}\la\xi_1-\eta_1\ra^{2(s_2-s')}}\right) \\
	 &\hspace{12ex}\times\int\frac{d\eta'}{\la\eta'\ra^{2a}\la\xi'-\eta'\ra^{2b}}
 \end{align*}
 is uniformly bounded in $\xi$ by Lemma~\ref{LemmaFracEst} in view of $s'<s_1+s_2-1/2<\min\{s_1,s_2\}$, thus $s_1-s'>0$ and $s_2-s'>0$, and $s_1+s_2-s'>1/2$, as well as $a+b>(n-1)/2$.
\end{proof}

\begin{cor}
\label{CorHskLowRegPower}
  Let $p\in\N$ and $s=1/2-\eps$ with $0\leq\eps<1/2p$, and let $k>(n-1)/2$. Then $u\in H^{s,k}(X,Y)$ $\Rightarrow$ $u^p\in H^{0,k}(X,Y)$.
\end{cor}
\begin{proof}
  Proposition~\ref{PropHskLowRegMult} gives $u^2\in H^{1/2-2\eps-\eps'_2,k}$ for all $\eps'_2>0$, thus $u^3\in H^{1/2-3\eps-\eps'_3,k}$ for all $\eps'_3>0$, since $\eps'_2>0$ is arbitrary; continuing in this way gives $u^p\in H^{1/2-p\eps-\eps'_p,k}$ for all $\eps'_p>0$, and the claim follows.
\end{proof}

\subsection{A class of semilinear equations}
\label{subsec:semi-dS}

Recall that we have a forward solution operator $S_\sigma\colon H^{s-1,k}(\Omega)^{\bullet,-}\to H^{s,k}(\Omega)^{\bullet,-}$ of $P_\sigma$, defined in \eqref{EqPSigma}, provided $s<1/2-\Im\sigma$. Let us fix such $s\in\R$ and $\sigma\in\C$. Undoing the conjugation, we obtain a forward solution operator
\begin{align*}
  S&=\mu^{-1/2}\mu^{-i\sigma/2+(n+1)/4}S_\sigma \mu^{i\sigma/2-(n+1)/4}\mu^{-1/2}, \\
  S&\colon \mu^{(n+3)/4+\Im\sigma/2}H^{s-1,k}(\Omega)^{\bullet,-} \to \mu^{(n-1)/4+\Im\sigma/2}H^{s,k}(\Omega)^{\bullet,-}
\end{align*}
of $(\Box_g-(n-1)^2/4-\sigma^2)$. Since $g$ is a 0-metric, the natural vector fields to appear in a non-linear equation are 0-vector fields; see \S\ref{SubsecGlobalFromStatic} for a brief discussion of these concepts. However, since the analysis is based on ordinary Sobolev spaces relative to which one has b-regularity (regularity with respect to the module $\mc M$), we consider b-vector fields in the non-linearities. In case one does use 0-vector fields, the solvability conditions can be relaxed; see \S\ref{SubsecSemiPolyDS}.

\begin{thm}
\label{ThmGlobalDSQu}
Suppose $s<1/2-\im\sigma$.
  Let
  \begin{align*}
    q\colon \mu^{(n-1)/4+\Im\sigma/2}&H^{s,k}(\Omega)^{\bullet,-}\times\mu^{(n-1)/4+\Im\sigma/2}H^{s,k-1}(\Omega;\Tb^*\Omega)^{\bullet,-} \\
	  &\to \mu^{(n+3)/4+\Im\sigma/2}H^{s-1,k}(\Omega)^{\bullet,-}
  \end{align*}
  be a continuous function with $q(0,0)=0$ such that there exists a continuous non-decreasing function $L\colon\R_{\geq 0}\to\R$ satisfying
  \[
    \|q(u,\bdiff u)-q(v,\bdiff v)\|\leq L(R)\|u-v\|,\quad \|u\|,\|v\|\leq R.
  \]
  Then there is a constant $C_L>0$ so that the following holds: If $L(0)<C_L$, then for small $R>0$, there exists $C>0$ such that for all $f\in\mu^{(n+3)/4+\Im\sigma/2}H^{s-1,k}(\Omega)^{\bullet,-}$ with norm $\leq C$, the equation
  \[
    \biggl(\Box_g-\Bigl(\frac{n-1}{2}\Bigr)^2-\sigma^2\biggr)u=f+q(u,\bdiff u)
  \]
  has a unique solution $u\in\mu^{(n-1)/4+\Im\sigma/2}H^{s,k}(\Omega)^{\bullet,-}$, with norm $\leq R$, that depends continuously on $f$.
\end{thm}
\begin{proof}
  Use the Banach fixed point theorem as in the proof of Theorem~\ref{ThmDSQu}.
\end{proof}

\begin{rmk}
  As in Theorem~\ref{ThmDSQu}, we can also allow non-linearities $q(u,\bdiff u,\Box_g u)$, provided
  \begin{align*}
    q\colon &\mu^{(n-1)/4+\Im\sigma/2}H^{s,k}(\Omega)^{\bullet,-}\times\mu^{(n-1)/4+\Im\sigma/2}H^{s-1,k}(\Omega;\Tb^*\Omega)^{\bullet,-} \\
	  &\hspace{15ex}\times\mu^{(n+3)/4+\Im\sigma/2}H^{s-1,k}(\Omega)^{\bullet,-} \\
	  &\to \mu^{(n+3)/4+\Im\sigma/2}H^{s-1,k}(\Omega)^{\bullet,-}
  \end{align*}
  is continuous, $q(0,0,0)=0$ and $q$ has a small Lipschitz constant near $0$.
\end{rmk}

\subsection{Semilinear equations with polynomial non-linearity}
\label{SubsecSemiPolyDS}

Next, we want to find a forward solution of the semilinear PDE
\begin{equation}
\label{EqDSPDE}\biggl(\Box_g-\Bigl(\frac{n-1}{2}\Bigr)^2-\sigma^2\biggr)u=f+c\mu^A u^p X(u)
\end{equation}
where $c\in C^\infty(\wt X)$, and $X(u)=\prod_{j=1}^q X_j u$ is a
$q$-fold product of derivatives of $u$ along vector fields $X_j\in\mc
M$. The purpose of the following computations is to obtain conditions on $A,p,q$ which guarantee that the map $u\mapsto c\mu^A u^p X(u)$ satisfies the conditions of the map $q$ in Theorem~\ref{ThmGlobalDSQu}. Note that the derivatives in the non-linearity lie in the module $\mc M$ (in coordinates: $\mu\pa_\mu$, $\pa_y$), whereas, as mentioned above, the natural vector fields are 0-derivatives (in coordinates: $x\pa_x=2\mu\pa_\mu$ and $x\pa_y=\mu^{1/2}\pa_y$), but since it does not make the computation more difficult, we consider module instead of 0-derivatives and compensate this by allowing any weight $\mu^A$ in front of the non-linearity.

Rephrasing the PDE in terms of $P_\sigma$ using $\wt u=\mu^{i\sigma/2-(n+1)/4+1/2}u$ and $\wt f=\mu^{-1/2+i\sigma/2-(n+1)/4}f$, we obtain
\begin{align*}
  P_\sigma\wt u&=\wt f+c\mu^A\mu^{-1/2+i\sigma/2-(n+1)/4}\mu^{(p+q)(-i\sigma/2+(n-1)/4)}\wt u^p\prod_{j=1}^q (f_j+X_j\wt u) \\
	&=\wt f+c\mu^\ell\wt u^p\prod_{j=1}^q (f_j+X_j\wt u),
\end{align*}
where $f_j\in\CI(\wt X)$ and
\begin{equation}
\label{EqGlobalDSEll}
  \ell=A+(p+q-1)(-i\sigma/2+(n-1)/4)-1.
\end{equation}
Therefore, if $\wt u\in H^{s,k}(\Omega)^{\bullet,-}$, we obtain that the right hand side of the equation lies in $H^{s,k-1}(\Omega)^{\bullet,-}$ if $\wt f\in H^{s,k-1}(\Omega)^{\bullet,-}$, $s>1/2,k>(n+1)/2$, which by Proposition~\ref{PropHsk} implies that $H^{s,k-1}(\Omega)^{\bullet,-}$ is an algebra, and if
\begin{equation}
\label{EqGeneralDSCond}
  \Re\ell+1/2=A+(p+q-1)(\Im\sigma/2+(n-1)/4)-1/2>s,
\end{equation}
since this condition ensures that $\mu^\ell\in H^{s,\infty}(X)$, which
implies that multiplication by $\mu^\ell$ is a bounded map
$H^{s,k-1}(\Omega)^{\bullet,-}\to
H^{s,k-1}(\Omega)^{\bullet,-}$.\footnote{If one works in higher
  regularity spaces, $s\geq 3/2$, we in fact only need
  $\Re\ell+3/2>s$, since then multiplication by $\mu^\ell$ is a
  bounded map $H^{s,k-1}(\Omega)^{\bullet,-}\subset
  H^{s-1,k}(\Omega)^{\bullet,-}\to
  H^{s-1,k}(\Omega)^{\bullet,-}$. However, the solvability criterion
  \eqref{EqDSCondition} would be weaker, namely the role of the dimension $n$ shifts by $2$, since in order to use $s\geq 3/2$, we need $\Im\sigma<-1$.} Given the restriction \eqref{EqGlobalDSForwardCond} on $s$ and $\Im\sigma$, we see that by choosing $s>1/2$ close to $1/2$, $\Im\sigma<0$ close to $0$, we obtain the condition
\begin{equation}
\label{EqDSCondition} p+q>1+\frac{4(1-A)}{n-1}.
\end{equation}
If these conditions are satisfied, the right hand side of the
re-written PDE lies in $H^{s,k-1}(\Omega)^{\bullet,-}\subset
H^{s-1,k}(\Omega)^{\bullet,-}$, so Theorem~\ref{ThmGlobalDSQu} is
applicable, and thus \eqref{EqDSPDE} is well-posed in these spaces.

From \eqref{EqDSCondition} with $A=0$, we see that quadratic non-linearities are fine for $n\geq 6$, cubic ones for $n\geq 4$.

To sum this up, we revert back to $u=\mu^{(n-1)/4-i\sigma/2}\wt u$ and $f=\mu^{(n+3)/4-i\sigma/2}\wt f$:

\begin{thm}
\label{ThmGlobalDS}
  Let $s>1/2,k>(n+1)/2$, and assume $A\in\R$ and $p,q\in\N_0$, $p+q\geq 2$ satisfy condition \eqref{EqGeneralDSCond}. Moreover, suppose $\sigma\in\C$ satisfies \eqref{EqGlobalDSForwardCond}, i.e.\ $\Im\sigma<1/2-s$. Finally, let $c\in\CI(\wt M)$ and $X(u)=\prod_{j=1}^q X_j u$, where $X_j$ are vector fields in $\mc M$. Then for small enough $R>0$, there exists a constant $C>0$ such that for all $f\in\mu^{(n+3)/4+\Im\sigma/2}H^{s,k}(\Omega)^{\bullet,-}$ with norm $\leq C$, the PDE
  \[
    \biggl(\Box_g-\Bigl(\frac{n-1}{2}\Bigr)^2-\sigma^2\biggr)u=f+c\mu^A u^p X(u)
  \]
  has a unique solution $u\in\mu^{(n-1)/4+\Im\sigma/2}H^{s,k}(\Omega)^{\bullet,-}$, with norm $\leq R$, that depends continuously on $f$.
  
  The same conclusion holds if the non-linearity is a finite sum of terms of the form $c\mu^A u^p X(u)$, provided each such term separately satisfies \eqref{EqGlobalDSForwardCond}.
\end{thm}
\begin{proof}
  Reformulating the PDE in terms of $\wt u$ and $\wt f$ as above, this follows from an application of the Banach fixed point theorem to the map
  \[
    H^{s,k}(\Omega)^{\bullet,-}\ni\wt u\mapsto S_\sigma\biggl(\wt f+\mu^\ell\wt u^p\prod_{j=1}^q (f_j+X_j\wt u)\biggr)\in H^{s,k}(\Omega)^{\bullet,-}
  \]
  with $\ell$ given by \eqref{EqGlobalDSEll} and $f_j\in\CI(\wt X)$. Here, $p+q\geq 2$ and the smallness of $R$ ensure that this map is a contraction on the ball of radius $R$ in $H^{s,k}(\Omega)^{\bullet,-}$.
\end{proof}

\begin{rmk}
  Even though the above conditions force $\Im\sigma<0$, let us remark that the conditions of the theorem, most importantly \eqref{EqGeneralDSCond}, can be satisfied if $m^2=(n-1)^2/4+\sigma^2>0$ is real, which thus means that we are in fact considering a non-linear equation involving the Klein-Gordon operator $\Box_g-m^2$. Indeed, let $\sigma=i\wt\sigma$ with $\wt\sigma<0$, then condition~\eqref{EqGeneralDSCond} with $A=0,p+q=2$, becomes $\wt\sigma>2-(n-1)/2$ (where we accordingly have to choose $s>1/2$ close, depending on $\wt\sigma$, to $1/2$), and the requirement $\wt\sigma<0$ forces $n\geq 6$. On the other hand, we want $(n-1)^2/4-\wt\sigma^2=m^2>0$; we thus obtain the condition
  \[
    0<m^2<\left(\frac{n-1}{2}\right)^2-\left(2-\frac{n-1}{2}\right)^2
  \]
  for masses $m$ that Theorem~\ref{ThmGlobalDS} can handle, which does give a non-trivial range of allowed $m$ for $n\geq 6$.
\end{rmk}

\begin{rmk}
  Let us compare the numerology in Theorem~\ref{ThmGlobalDS} with the numerology for the static model of an asymptotically de Sitter space in \S\ref{SecStaticDeSitter}: First, we can solve fewer equations globally on asymptotically de Sitter spaces, and second, we need stronger regularity assumptions in order to make an iterative argument work: In the static model, we needed to be in a b-Sobolev space of order $>(n+2)/2$, which in the non-blown-up picture corresponds to 0-regularity of order $>(n+2)/2$, whereas in the global version, we need a background Sobolev regularity $>1/2$, relative to which we have `b-regularity' (i.e.\ regularity with respect to the module $\mc M$) of order $>(n+1)/2$. This comparison is of course only a qualitative one, though, since the underlying geometries in the two cases are different.
\end{rmk}

Using Proposition~\ref{PropHskModule} and Corollary~\ref{CorMuMult}, one can often improve this result. Thus, let us consider the most natural case of equation~\eqref{EqDSPDE} in which we use 0-derivatives $X_j$, corresponding to the 0-structure on the \emph{not} even-ified manifold $X$, and no additional weight. The only difference this makes is if there are tangential 0-derivatives (in coordinates: $\mu^{1/2}\pa_y$). For simplicity of notation, let us therefore assume that $X_j=\mu^{1/2}\wt X_j$, $1\leq j\leq\alpha$, and $X_j=\wt X_j$, $\alpha<j\leq q$, where the $\wt X_j$ are vector fields in $\mc M$. Then the PDE~\eqref{EqDSPDE}, rewritten in terms of $P_\sigma$, $\wt u$ and $\wt f$, becomes
\begin{equation}
\label{EqDSNullDeriv}
  P_\sigma\wt u=\wt f+c\mu^\ell\wt u^p\prod_{j=1}^q(\wt f_j+\wt X_j\wt u)
\end{equation}
with $\wt f_j\in\CI(\wt X)$, where
\[
  \ell=\alpha/2+(p+q-1)(-i\sigma/2+(n-1)/4)-1.
\]
First, suppose that there are no derivatives in the non-linearity so that $p\geq 2$, $q=\alpha=0$. Then $\mu^\ell\wt u^p\in H^{s-1,k}(\Omega)^{\bullet,-}$ provided $\Re\ell+3/2>s>1/2$ by Corollary~\ref{CorMuMult}; choosing $s$ arbitrarily close to $1/2$, this is equivalent to
\begin{equation}
\label{EqDecay}
  \Im\sigma/2+(n-1)/4>0.
\end{equation}
This is a very natural condition: The solution operator for the linear
wave equation produces solutions with asymptotics $\mu^{(n-1)/4\pm
  i\sigma/2}$; see \eqref{eq:dS-poles}, and recall that we are working with the even-ified manifold with boundary defining function $\mu=x^2$. The non-linear equation~\eqref{EqDSPDE} should therefore only be well-behaved if solutions to the linear equation decay at infinity, i.e.\ if $\pm\Im\sigma+(n-1)/4\geq 0$. Since we need $\Im\sigma<0$ to be allowed to take $s>1/2$, condition~\eqref{EqDecay} is equivalent to the (small) decay of solutions to the linear equation at infinity (where $\mu=0$).

Next, let us assume that $q>0$. Then the non-linear term in equation~\eqref{EqDSNullDeriv} is an element of
\[
  \mu^\ell H^{s,k}(\Omega)^{\bullet,-}\cdot H^{s,k-1}(\Omega)^{\bullet,-}\subset H^{s,k-1}(\Omega)^{\bullet,-}
\]
by Proposition~\ref{PropHskModule}, provided $\Re\ell+1/2>s>1/2$, which gives the condition
\[
  \Im\sigma/2+(n-1)/4>1-\alpha/2
\]
where we again choose $s>1/2$ arbitrarily close to $1/2$, i.e.\ for $\alpha=2$, we again get condition \eqref{EqDecay}, and for $\alpha>2$, we get an even weaker one.

\ \\
Finally, let us discuss a non-linear term of the form $c\mu^A u^p$, $p\geq 2$, in the setting of even lower regularity $0\leq s<1/2$, the technical tool here being Corollary~\ref{CorHskLowRegPower}: Rewriting the PDE~\eqref{EqDSPDE} with this non-linearity in terms of $P_\sigma$, $\wt u$ and $\wt f$, we get
\[
  P_\sigma\wt u=\wt f+c\mu^\ell\wt u^p,\quad \ell=A+(p-1)(-i\sigma/2+(n-1)/4)-1.
\]
Let $s=1/2-\eps$ with $0\leq\eps<1/2p$. Then if $\wt u\in H^{1/2-\eps,k}(\Omega)^{\bullet,-}$ with $k>(n-1)/2$, Corollary~\ref{CorHskLowRegPower} yields $\wt u^p\in H^{0,k}(\Omega)^{\bullet,-}$, thus
\[
  \mu^\ell\wt u^p\in H^{0,k}(\Omega)^{\bullet,-}\subset H^{s-1,k}(\Omega)^{\bullet,-}
\]
provided $\Re\ell\geq 0$, i.e.
\begin{equation}
\label{EqLowRegCond}
  n>1+\frac{4(1-A)}{p-1}-2\Im\sigma,
\end{equation}
where we still require $\Im\sigma<1/2-s=\eps$, which in particular allows $\sigma$ to be real if $\eps>0$.

In summary:
\begin{thm}
\label{ThmGlobalDSLowReg}
  Let $p\geq 2$ be an integer, $1/2-1/2p<s\leq 1/2,k>(n-1)/2$, and suppose $\sigma\in\C$ is such that $\Im\sigma<1/2-s$. Moreover, assume $A\in\R$ and the dimension $n$ satisfy condition \eqref{EqLowRegCond}. Then for small enough $R>0$, there exists a constant $C>0$ such that for all $f\in\mu^{(n+3)/4+\Im\sigma/2}H^{s,k}(\Omega)^{\bullet,-}$ with norm $\leq C$, the PDE
  \[
    \biggl(\Box_g-\Bigl(\frac{n-1}{2}\Bigr)^2-\sigma^2\biggr)u=f+c\mu^A u^p
  \]
  has a unique solution $u\in\mu^{(n-1)/4+\Im\sigma/2}H^{s,k}(\Omega)^{\bullet,-}$, with norm $\leq R$, that depends continuously on $f$.
\end{thm}

In particular, if $1/4<s<1/2$, $0<\Im\sigma<1/2-s$ and $A=0$, then
quadratic non-linearities are fine for $n\geq 5$; if $\im\sigma=0$ and
$A=0$, then they work for $n\geq 6$.

\subsubsection{Backward solutions to semilinear equations with polynomial non-linearity}\label{sec:backward-dS}

Recalling the setting of \S\ref{SubsubsecDSBackward}, let us briefly turn to the backward problem for \eqref{EqDSPDE}, which we rephrase in terms of $P_\sigma$ as above. For simplicity, let us only consider the `least sophisticated' conditions, namely $s>1/2,k>(n+1)/2$,
 \begin{equation}
 \label{EqGeneralDSCondBack}
   A+(p+q-1)(\Im\sigma/2+(n-1)/4)-1/2>s,
 \end{equation}
and, this is the important change compared to the forward problem, $s>1/2-\Im\sigma$, where the latter guarantees the existence of the backward solution operator $S_\sigma^-$. Thus, if $\Im\sigma>0$ is large enough and $s>1/2$ satisfies \eqref{EqGeneralDSCondBack}, then \eqref{EqDSPDE} is solvable in any dimension.

In the special case that we only consider 0-derivatives and no extra weight, which corresponds to putting $A=q+\alpha/2$, we obtain the condition
\[
  \Im\sigma>\frac{4(1-q-\alpha/2)-(p+q-1)(n-1)}{2(p+q+1)}
\]
if we choose $s>1/2-\Im\sigma$ close to $1/2$, which in particular allows $\im\sigma\geq 0$, and thus $\sigma^2$ {\em arbitrary}, if
$p>1+\frac{4}{n-1}$ (so $p\geq 2$ is acceptable if $n\geq 6$) or $q+\alpha/2\geq 1$.

\subsection{From static parts to global asymptotically de Sitter spaces}
\label{SubsecGlobalFromStatic}

Let us consider the equation
\begin{equation}
\label{EqGS}
  (\Box_g-m^2)u=f+q(u,{}^0du),
\end{equation}
where the reason for using the 0-differential ${}^0d$, see below, will be given momentarily. The idea is that every point in $X$ lies in the interior of the backward light cone from some point $p$ at future infinity $Y_+$, denoted $S_p$; that is, the blow-up of $\ol X$ at $p$ contains the static part $S_p$ of an asymptotically de Sitter space where the solvability statements have been explained in \S\ref{SecStaticDeSitter}. Consider a suitable neighborhood $\Omega_p\subset[\ol X;p]$ of the static patch as in \S\ref{SecStaticDeSitter}, so the boundary of $\Omega_p$ is the union of $\pa S_p$ and an `artificial' spacelike boundary, which on the non-blown-up space $\ol X$ all meet at the point $p$, and a Cauchy surface. In fact, we may choose the $\Omega_p$ in a fashion that is uniform in $p$. We then solve equation~\eqref{EqGS} on $\Omega_p$, thereby obtaining a forward solution $u_p$, and by local uniqueness for $\Box_g-m^2$ in $X$, all such solutions agree on their overlap, i.e.\ $u_p\equiv u_q$ on $\Omega_p\cap\Omega_q$. Therefore, we can define a function $u$ by setting $u=u_p$ on $\Omega_p$, $p\in Y_+$, which then is a solution of \eqref{EqGS} on $X$. To make this precise, we need to analyze the relationships between the function spaces on the $\Omega_p$, $p\in Y_+$, and $X$. As we will see in Lemma~\ref{LemmaHbH0} below, b-Sobolev spaces on the blow-ups $\Omega_p$ of $\overline{X}$ at boundary points are closely related to 0-Sobolev spaces on $X$.

Recall the definition of 0-Sobolev spaces on a manifold with boundary $M$ (for us, $M=\overline{X}$) with a 0-metric, i.e.\ a metric of the form $x^{-2}\wh g$ with $x$ a boundary defining function, where $\wh g$ extends non-degenerately to the boundary: If $\Vf_0(M)=x\Vf(M)$ denotes the Lie algebra of 0-vector fields, where $\Vf(M)$ are smooth vector fields on $M$, and $\Diff_0^*(M)$ the enveloping algebra of 0-differential operators, then
\[
  H_0^k(M)=\{u\in L^2(M,d\mathrm{vol})\colon Pu\in L^2(M,d\mathrm{vol}), P\in\Diff_0^k(M)\},
\]
and $H_0^{k,\ell}(M)=x^\ell H_0^k(M)$. For clarity, we shall write $L^2_0(M)=L^2(M,d\mathrm{vol})$. We also recall the definition of the 0-(co)tangent spaces: If $\cI_p$ denotes the ideal of $\CI(M)$ functions vanishing at $p\in M$, then the 0-tangent space at $p$ is defined as ${}^0T_pM=\Vf_0(M)/\cI_p\cdot\Vf_0(M)$, and the 0-cotangent space at $p$, ${}^0T^*_pM$, as the dual of ${}^0T_pM$. In local coordinates $(x,y)\in\R_x\times\R_y^{n-1}$ near the boundary of $M$, we have $d\mathrm{vol}=f(x,y)\frac{dx}{x}\frac{dy}{x^{n-1}}$ with $f$ smooth and non-vanishing, and $\Vf_0(M)$ is spanned by $x\pa_x$ and $x\pa_y$; also $x\pa_x$ and $x\pa_{y_j}$, $j=2,\ldots,n$, form a basis of ${}^0T_pM$ (for $p\in\pa M$, which is the only place where 0-spaces differ from the standard spaces), and $\frac{dx}{x}$, $\frac{dy_j}{x}$, $j=2,\ldots,n$, form a basis of ${}^0T^*_pM$. The exterior derivative $d$ induces the first order 0-differential operator ${}^0d$ on sections of $\Lambda{}^0TM$; this follows from
\[
  df = (\pa_x f)\,dx + (\pa_y f)\,dy = (x\pa_x f)\,\frac{dx}{x} + (x\pa_y f)\,\frac{dy}{x}.
\]

Now, let $\Omega\subset\overline{X}$ be a domain as in \S\ref{SubsecGlobalDSLinear}. Moreover, let $\beta_p\colon\Omega_p\to X$ be the blow-down map. We then have:

\begin{lemma}
\label{LemmaHbH0}
  Let $k\in\N_0$, $\ell\in\R$. Then there are constants $C>0$ and $C_\delta>0$ such that for all $\delta>0$,
  \begin{equation}
  \label{EqH0Hb}
    \|f\|_{H_0^{k,\ell-(n-1)/2-\delta}(\Omega)^\bullet}\leq C_\delta\sup_{p\in Y_+}\|\beta_p^* f\|_{\Hb^{k,\ell}(\Omega_p)^{\bullet,-}}\leq CC_\delta\|f\|_{H_0^{k,\ell}(\Omega)^\bullet}.
  \end{equation}
  Here, ($\bullet$) indicates supported distributions at the `artificial' boundary and ($-$) extendible distributions at all other boundary hypersurfaces.
\end{lemma}

\begin{proof}
  Let us work locally near a point $p\in Y_+$; since $Y_+\cong\Sphere^{n-1}$ is compact, all constructions below can be made uniformly in $p$. The only possible issues are near the boundary $Y_+=\{x=0\}$, with $x$ a boundary defining function; hence, let us work in a product neighborhood $Y_+\times[0,2\eps)_x$, $\eps>0$, of $Y_+$, and let us assume $u$ is supported is $Y_+\times[0,\eps]$.
  
  We use coordinates $x,y_2,\ldots,y_n$ such that $y_j=0$ at $p$. Coordinates on $S_p$ are then $x,z_2,\ldots,z_n$ with $z_j=y_j/x$, i.e.\ $\beta_p(x,z)=(x,xz)$, with the restriction $\sum_{j=2}^n |z_j|^2\leq 1$. Therefore,
  \begin{align*}
    \|\beta_p^* f\|_{L^2_\bl}^2&\approx\int_{S_p} |\beta_p^* f(x,z)|^2\,\frac{dx}{x}\,dz =\int_{\beta_p(S_p)} |f(x,xz)|^2\,\frac{dx}{x}\,dz \\
	  &\leq \int |f(x,y)|^2\,\frac{dx}{x}\,\frac{dy}{x^{n-1}}\approx \|f\|_{L^2_0}^2.
  \end{align*}
  Adding weights to this estimate is straightforward. Next, we observe
  \begin{equation}
  \label{Eq0tob}
    \begin{split}
      x\pa_x(\beta_p^* f)(x,z)&=x\pa_x f(x,xz)+zx\pa_y f(x,xz) \\
	  \pa_z(\beta_p^* f)(x,z) &=x\pa_y f(x,xz),
	\end{split}
  \end{equation}
  and since $|z|\leq 1$, we conclude that $\beta_p^* f\in\Hb^1(S_p)$ is equivalent to $f,x\pa_x f,x\pa_y f\in L^2_0(\beta_p(S_p))$, which proves the second inequality in \eqref{EqH0Hb} in the case $k=1$; the general case is similar.

  For the first inequality in \eqref{EqH0Hb}, we first note that the additional weight comes from the number of static parts, i.e.\ interiors of backward light cones from points in $Y_+$, that one needs to cover any fixed half space $\{x\geq x_0\}$: Namely, for $0<x_0\leq\eps$, let $\sB(x_0)\subset Y_+$ be a set of points such that every point in $\{x\geq x_0\}$ lies in $S_p$ for some $p\in\sB(x_0)$; then we can choose $\sB(x_0)$ such that $|\sB(x_0)|\leq Cx_0^{-(n-1)}$, where $|\cdot|$ denotes the number of elements in a set. This follows from the observation that the area of the slice $x=x_0$ of $S_p$ within $Y_+\cong\Sphere^{n-1}$ (keeping in mind that we are working in a product neighborhood of $Y_+$) is bounded from below by $cx_0^{n-1}$ for some $p$-independent constant $c>0$. Indeed, note that null-geodesics of the 0-metric $g$ are, up to reparametrization, the same as null-geodesics of the conformally related metric $x^2g$, which is a non-degenerate Lorentzian metric up to $Y_+$. See also Figure~\ref{FigStaticToGlobal} below.
  
  Thus, putting $\alpha=(n-1)/2+\delta$, $\delta>0$, we estimate
  \begin{align*}
    \int_{x\leq\eps}&|x^\alpha f(x,y)|\,\frac{dx}{x}\,\frac{dy}{x^{n-1}}=\sum_{j=0}^\infty \int_{2^{-j-1}\eps<x\leq 2^{-j}\eps} |x^\alpha f(x,y)|^2\,\frac{dx}{x}\,\frac{dy}{x^{n-1}} \\
	  &\lesssim \sum_{j=0}^\infty 2^{-2\alpha j} \sum_{p\in\sB(2^{-j-1}\eps)}\|\beta_p^* f\|_{L^2_\bl}^2 \lesssim \sum_{j=0}^\infty 2^{-2\alpha j} (2^{-j-1}\eps)^{-n+1} \sup_{p\in Y_+}\|\beta_p^* f\|_{L^2_\bl}^2 \\
	  &\lesssim \sum_{j=0}^\infty 2^{-j(2\alpha-n+1)} \sup_{p\in Y_+}\|\beta_p^* f\|_{L^2_\bl}^2,
  \end{align*}
  with the sum converging since $2\alpha-n+1=2\delta>0$. Weights and higher order Sobolev spaces are handled similarly, using \eqref{Eq0tob}.
\end{proof}

In particular, this explains why in equation~\eqref{EqGS} we take $d={}^0d\colon H_0^{k,\ell}(X)\to H_0^{k-1,\ell}(X;{}^0T^*X)$, namely this is necessary in order to make the global equation interact well with the static patches.

Since we want to consider local problems to solve the global one, the non-linearity $q$ must be local in the sense that $q(u,{}^0du)(p)$ for $p\in X$ only depends on $p$ and its arguments evaluated at $p$; let us for simplicity assume that $q$ is in fact a polynomial as in \eqref{EqDSPolynomial}.

Using Corollary~\ref{cor:ndS}, we then obtain:

\begin{thm}
\label{ThmGSdSQu}
  Let $0\leq\eps<\eps_0$ with $\eps_0$ as in \S\ref{SubsecStaticDSSemi}, and $s>\max(3/2+\eps,n/2+1)$, $s\in\N$. Let
  \[
    q(u,{}^0du)=\sum_{2\leq j+|\alpha|\leq d} q_{j\alpha} u^j \prod_{k\leq|\alpha|}X_{\alpha,k}u,
  \]
  $q_{j,\alpha}\in\Cx+H_0^s(\overline X)$, $X_{\alpha,k}\in\Vf_0(M)$. Then there exists $C>0$ such that for all $f\in H_0^{s-1,\eps}(\Omega)^\bullet$ with norm $\leq C$, the equation
  \[
    (\Box_g-m^2)u=f+q(u,{}^0du)
  \]
  has a unique solution $u\in\bigcap_{\delta>0} H_0^{s,\eps-(n-1)/2-\delta}(\Omega)^\bullet$ that depends continuously on $f$. Here, we allow $m=0$ if every summand of $q$ contains at least one 0-derivative, and require $m>0$ if this is not the case, e.g.\ if $q=q(u)$ is simply the sum of (multiple of) powers of $u$.

  The analogous conclusion also holds for $\Box_g u=f+q({}^0du)$ provided $\eps>0$, with the solution $u$ being in $\bigcap_{\delta>0}H_0^{s,-(n-1)/2-\delta}(\Omega)^\bullet$. Moreover, for all $p\in Y_+$, the limit $u_\pa(p):=\lim_{p'\to p,p'\in X} u(p')$ exists, $u_\pa\in C^{0,\eps}(Y_+)$, and $u-u_\pa(\phi\circ\frakt_1)\in x^\eps C^0(\overline X)$, where $\phi\circ\frakt_1$ is identically $1$ near $Y_+$ and vanishes near the `artificial' boundary of $\Omega$.
\end{thm}
\begin{proof}
  We start by proving the first part: If $f\in H_0^{s-1,\eps}(\Omega)^\bullet$, then $f_p=\beta_p^* f\in\Hb^{s-1,\eps}(S_p)$ is a uniformly bounded family in the respective norms by Lemma~\ref{LemmaHbH0}. We can then use Corollary~\ref{cor:ndS} to solve
  \[
    (\Box_g-m^2)u_p=f_p+q(u_p,\bdiff u_p)
  \]
  in the static part $S_p$, where we use that $q$ is a polynomial and the fact that $\Tb^*_{p'}S_p$ naturally injects into ${}^0T^*_{\beta_p(p')}\Omega$ for $p'\in S_p$ to make sense of the non-linearity; we thus obtain a uniformly bounded family $u_p=\wt u_p|_{S_p}\in\Hb^{s,\eps}(S_p)^{\bullet,-}$. By local uniqueness and since $f$ vanishes near $Y_-$, we see that the function $u$, defined by $u(\beta_p(p'))=u_p(p')$ for $p\in Y_+$, $p'\in S_p$, is well-defined, and by Lemma~\ref{LemmaHbH0}, we indeed have $u\in H_0^{s,\eps-(n-1)/2-\delta}(\Omega)^\bullet$ for all $\delta>0$.

  For the second part, we follow the same strategy, obtaining solutions $u_p=c_p(\phi\circ\frakt_1)+u'_p$ of
  \[
    \Box_g u_p=f_p+q(\bdiff u_p),
  \]
  where $c_p\in\C$ and $u'_p\in\Hb^{s,\eps}(S_p)^{\bullet,-}$ are uniformly bounded, thus $u_p$ is uniformly bounded in $\Hb^{s,-\delta}(\Omega)^\bullet$ for every fixed $\delta>0$, and therefore the existence of a unique solution $u$ follows as before. Put $u_\pa(p):=c_p$, then $u_\pa(p)=\lim_{p'\to p,p'\in S_p} u(p')$, since $u'_p\in x^\eps C^0(S_p)$ by the Sobolev embedding theorem. We first prove that $u_\pa$ so defined is $\eps$-H\"older continuous. Let us work in local coordinates $(x,y)$ near a point $(0,y_0)$ in $Y_+$. Now, $u'_p$ is uniformly bounded in $x^\eps C^0(S_p)$, and since for $x_0>0$ arbitrary, we have $c_{p_1}+u'_{p_1}(x_0,y_*)=c_{p_2}+u'_{p_2}(x_0,y_*)$ for all $p_1,p_2\in Y_+$, provided $|p_1-p_2|\leq cx_0$ for some constant $c>0$, which ensures that $S_{p_1}\cap S_{p_2}\cap\{x=x_0\}$ is non-empty and thus contains a point $(x_0,y_*)$ (see Figure~\ref{FigStaticToGlobal}), we obtain
  \[
    |c_{p_1}-c_{p_2}|=|u'_{p_1}(x_0,y_*)-u'_{p_2}(x_0,y_*)|\leq C x_0^\eps,\quad |p_1-p_2|\leq cx_0
  \]
  for all $x_0$, thus
  \[
    \frac{|u_\pa(p_1)-u_\pa(p_2)|}{|p_1-p_2|^\eps}\leq C,\quad p_1,p_2\in Y_+.
  \]
  \begin{figure}[!ht]
    \centering
	\includegraphics{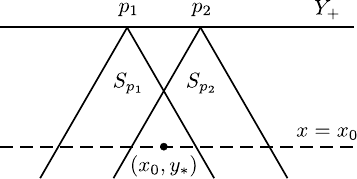}
	\caption{Setup for the proof of $u_\pa\in C^{0,\eps}(Y_+)$: Shown are the backward light cones from two nearby points $p_1,p_2\in Y_+$ that intersect within the slice $\{x=x_0\}$ at a point $(x_0,y_*)$.}
	\label{FigStaticToGlobal}
  \end{figure}

  This in particular implies that
  \begin{equation}
  \label{EqLimEstimate}
    \begin{split}
      |u(x,y)-u_\pa(0,y_0)|&\leq |u(x,y)-u_\pa(0,y)|+|u_\pa(0,y)-u_\pa(0,y_0)| \\
	    &\leq C(|y-y_0|^\eps+x^\eps) \xra{x\to 0, y\to y_0} 0,
	\end{split}
  \end{equation}
  hence we in fact have $u_\pa(p)=\lim_{p'\to p,p'\in X}u(p')$. Finally, putting $y=y_0$ in \eqref{EqLimEstimate} proves that $u-u_\pa(\phi\circ\frakt_1)\in x^\eps C^0(\overline{X})$.
\end{proof}

The major lossy part of the argument is the conversion from $f$ to the family $\beta_p^* f$: Even though the second inequality in Lemma~\ref{LemmaHbH0} is optimal (e.g., for functions which are supported in a single static patch), one loses $(n-1)/2$ orders of decay relative to the gluing estimate, i.e.\ the first inequality in Lemma~\ref{LemmaHbH0}, which is used to pass from the family $u_p$ to $u$.

Observe on the other hand that the decay properties of $u$, without regard to those of $f$, in the first part of the theorem are very natural, since the constant function $1$ is an element of $\bigcap_{\delta>0}H_0^{\infty,-(n-1)/2-\delta}(X)$, thus $u$ has an additional decay of $\eps$ relative to constants.

\begin{rmk}
  Notice that for the proof of Theorem~\ref{ThmGSdSQu} it is irrelevant whether certain 0-Sobolev spaces are algebras, since the main analysis, Corollary~\ref{cor:ndS}, is carried out on b-Sobolev spaces.
\end{rmk}

\section{Lorentzian scattering spaces}
\label{SecMinkowski}

\subsection{The linear Fredholm framework}
\label{SecMinkowski-Fredholm}

We now consider $n$-dimensional non-trapping asymptotically Minkowski
spacetimes $(M,g)$, a notion which includes
the radial compactification of Minkowski spacetime. This notion
was briefly recalled in the introduction; here we restate this in the
notation of \cite[\S3]{Ba13} where this notion was introduced.

Thus, $M$ is compact with smooth boundary, with a boundary defining
function $\rho$ (we switch the notation from $\tau$ mainly to
emphasize that $\rho$
is not everywhere timelike), and {\em scattering vector
fields} $V\in\Vsc(M)$, introduced by Melrose \cite{Me94}, are smooth vector fields of the form $\rho V'$,
$V'\in\Vb(M)$. Hence, if the $z_j$ are local coordinates on $\pa M$
extended to a neighborhood in $M$, then a local basis of these vector
fields over $\CI(M)$ is
$\rho^2\pa_{\rho},\rho\pa_{z_j}$. Correspondingly, $\Vsc(M)$ is the
set of smooth sections of a vector bundle $\Tsc M$, which is therefore, roughly
speaking, $\rho\Tb M$. The vector field $\rho^2\pa_\rho$ is
  well-defined up to a positive factor at $\rho=0$, and is called the
  {\em scattering normal vector field} of $\pa M$. The dual bundle of
  $\Tsc M$, called the {\em scattering
cotangent bundle}, is denoted by $\Tsc^*M$. If $M$ is the radial
compactification of $\RR^n$, by gluing a sphere at infinity via the
reciprocal polar coordinate map $(r,\omega)\mapsto
(r^{-1},\omega)\in(0,1)_\rho\times\Sphere^{n-1}_\omega$, i.e.\ adding
$\rho=0$ to the right hand side (corresponding to `$r=\infty$'), then $\Vsc(M)$ is spanned by (the
lifts of) the translation invariant vector fields over $\CI(M)$.

A {\em Lorentzian scattering metric} $g$ is a Lorentzian signature, taken to be $(1,n-1)$, metric on
$\Tsc M$, i.e.\ a smooth symmetric section of $\Tsc^* M\otimes\Tsc^*
M$ with this signature with the following additional properties:
\begin{enumerate}
\item
There is a real $\CI$ function $v$ defined on $M$ with $dv$, $d\rho$ linearly
independent at `the light cone at infinity', $S=\{v=0,\ \rho=0\}$,
\item
$g(\rho^2\pa_\rho,\rho^2\pa_\rho)$ has the same sign as $v$ at
$\rho=0$, i.e.\  $\rho^2\pa_\rho$ is timelike in $v>0$, spacelike in $v<0$,
\item
near $S$,
$$
g=v\frac{d\rho^2}{\rho^4}-\Bigl(\frac{d\rho}{\rho^2}\otimes\frac{\alpha}{\rho}+\frac{\alpha}{\rho}\otimes\frac{d\rho}{\rho^2}\Bigr)-\frac{\wt h}{\rho^2},
$$
where $\alpha$ is a smooth one-form on $M$,
$$
\alpha=\frac{1}{2}\,dv+\cO(v)+\cO(\rho),
$$
$\wt h$ is a smooth
2-cotensor on $M$, which is positive definite on the (codimension two)
annihilator of
$d\rho$ and $dv$.
\end{enumerate}
A Lorentzian scattering metric is {\em non-trapping} if
\begin{enumerate}
\item
$S=S_+\cup S_-$ (each a disjoint union of connected components), in
$X=\pa M$ the open set
$\{v>0\}\cap X$ decomposes as $C_+\cup C_-$ (disjoint union), with
$\pa C_+=S_+$, $\pa C_-=S_-$; we write $C_0=\{v<0\}\cap X$,
\item
 the projections of all null-bicharacteristics in $\Tsc^*M\setminus o$ to $M$
 tend to $S_\pm$ as their parameter tends to $\pm\infty$ or vice
 versa.
\end{enumerate}
Since a conformal factor only reparameterizes bicharacteristics, this
means that with $\wh g=\rho^2 g$, which is a b-metric on $M$, the
projections of all
null-bicharacteristics of $\wh g$ in $\Tb^*M\setminus o$ tend to $S_\pm$.
As already pointed out in the introduction,
the difference between the de Sitter-type and Minkowski
settings is 
that at the spherical conormal bundle $\SNb^* S$ of $S$
the nature of the radial points is source/sink rather than a saddle
point of the flow at $L_\pm$ discussed in \S\ref{SecStaticDeSitter-Fredholm}.

We first state solvability properties, namely we show that
under the assumptions of
\cite[\S3]{Ba13}, the problem of finding a tempered
solution to $\Box_{g} w = f $ is a
Fredholm problem in suitable weighted Sobolev spaces.  In particular, there is only a finite dimensional
obstruction to existence. Then we strengthen the assumptions somewhat
and show actual solvability in the strong sense that in these spaces
the solution $w$ satisfies that if $f$ is
vanishing to infinite order near $\overline{C_-}$, then so does $w$.

Let
$$
L=\rho^{-(n-2)/2}\rho^{-2}\Box_g\rho^{(n-2)/2} \in\Diffb^2(M)
$$
be the `conjugated' b-wave operator (as in \cite[\S4]{Ba13}), which is formally self-adjoint with respect to the density of the Lorentzian b-metric $\wh g=\rho^2 g$, further $L=\Box_{\wh g}-\gamma$, where $\gamma\in\CI(M)$ is real valued.  Let
\begin{equation}\begin{split}\label{eq:m-decreasing}
&m\in \CI(\Sb^*M)\ \text{a
variable (Sobolev) order function, decreasing along}\\
&\text{the direction of the Hamilton flow oriented to the future, i.e.\ towards}\ S_+.
\end{split}\end{equation}

\begin{rmk}
\label{RmkWeightFromBase}
  In the actual application of asymptotically Minkowski spaces, one can take $m$ to be a function on $M$ rather than $\Sb^*M$ by making it take constant values near $\overline{C_+}$, resp.\ $\overline{C_-}$, corresponding to the requirements at $\cR_+$, resp.\ $\cR_-$ below, and transitioning in between using a time function as in the discussion preceding Theorem~\ref{thm:asymp-Mink-lin}, i.e.\ making $m$ of the form $F\circ\wt\frakt$ for appropriate $F$. Since this simplifies some arguments below, we assume this whenever it is convenient.
\end{rmk}

With
$$
\cR_+=\Sb N^*S_+,\ \text{resp.}\ \cR_-=\Sb N^*S_-,
$$
the future, resp.\ past, radial sets in $\Sb^*M$, see \cite[\S3.6]{Ba13}, and with
$$  
m+l<1/2\ \text{at}\ \cR_+,\ m+l>1/2\ \text{at}\ \cR_-,
$$
$m$ constant near $\cR_+\cup\cR_-$,
one has an estimate
\begin{equation}\label{eq:L-b-symb-est}
\|u\|_{\Hb^{m,l}}\leq C\|Lu\|_{\Hb^{m-1,l}}+C\|u\|_{\Hb^{m',l}},
\end{equation}
provided one assumes $m'<m$,
$$ 
m'+l>1/2\ \text{at}\ \cR_-,\ u\in \Hb^{m',l}.
$$ 
To see this, we recall and record a slight improvement of \cite[Proposition~4.4]{Ba13}:

\begin{prop}\label{prop:b-source-sink}
Suppose $L$ is as above.

If  $m+l<1/2$, and if $u\in\Hb^{-\infty,l}(M)$
then $\cR_\pm$ (and thus a neighborhood of $\cR_\pm$) is disjoint from
$\WFb^{m,l}(u)$ provided $\cR_\pm\cap\WFb^{m-1,l}(L u)=\emptyset$ and
a punctured neighborhood of $\cR_\pm$, with $\cR_\pm$ removed, in
$\Sigma\cap\Sb^*M$ is disjoint from $\WFb^{m,l}(u)$.

On the other hand, if $m'+l>1/2$, $m\geq m'$, $u\in\Hb^{-\infty,l}(M)$
and if $\WFb^{m',l}(u)\cap\cR_\pm=\emptyset$
then $\cR_\pm$ (and thus a neighborhood of $\cR_\pm$) is disjoint from
$\WFb^{m,l}(u)$ provided $\cR_\pm\cap\WFb^{m-1,l}(L u)=\emptyset$.
\end{prop}

\begin{proof}
The first statement is proved in \cite[Proposition~4.4]{Ba13}. The
second statement follows the same way, but in that case the product of
the required powers of the boundary
defining functions, $\rho^{-2l}\wt\rho^{-2m+1}$, with $\wt\rho$
the defining function of fiber infinity\footnote{This defining
  function is denoted by $\nu$ in \cite{Ba13}.} as in \S\ref{SecStaticDeSitter-Fredholm}, in the commutant of
\cite[Proposition~4.4]{Ba13} provides a favorable sign, thus
\cite[Equation~(4.1)]{Ba13} holds without the $E$ term. However, when
regularizing, the regularizer contributes a term with the opposite
sign, exactly as in \cite[Proof of Propositions~2.3-2.4]{Va12}; this
forces the requirement on the a priori regularity, namely
$\WFb^{m',l}(u)\cap\cR_\pm=\emptyset$, exactly as in the referred
results of \cite{Va12}; see also Proposition~\ref{prop:b-saddle} above.
\end{proof}

Indeed, due to the closed graph theorem, \eqref{eq:L-b-symb-est} follows immediately from
the b-radial point
regularity statements of Proposition~\ref{prop:b-source-sink}
for
sources/sinks,
and the propagation of
b-singularities for variable order Sobolev spaces, which is not proved
in \cite{Ba13}, but whose analogue in standard Sobolev spaces is
proved there in \cite[Proposition~A.1]{Ba13} (with additional
references given to related results in the literature), and as it is a
purely symbolic argument, the extension to the b-setting is
straightforward. (We refer to Proposition~\ref{prop:b-saddle} here
  and \cite[Proposition~4.4]{Ba13} extending the radial point results,
  Propositions~2.3-2.4, of \cite{Va12}, from the boundaryless setting
  to the b-setting.)

One also has a similar estimate for $L$ when one replaces $m$ by a
weight $\wt m$ which is
increasing along the direction of the Hamilton flow oriented towards
the past,
$$
\wt m+\wt l>1/2\ \text{at}\ \cR_+,\ \wt m+\wt l<1/2\ \text{at}\ \cR_-,
$$
provided one assumes $\wt m'<\wt m$,
$$ 
\wt m'+\wt l>1/2\ \text{at}\ \cR_+,\ u\in \Hb^{\wt m',\wt l}.
$$ 
Further $L$ can be
replaced by $L^*$. Thus,
\begin{equation}\label{eq:L*-b-symb-est}
\|u\|_{\Hb^{\wt m,\wt l}}\leq C\|L^*u\|_{\Hb^{\wt
    m-1,\wt l}}+C\|u\|_{\Hb^{\wt m',\wt l}}.
\end{equation}

Just as in the asymptotically de Sitter/Kerr-de Sitter settings,
one wants to improve these estimates so that the space $\Hb^{m,l}$, resp.\
$\Hb^{\wt m,\wt l}$, on the left hand side includes
compactly into the error term on
the right hand side. This argument is completely analogous to
\S\ref{SecStaticDeSitter-Fredholm} using the Mellin transformed
normal operator estimates obtained in \cite[\S5]{Ba13}.
We thus
further assume that there are no poles of the Mellin conjugate $\wh
L(\sigma)$ on the line $\Im\sigma=-l$. Then using the
Mellin transform and the estimates for $\wh L(\sigma)$ (including the
high energy estimates, which imply that for all but a discrete set of
$l$ the aforementioned lines do not contain such poles), as in
\S\ref{SecStaticDeSitter-Fredholm},
we obtain that on $\RR^+_\rho\times\pa M$
\begin{equation}\label{eq:N-L-est}
\|v\|_{\Hb^{\wh m,l}}\leq C\|N(L)v\|_{\Hb^{\wh m-1,l}}
\end{equation}
when $\wh m\in\CI(S^*\pa M)$ is a
variable order function decreasing along the direction of
the Hamilton flow oriented to the future, $\Lambda_+$, resp.\ $\Lambda_-$, the
future, resp.\ past, radial sets in $S^*\pa M$, and with
$$  
\wh m+l<1/2\ \text{at}\ \Lambda_+,\ \wh m+l>1/2\ \text{at}\ \Lambda_-.
$$  
One can take
$$
\wh m=m|_{T^*\pa M},
$$
for instance, under the identification of
$T^*\pa M$ as a subspace of $\Tb^*_{\pa M}M$, taking into account that
homogeneous degree zero functions on $T^*\pa M\setminus o$ are exactly
functions on $S^*\pa M$, and analogously on $\Tb^*_{\pa M}M$. However,
in the limit $\sigma\to\infty$, one should use norms depending on
$\sigma$ reflecting the dependence of the semiclassical norm on
$h$. We recall from Remark~\ref{RmkWeightFromBase} that
in the main case of interest one can take $m$ to be a pullback from
$M$, and thus the Mellin transformed operator norms are independent of
$\sigma$. In either case, we
simply write $m$ in place of $\wh m$.

Again, we have an analogous estimate for $N(L^*)$:
\begin{equation}\label{eq:N-L*-est}
\|v\|_{\Hb^{\wt m,\wt l}}\leq C\|N(L^*)v\|_{\Hb^{\wt
    m-1,\wt l}},
\end{equation}
provided $-\wt l$ is not the imaginary part of a pole of $\widehat
{L^*}$, and provided $\wt m$ satisfies the requirements above.
As $\widehat{L^*}(\sigma)=(\wh L)^*(\overline{\sigma})$, the
requirement on $-\wt l$ is the same as $\wt l$ not
being the imaginary part of a pole of $\wh L$.

At this point the argument of the paragraph of
\eqref{eq:P-normal-op-break-up} in
\S\ref{SecStaticDeSitter-Fredholm}  can be repeated verbatim to
yield that for $m$ with $m+l>3/2$ at $\cR_-$ (with the stronger restriction
coming from the requirements on $m'$ at $\cR_-$, $\wt m'$ at
$\cR_+$, and $m'<m-1$, $\wt m'<\wt m-1$; recall that one needs
to estimate the normal operator on these primed spaces), and $m+l<1/2$ at $\cR_+$, 
\begin{equation}\label{eq:L-Fred-est}
\|u\|_{\Hb^{m,l}}\leq C\|Lu\|_{\Hb^{m-1,l}}+C\|u\|_{\Hb^{m'+1,l-1}},
\end{equation}
where now the inclusion $\Hb^{m,l}\to \Hb^{m'+1,l-1}$ is
compact (as we choose $m'<m-1$); this argument required $m,l,m'$ satisfied the
requirements preceding \eqref{eq:L-b-symb-est}, and
that $-l$ is
not the imaginary part of any pole of $\wh L$.

Analogous estimates hold for $L^*$:
\begin{equation}\label{eq:L*-Fred-est}
\|u\|_{\Hb^{\wt m,\wt l}}\leq
C\|L^*u\|_{\Hb^{\wt m-1,\wt l}}+C\|u\|_{\Hb^{m'+1,\wt l-1}},
\end{equation}
provided $\wt m$, $\wt l$, $\wt m'$
satisfy the requirements stated before \eqref{eq:L*-b-symb-est},
$\wt m'<\wt m-1$, and provided
$-\wt l$ is not the imaginary part of a pole of $\widehat {L^*}$
(i.e.\ $\wt l$ of $\wh L$).

Via the same functional analytic argument as in \S\ref{SecStaticDeSitter-Fredholm} we thus obtain Fredholm
properties of $L$, in particular solvability, modulo a (possible) finite
dimensional obstruction, in $\Hb^{m,l}$ if
$$
m+l>3/2\ \text{at}\ \cR_-,\ m+l<-1/2\ \text{at}\ \cR_+.
$$
More precisely, we take $\wt
m=1-m$, $\wt l=-l$, so $m+l<-1/2$ at $\cR_+$ means $\wt m+\wt
l=1-(m+l)>3/2$, so the space on the left hand side of
\eqref{eq:L-Fred-est} is dual to that in the first term on the right
hand side of \eqref{eq:L*-Fred-est}, and the same for the equations
interchanged. Then the Fredholm statement is for
$$
L:\cX^{m,l}\to\cY^{m-1,l},
$$
with
$$
\cY^{s,r}=\Hb^{s,r},\ \cX^{s,r}=\{u\in\Hb^{s,r}:\ Lu\in\Hb^{s-1,r}\}.
$$
Note that, by propagation of singularities, i.e.\ most importantly
using Proposition~\ref{prop:b-source-sink}, with $\Ker
L\subset\Hb^{m,l}$, $\Ker L^*\subset\Hb^{1-m,-l}$ a priori,
\begin{equation}\begin{split}\label{eq:Mink-kernel}
&\Ker L\subset\Hb^{m^\flat,l},\ \Ker L^*\subset\Hb^{1-m^\flat,-l}
\ \text{if}\\
&m^\flat+l>1/2\ \text{at}\ \cR_-,\ m^\flat+l<1/2\ \text{at}\ \cR_+.
\end{split}\end{equation}

We can improve this further using the propagation of
singularities. Namely, suppose one merely has
\begin{equation}\label{eq:weaker-Mink-Freq-req}
m+l>3/2\ \text{at}\ \cR_-,\ m+l<1/2\ \text{at}\ \cR_+,
\end{equation}
so the requirement at $\cR_+$ is weakened. Then let $m^\sharp=m-1$
near $\cR_+$, $m^\sharp\leq m$ everywhere, but still satisfying the requirements for the order
function along the Hamilton flow,
so the Fredholm result is applicable with $m^\sharp$ in place of $m$.
Now, if $u\in\cX^{m^\sharp,l}$, $Lu=f$, $f\in\cY^{m-1,l}\subset\cY^{m^\sharp-1,l}$, then
Proposition~\ref{prop:b-source-sink} gives $u\in\cX^{m,l}$. Further,
if $\Ker L$ and $\Ker L^*$ are trivial, this gives that for $m,l$ as
in \eqref{eq:weaker-Mink-Freq-req}, satisfying also the conditions
along the Hamilton flow, $L:\cX^{m,l}\to\cY^{m-1,l}$ is invertible.

Now, as invertibility (the absence of kernel and cokernel) is
preserved under sufficiently small perturbations, it holds in
particular for perturbations of the Minkowski metric which are
Lorentzian scattering metrics in our sense, with closeness measured
in smooth sections of the second symmetric power of $\Tb^*M$. (Note
that non-trapping is also preserved under such perturbations.)

For more general asymptotically Minkowski metrics we note that, due to
Theorem~\ref{thm:b-main} (which does not have any requirements for the
timelike nature of the boundary defining function, and which works
locally near $\overline{C_-}$ either by working on (extendible)
function spaces or by using the localization given by wave propagation
as in \S3.3 of \cite{Va12} or \S\ref{SubsecGlobalDSLinear} here)
elements of $\Ker L$ on $\Hb^{m,l}$, with $m,l$ as above, lie in $\dCI(M)$ locally near
$\overline{C_-}$ provided all resonances, i.e.\ poles of $\wh
L(\sigma)$, in $\im\sigma<-l$ have polar parts (coefficients of the
Laurent series) that map into distributions supported on
$\overline{C_+}$. As shown in \cite[Remark~4.17]{Va13} when $\wh
L(\sigma)$ arises from a Lorentzian conic metric as in\footnote{In
  \cite{Va13}, the boundary defining function used to define the
  Mellin transform is replaced by its reciprocal, which effectively
  switches the sign of $\sigma$ in the operator, but also the backward propagator is
  considered (propagating toward the past light cone), which reverses
  the role of $\sigma$ and $-\sigma$ again, so in fact, the signs in
  \cite{Va13} and \cite{Ba13} agree for the formulae connecting the
  asymptotically hyperbolic resolvents and the global operator, $\wh L(\sigma)$.} 
\cite[Equation~(3.5)]{Va13}, but with the arguments applicable without
significant changes in our more general case, see also
\cite[\S7]{Ba13} for our general setting, and
\cite[Remark~4.6]{Va12} for a related discussion with complex
absorption, the resonances of $\wh L(\sigma)$ consist of the
resonances of the asymptotically hyperbolic resolvents on the caps,
namely $\cR_{C_+}(\sigma)$, $\cR_{C_-}(-\sigma)$, as well as possibly
imaginary integers, $\sigma\in i\ZZ\setminus\{0\}$, with resonant
states when $\im\sigma<0$ being differentiated delta distributions at
$S_+=\pa C_+$ while the dual states are differentiated delta
distributions at $S_-=\pa C_-$ when $\im\sigma>0$;
the latter arise, e.g.\ as poles on even dimensional Minkowski space.
More generally, when composed with extension of
$\CI_c(\overline{C_-}\cup C_0)$ by zero to $\CI(X)$ from the right and with
restriction to $\overline{C_-}\cup C_0$ from the left, the only poles
of $\wh L(\sigma)$ are those of $\cR_{C_-}(-\sigma)$ as well as the possible
$\sigma\in i\NN_+$. Thus, fixing $l>-1$, one can conclude that elements
of $\Ker L$ are in $\dCI(M)$ locally near $\overline{C_-}$ provided
$\cR_{C_-}(\wt\sigma)$ has no poles in $\im\wt\sigma>l$. (The only change for $l\leq -1$
  is that one needs to exclude the potential pure imaginary integer
  poles as well.) The
analogous statement for $\Ker L^*$ on $\Hb^{\wt m,\wt l}$ is that fixing $\wt l>-1$,
elements are in $\dCI(M)$ near $\overline{C_+}$ provided
$\cR_{C_+}(\wt\sigma)$ has no poles in $\im\wt\sigma>\wt
l$. As $\wt l=-l$ for our duality arguments, the weakest symmetric
assumption (in terms of strength at $C_+$ and $C_-$) is that
$\cR_{C_\pm}$ do not have any poles in the closed upper half plane;
here the closure is added to make sure $L$ is actually Fredholm on
$\Hb^{m,l}$ with $l=0$. In general, if one wants to use other values
of $l$, one needs to assume the absence of poles in $\im\sigma\geq
-|l|$ (if one wants to keep the hypotheses symmetric).

Note that assuming $\frac{d\rho}{\rho}$ is
timelike (with respect to $\wh g$) 
near $\overline{C_-}$, one {\em automatically} has the
absence of poles of $\cR_{C_-}$
in an upper half plane,
and the finiteness (with multiplicity) of the number of poles in any upper half plane,
by the semiclassical estimates of \cite{Va12}, see \S3.2 and 7.2
(one can ignore the complex absorption discussion there), so in this case the
issue is that of a possible finite number of resonances. There is an
analogous statement if $\frac{d\rho}{\rho}$ is timelike near
$\overline{C_+}$ for $\cR_{C_+}$.

Now, assuming still that $\frac{d\rho}{\rho}$ is timelike
at, hence
near $\overline{C_-}$, it is easy
to construct a function $\frakt$ which has a timelike differential near
$\overline{C_-}$, and appropriate sublevel sets are small
neighborhoods of $\overline{C_-}$. Once one has such a function
$\frakt$, energy estimates can be used to conclude that rapidly
vanishing, in such a neighborhood, solutions of $Lu=0$ actually vanish
in this neighborhood, so elements of $\Ker L$ have support disjoint
from $\overline{C_-}$; similarly elements of $\Ker L^*$ have support
disjoint from $\overline{C_+}$.

Concretely, with $\wh G$ the dual b-metric of $\wh g$, let $U_-$ be a neighborhood of $\overline{C_-}$, and let $0<\ep_0<\ep_1$, $\wt\ep>0$, $\delta>0$ be such that $\{\rho\leq\wt\ep,\ v\geq -\ep_1\}\cap U_-$ is a compact subset of $U_-$, and on $U_-$
\begin{gather*}
  \rho<\wt\ep,\ v>-\ep_1\Rightarrow \wh G\Bigl(\frac{d\rho}{\rho},\frac{d\rho}{\rho}\Bigr)>\delta, \\
  \rho<\wt\ep,\ -\ep_1<v<-\ep_0\Rightarrow \wh G\Bigl(\frac{d\rho}{\rho},dv\Bigr)<0,\ \wh G(dv,dv)>0.
\end{gather*}
Such $U_-$ and constants indeed exist.  First, there is $U_-$ and $\wt\eps'>0$, $\eps'_1>0$ such that $\{\rho\leq\wt\eps',\ v\geq -\eps'_1\}\cap U_-$ is a compact subset of $U_-$ since $\overline{C_-}$ is defined by $\{\rho=0,\ v\geq 0\}$ in a neighborhood of $\overline{C_-}$ with $d\rho\neq 0$ there and $dv\neq 0$ near $v=0$; we then consider $\wt\eps<\wt\eps'$, $\eps_1<\eps_1'$ below.  Next, since $\wh G(\frac{d\rho}{\rho},\frac{d\rho}{\rho})$ is positive on a neighborhood of $\overline{C_-}$ by assumption (thus for any sufficiently small $\eps_1,\wt\eps$ there is a desired $\delta$ so that the first inequality is satisfied) and $\wh G(\frac{d\rho}{\rho},dv)|_{S_-}=-2$, any sufficiently small $\eps_1$ and $\wt\eps$ give $\wh G(\frac{d\rho}{\rho},dv)<0$ in the desired region, and finally $\wh G(dv,dv)>0$ on $C_0$ near $S_-$ (as $\wh G(dv,dv)=-4v+\cO(v^2)$ there), so choosing $\eps_1$ sufficiently small, $\eps_0<\eps_1$, and then $\wt\eps$ sufficiently small satisfies all criteria.

Now let $\ep_-,\ep_+$ be such that
$0<\ep_-<\ep_+<\wt\ep$, and let $\phi\in\CI(\RR)$ have $\phi'\leq
0$, $\phi=0$ near $[-\ep_0,\infty)$, $\phi>\wt\ep$ near
$(-\infty,-\ep_1]$, $\phi'<0$ when $\phi$ takes values in
$[\ep_-,\ep_+]$. Then $\frakt=\rho+\phi(v)$ has the property that on $U_-$
$$
\frakt\leq\ep_+\Rightarrow\rho,\phi(v)\leq\ep_+\Rightarrow\rho<\wt\ep,v>-\ep_1,
$$
and
$$
v\geq -\ep_0\Rightarrow\frakt=\rho.
$$
Thus, on $U_-$ if $v\geq -\ep_0$ and $\frakt\leq\ep_+$ then $d\frakt$ is timelike as
$d\rho$ is such, while if $v<-\ep_0$, $\frakt\leq\ep_+$ then
$$
\wh G(d\frakt,d\frakt)=\rho^2
\wh G\Bigl(\frac{d\rho}{\rho},\frac{d\rho}{\rho}\Bigr)+2\phi'(v)\rho
\wh G\Bigl(\frac{d\rho}{\rho},dv\Bigr)+(\phi'(v))^2\wh G(dv,dv)
$$
and
all terms are $\geq 0$ in view of $-\ep_1<v<-\ep_0$,
$\rho\leq\wt\ep$, with the inequality being strict when
$\frakt\in[\ep_-,\ep_+]$ (as well as in
$M^\circ\cap\frakt^{-1}((-\infty,\ep_+])$). Thus, near
$\frakt^{-1}([\ep_-,\ep_+])\cap U_-$,
$\frakt$ is a timelike function; the same is true on
$M^\circ\cap\frakt^{-1}((-\infty,\ep_+])\cap U_-$. Let $\chi\in\CI(\RR)$ with $\chi'\leq 0$,
$\chi=1$ near $(-\infty,\ep_-]$, $\chi=0$ near $[\ep_+,\infty)$, and
let $\chi\circ\frakt$, defined by this formula in $U_-$, be extended
to $M$ as $0$ outside $U_-$; since $\frakt^{-1}((-\infty,\ep_+])\cap
U_-$ is a compact subset of $U_-$, this gives a $\CI$ function. Further, $\rho$ is also timelike,
with $\frac{d\rho}{\rho}$ and $d\frakt$ in the same component of the
timelike cone; see Figure~\ref{FigMinkLvl}. Correspondingly, one can apply energy
estimates using the timelike vector field $V=(\chi\circ\frakt)\rho^{-\ell}
\wh G(\frac{d\rho}{\rho},.)$, cf.\ \cite[\S3.3]{Va12} leading up to Equation~(3.24) and the
subsequent discussion, which in turn is based on
\cite[\S\S3-4]{Vasy:AdS}. Here one needs to make both $-\chi'$
large relative to $\chi$ and $\ell>0$ large (making the b-derivative of
$\rho^{-\ell}$ large relative to $\rho^{-\ell}$), as discussed in the
Mellin transformed setting in \cite[\S3.3]{Va12}, in
\cite[\S\S3-4]{Vasy:AdS}, as well as
\S\ref{SecStaticDeSitter-Fredholm} here (with $\tau$ in place of
$\rho$, but with the sign of $\ell$ reversed due to the difference
between b-saddle points and b-sinks/sources). Notice that taking $\ell$ large is
exactly where the rapid decay near $\overline{C_-}$ is used.

\begin{figure}[!ht]
  \centering
  \includegraphics{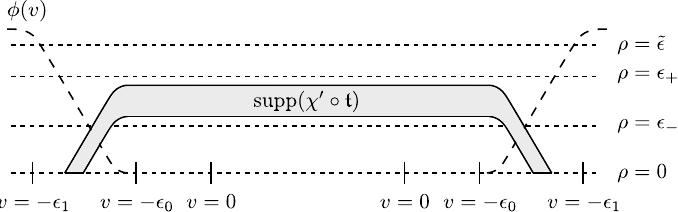}
  \caption{Setup for energy estimates near $\overline{C_-}$: The shaded region is the support of $\chi'\circ\frakt$, where $-\chi'$ is used to dominate $\chi$ to give positivity in the energy estimate; near $\rho=0$ and on $\supp(\chi\circ\frakt)$, i.e.\ in the region between $\rho=0$ and the shaded region, a sufficiently large weight $\rho^{-\ell}$ gives positivity.}
\label{FigMinkLvl}
\end{figure}

We have seen that the existence of appropriate timelike functions, such as
$\frakt$, in a neighborhood of $\overline{C_+}$ and $\overline{C_-}$
is automatic (in a slightly degenerate sense at $\overline{C_\pm}$
themselves) when $\frac{d\rho}{\rho}$ is timelike in these regions;
indeed these functions could be extended to a neighborhood of
$C_0$ if $v$ is appropriately chosen. In order to conclude that
elements of $\Ker L$ and $\Ker L^*$ vanish globally, however, we need to
control {\em all} of the interior of $M$. This can be accomplished by
showing global hyperbolicity\footnote{In Geroch's notation, our
  $M^\circ$ is $M$.} of $M^\circ$, which in turn can be seen by applying a
result due to Geroch \cite{Geroch:Domain}.  Namely, by
\cite[Theorem~11]{Geroch:Domain}
it suffices to show
that a suitable $\cS$ is a Cauchy surface, which by
\cite[Property~6]{Geroch:Domain} follows if we show that $\cS$ is
achronal, closed, and every null-geodesic intersects and then
re-emerges from $\cS$.
In order to define $\cS$, it is useful to define
$\wh\frakt=\psi\circ\frakt$ in $U_-$, where $\psi\in\CI(\RR)$,  $\psi'\geq 0$,
$\psi(t)=t$ near $t\leq\ep_-$, $\psi'(t)>0$ for $t<\ep_+$,
$\psi'(t)=0$ for $t\geq\ep_+$; let $T=\psi(\ep_+)>\ep_-$. Further,
extend $\wh\frakt$ to $M$ as $=T$ outside $U_-$; since
$U_-\cap\frakt^{-1}((-\infty,\ep_+])$ is compact, this gives a $\CI$
function on $M$. Thus,
$\wh\frakt\in\CI(M)$ is a globally weakly time-like function in that $\wh
G(d\wh\frakt,d\wh\frakt)\geq 0$, and it is strictly time-like in
$M^\circ\cap \frakt^{-1}((-\infty,\ep_+))$. In particular, it is
monotone along all null-geodesics. Further, $\wh\frakt=0$ at $S_-$
and $\wh\frakt=T>0$ at $S_+$, indeed near $S_+$. Then we claim that
$\cS=\wh\frakt^{-1}(\ep_-)\cap M^\circ$ is a Cauchy surface.

Now, $\cS$ is closed in $M^\circ$ since
$\overline{\cS}$ is closed in $M$; indeed it is a closed embedded
submanifold.
By our non-trapping assumption, every null-geodesic in $M^\circ$ tends to $S_+$ in one
direction and $S_-$ in the other direction, so on future oriented
null-geodesics (ones tending to $S_+$), $\wh\frakt$ is monotone
increasing, attaining all values in $(0,T]$. Since at the $\ep_-$
level set of $\frakt$, hence of $\wh\frakt$, $d\wh\frakt$ is
strictly time-like, the value $\ep_-$ is attained exactly once for
$\wh\frakt$ along null-geodesics. Thus, every
null-geodesic intersects $\cS$ and then re-emerges from it. Finally,
$\cS$ is achronal, i.e.\ there exist no time-like curves connecting
two points on $\cS$: any future oriented time-like curve (meaning with
tangent vector in the time-like cone whose boundary is the future light cone) in $M^\circ\cap \frakt^{-1}((-\infty,\ep_+))$ has
$\wh\frakt$ monotone increasing, with the increase being strict near
$\cS$, so again the value $\ep_-$ can be attained at most once on such
a curve. In summary, this proves that $M^\circ$ is globally
hyperbolic, so every solution of $Lu=0$ with vanishing Cauchy data on
$\cS$ vanishes identically, in particular by what we have observed,
$\Ker L$ and $\Ker L^*$ are trivial on the indicated spaces.

In summary:

\begin{thm}\label{thm:asymp-Mink-lin}
If $(M,g)$ is a non-trapping Lorentzian scattering metric
in the sense of \cite{Ba13}, $|l|<1$, and
\begin{enumerate}
\item
The induced asymptotically hyperbolic resolvents $\cR_{C_\pm}$ have no
poles in $\im\sigma\geq-|l|$,
\item
$\frac{d\rho}{\rho}$ is timelike near
$\overline{C_+}\cup\overline{C_-}$,
\end{enumerate}
then for order functions $m\in\CI(\Sb^*M)$ satisfying
\eqref{eq:m-decreasing} and \eqref{eq:weaker-Mink-Freq-req},
the forward problem for the conjugated wave operator $L$, i.e.\
with $L$ considered as a map
$$
L:\cX^{m,l}\to\cY^{m-1,l},
$$
is invertible.
\end{thm}

Extending the notation of \cite{Ba13}, especially \S4, we denote
by $\Hb^{m,l,k}(M)$, where $m,l\in\R,k\in\N_0$,
the space of all $u\in \Hb^{m,l}(M)$ (i.e.\ $u\in\rho^l H^m_\bl(M)$,
where $\rho$ is the boundary defining function of $M$) such that $\mc
M^j u\in \Hb^{m,l}(M)$ for all $0\leq j\leq k$. Here, $\mc
M\subset\Psi_\bl^1(M)$ is the $\Psi_\bl^0(M)$-module of
pseudodifferential operators with principal symbol vanishing on the
radial set $\mc R_+$ of the operator
$L=\rho^{-(n-2)/2}\rho^{-2}\Box_g\rho^{(n-2)/2}$; in the coordinates $\rho,v,y$ as in
\cite{Ba13} ($\rho$ being as above, $v$ a defining function of the light cone at infinity within $\partial M$, $y$
coordinates within in the light cone at infinity), $\mc M$ has local
generators $\rho\partial_\rho,\rho\partial_v,v\partial_v,\partial_y$.
Then the results of \cite{Ba13}, concretely Proposition~4.4, extend our theorem to the spaces with
module regularity.

Namely the reference, \cite[Proposition~4.4]{Ba13}, guarantees the module
regularity $u\in\Hb^{m,l,k}(M)$ of a solution $u$ of $Lu=f$ if $f$ has matching module
regularity $f\in\Hb^{m-1,l,k}(M)$ and if $u$ is in $\Hb^{m+k,l}(M)$ near
$\overline{C_-}$.
To be precise, this Proposition in \cite{Ba13} is stated making the stronger
  assumption, $f\in\Hb^{m-1+k,l}(M)$. However, the proof goes through
  for just $f\in\Hb^{m-1,l,k}(M)$ in a completely analogous manner to
  the result of Haber and Vasy \cite[Theorem~6.3]{Haber-Vasy:Radial}, where (in the
  boundaryless setting, for a Lagrangian radial set) the result is
  stated in this generality. 

If $f\in\Hb^{m-1,l,k}(M)$, then in particular $f$ is
locally in $\Hb^{m+k-1,l}$ near $\overline{C_-}$, thus, taking into
account that $m+l>1/2$ already there, $u$ is in $\Hb^{m+k,l}$ in that
region by Proposition~\ref{prop:b-source-sink} (by the first case there, i.e.\
in the high regularity
regime). Thus, an application of the closed graph theorem gives the
following boundedness result:

\begin{thm}
Under the assumptions of Theorem~\ref{thm:asymp-Mink-lin},
$L^{-1}$ has the property that it restricts to
$$
L^{-1}:\Hb^{m-1,l,k}\to\Hb^{m,l,k},\ k\geq 0,
$$
as a bounded map.
\end{thm}

In particular, letting $\Omega=\{\wt\frakt\geq 0\}$, where $\wt\frakt=\wh\frakt-\ep_-$ so that it attains the value $0$ within $M\setminus(\overline{C_+}\cup\overline{C_-})$, we have a forward solution operator $S$ of $L$ which maps $\Hb^{m-1,l,k}(\Omega)^\bullet$ into $\Hb^{m,l,k}(\Omega)^\bullet$, given that $m+l<1/2$; let us assume that $m$ is constant in $\Omega$. Here, $\Hb^{m,l,k}(\Omega)^\bullet$ consists of supported distributions at $\pa\Omega\cap C_0^\circ=\{\wt\frakt=0\}$.

\begin{rmk}
  Using the arguments leading to Theorem~\ref{thm:asymp-Mink-lin} in the current, forward problem, setting, but now also using standard energy estimates near the artificial boundary $\wt\frakt=0$ of $\Omega$, we see that if suffices to control the resonances of the asymptotically hyperbolic resolvent in the upper cap $C_+$ in order to ensure the invertibility of the forward problem.
\end{rmk}

\subsection{Algebra properties of $\Hb^{m,-\infty,k}$}
\label{SecMinkAlgebra}

In order to discuss non-linear wave equations on an asymptotically Minkowski space, we need to discuss the algebra properties of $\Hb^{m,-\infty,k}=\bigcup_{l\in\R} \Hb^{m,l,k}$. Even though we are only interested in the space $\Hb^{m,-\infty,k}(\Omega)^\bullet$, we consider $\Hb^{m,-\infty,k}(M)$, where $m$ is constant on $M$ for notational simplicity, and the results we prove below are valid for $\Hb^{m,-\infty,k}(\Omega)^\bullet$ by the same proofs.

We start with the following lemma:

\begin{lemma}
\label{LemmaHbAlgebra}
  Let $l_1,l_2\in\R$, $k>n/2$. Then $\Hb^{0,l_1,k}\cdot \Hb^{0,l_2,k}\subset \Hb^{0,l_1+l_2-1/2,k}$.
\end{lemma}
\begin{proof}
  The generators $\rho\partial_\rho,\rho\partial_v,v\partial_v,\partial_y$ of $\mc M$ take on a simpler form if we blow up the point $(\rho,v)=(0,0)$. It is most convenient to use projective coordinates on the blown-up space, namely:
  \begin{enumerate}[leftmargin=4.5ex]
  \item Near the interior of the front face, we use the coordinates $\wt\rho=\rho\geq 0$ and $s=v/\rho\in\R$. We compute $\rho\partial_\rho=\wt\rho\partial_{\wt\rho}-s\partial_s$, $v\partial_v=s\partial_s$, $\rho\partial_v=\partial_s$; and since $\frac{d\rho}{\rho}\,dv\,dy=d\wt\rho\,ds\,dy$ (this is the b-density from $\Hb^{0,l,k}$), the space $\Hb^{0,l,k}$ becomes
  \[
    A^{l,k}:=\{u\in\wt\rho^l L^2(d\wt\rho\,ds\,dy) \colon \mc A^j u\in\wt\rho^l L^2(d\wt\rho\,ds\,dy),0\leq j\leq k\},
  \]
  where $\mc A$ is the $C^\infty$-module of differential operators generated by $\partial_s,\wt\rho\partial_{\wt\rho},\partial_y$.

  Now, observe that $\wt\rho^l L^2(d\wt\rho\,ds\,dy)=\wt\rho^{l-1/2} L^2(\tfrac{d\wt\rho}{\rho}\,ds\,dy)$; therefore, we can rewrite
  \begin{align*}
    A^{l,k}&=\{u\in\wt\rho^{l-1/2} L^2(\tfrac{d\wt\rho}{\rho}\,ds\,dy) \colon \mc A^j u\in\wt\rho^{l-1/2} L^2(\tfrac{d\wt\rho}{\rho}\,ds\,dy),0\leq j\leq k\} \\
	  &= \wt\rho^{l-1/2}\Hb^k(\tfrac{d\wt\rho}{\rho}\,ds\,dy).
  \end{align*}
  In particular, by the Sobolev algebra property,
  Lemma~\ref{LemmaHsAlgebra}, and the locality of the multiplication, choosing $k>n/2$ ensures that $\wt\rho^{l_1-1/2}\Hb^k\cdot\wt\rho^{l_2-1/2}\Hb^k\subset\wt\rho^{l_1+l_2-1}\Hb^k$, which is to say $A^{l_1,k}\cdot A^{l_2,k}\subset A^{l_1+l_2-1/2,k}$.
  \item Near either corner of the blown-up space, we use $\wt v=v$ and $t=\rho/v$ (say, $\wt v\geq 0,t\geq 0$). We compute $\rho\partial_\rho=t\partial_t$, $v\partial_v=\wt v\partial_{\wt v}-t\partial_t$, $\rho\partial_v=t\wt v\partial_{\wt v}-t^2\partial_t$; and since $\frac{d\rho}{\rho}\,dv\,dy=\frac{dt}{t}\,d\wt v\,dy$, the space $\Hb^{0,l,k}$ becomes
  \[
    B^{l,k}:=\{u\in (t\wt v)^l L^2(\tfrac{dt}{t}\,d\wt v\,dy) \colon \mc B^j u\in(t\wt v)^l L^2(\tfrac{dt}{t}\,d\wt v\,dy),0\leq j\leq k\},
  \]
  where $\mc B$ is the $C^\infty$-module of differential operators generated by $t\partial t,\wt v\partial_{\wt v},\partial_y$. Again, we can rewrite this as
  \[
    B^{l,k}=t^l\wt v^{l-1/2} \Hb^k(\tfrac{dt}{t}\,\tfrac{d\wt v}{\wt v}\,dy),
  \]
  which implies that for $k>n/2$,
  \[
    B^{l_1,k}\cdot B^{l_2,k}\subset t^{l_1+l_2}v^{l_1+l_2-1} \Hb^k(\tfrac{dt}{t}\,\tfrac{d\wt v}{\wt v}\,dy)\subset B^{l_1+l_2-1/2,k}.
  \]
  \end{enumerate}
  To relate these two statements to the statement of the lemma, we use cutoff functions $\chi_A,\chi_B$ to localize within the two coordinate systems. More precisely, choose a cutoff function $\chi\in C^\infty_c(\R_s)$ such that $\chi(s)\equiv 1$ near $s=0$, $\chi(s)=0$ for $|s|\geq 2$, and $\chi^{1/2}\in C^\infty_c(\R_s)$. Then multiplication with $\chi_A(\rho,v):=\chi(v/\rho)$ is a continuous map $\Hb^{0,l,k}\to A^{l,k}$. Indeed, to check this, one simply observes that $\mc M^j\chi_A\in L^\infty$ for all $j\in\N_0$. Similarly, letting $\chi_B(\rho,v):=1-\chi_A(\rho,v)$, multiplication with $\chi_B$ is a continuous map $\Hb^{0,l,k}\to B^{l,k}$. Finally, note that we have $A^{l,k},B^{l,k}\subset \Hb^{0,l,k}$. 

  To put everything together, take $u_j\in \Hb^{0,l_j,k}$ ($j=1,2$), then
  \[
    u_1u_2=(\chi_A u_1)(\chi_A u_2)+(\chi_B u_1)(\chi_B u_2)+(\chi_A u_1)(\chi_B u_2)+(\chi_B u_1)(\chi_A u_2).
  \]
  The first two terms then lie in $\Hb^{0,l_1+l_2-1/2,k}$. To deal with the third term, write
  \[
    (\chi_A u_1)(\chi_B u_2)=(\chi_A^{1/2} u_1)(\chi_A^{1/2}\chi_B u_2)\in A^{l_1,k}\cdot A^{l_2,k}\subset \Hb^{0,l_1+l_2-1/2,k};
  \]
  likewise for the fourth term. Thus, $u_1u_2\in \Hb^{0,l_1+l_2-1/2,k}$, as claimed.
\end{proof}

\begin{rmk}
\label{RmkAlgFF}
  The proof actually shows more, namely that
  \begin{equation}
  \label{EqPreciseAlgStatement} \Hb^{0,l,k}\Hb^{0,l',k}\subset\rho_{\ff}^{-1/2}\Hb^{0,l+l',k},
  \end{equation}
  where $\rho_{\ff}$ is the defining function of the front face $\rho=v=0$, e.g.\ $\rho_{\ff}=(\rho^2+v^2)^{1/2}$. The reason for \eqref{EqPreciseAlgStatement} to be a natural statement is that module- and b-derivatives are the same away from $\rho=v=0$, hence regularity with respect to the module $\mc M$ is, up to a weight, which is a power of $\rho_{\ff}$, the same as b-regularity.

  More abstractly speaking, the above proof shows the following: If $\rho_b$ denotes a boundary defining function of the other boundary hypersurface of $[M;S_+]$, i.e.\ $\pa[M;S_+]\setminus\ff$, then
  \[
    \Hb^{0,l,k}\cong \rho_\ff^{-1/2} (\rho_\ff\rho_b)^l \Hb^k([M;S_+]).
  \]
  Note that one can also show this in one step, introducing the coordinates $\rho_\ff\geq 0$ and $s=v/(\rho+\rho_\ff)\in[-1,1]$ on $[M;S_+]$ in a neighborhood of $\ff$, and mimicking the above proof, which however is computationally less convenient.
\end{rmk}

\begin{rmk}
  We can extend the lemma to $\Hb^{m,l,k}\Hb^{m,l',k}\subset \Hb^{m,l+l'-1/2,k}$ for $m\in\N_0$ using the Leibniz rule to distribute the $m$ b-derivatives among the two factors, and then using the lemma for the case $m=0$.
\end{rmk}

The following corollary, which will play an important role in \S\ref{SecMinkNullform}, improves Lemma~\ref{LemmaHbAlgebra} if we have higher b-regularity.
\begin{cor}
\label{CorDeltaImprovement}
 Let $k>n/2$, $0\leq\delta<1/n$ and $l,l'\in\R$. Then
 \begin{enumerate}
   \item $\Hb^{1,l,k}\Hb^{0,l',k}\subset \Hb^{0,l+l'-1/2+\delta,k}$.
   \item $\Hb^{1,l,k}\Hb^{1,l',k}\subset \Hb^{1,l+l'-1/2+\delta,k}$.
 \end{enumerate}
\end{cor}
\begin{proof}
  Take $s=1/(2\delta)>n/2$, then
  \begin{equation}
  \label{EqInterpolation1} \Hb^{s,l,k}\Hb^{0,l',k}\subset \Hb^{0,l+l',k};
  \end{equation}
  indeed, using the Leibniz rule to distribute the $k$ module derivatives among the two factors and cancelling the weights, this amounts to showing that $\Hb^{s,0,k_1}\Hb^{0,0,k_2}\subset \Hb^{0,0,0}$ for $k_1+k_2\geq k$; but this is true even for $k_1=k_2=0$, since $\Hb^s$ is a multiplier on $\Hb^0$ provided $s>n/2$.
  
  The lemma on the other hand gives
  \begin{equation}
  \label{EqInterpolation2} \Hb^{0,l,k}\Hb^{0,l',k}\subset\rho^{-1/2}\Hb^{0,l+l',k}.
  \end{equation}
  Interpolating in the first factor between \eqref{EqInterpolation1} and \eqref{EqInterpolation2} thus gives the first statement.

  For the second statement, use the Leibniz rule to distribute the one b-derivative to either factor; then, one has to show $\Hb^{1,l,k}\Hb^{0,l',k}\subset \Hb^{0,l+l'-1/2+\delta,k}$, and the same inclusion with $l$ and $l'$ switched, which is what we just proved.
\end{proof}

Lemma~\ref{LemmaHbAlgebra} and the remark following it imply that for $u\in \Hb^{m,l,k}$, $p\geq 1$, with $m\geq 0,k>n/2$, we have $u^p\in \Hb^{m,pl-(p-1)/2,k}$; in fact, $u^p\in\rho_\ff^{-(p-1)/2}\Hb^{m,pl,k}$, see Remark~\ref{RmkAlgFF}. Using Corollary~\ref{CorDeltaImprovement}, we can improve this to the statement $u\in \Hb^{m,l,k}$ $\Rightarrow$ $u^p\in \Hb^{m,pl-(p-1)/2+(p-1)\delta,k}$ for $m\geq 1$.

For non-linearities that only involve powers $u^p$, we can afford to lose differentiability, as at the end of \S\ref{SubsecDeSitterAlgebra}, and gain decay in return, as the following lemma shows.

\begin{lemma}
\label{LemmaMinkDecayForDer}
  Let $\alpha>1/2$, $l\in\R$, $k\in\N_0$. Then $\rho_\ff^{-\alpha}\Hb^{0,l,k}\subset\rho^{1/2-\alpha}\Hb^{-1,l,k}$, where $\rho_\ff=(\rho^2+v^2)^{1/2}$.
\end{lemma}
\begin{proof}
We may assume $l=0$, and that $u$ is supported in $|v|<1$, $\rho<1$.
First, consider the case $k=0$. Let $u\in\rho_\ff^{-\alpha}\Hb^0$, and put
  \[
    \wt u(\rho,v,y)=\int_{-\infty}^v u(\rho,w,y)\,dw,
  \]
  so $\pa_v\wt u=u$. We have to prove $\chi\wt
  u\in\rho^{1/2-\alpha}\Hb^0$ if $\chi\equiv 1$ near $\supp u$, which
  implies $u\in\Hb^{-1}$, as $\pa_v:\Hb^0\to\Hb^{-1}$, and the
  b-Sobolev space are local spaces. But
\begin{equation}\label{eq:half-gain-CS}
    |\wt u(\rho,v,y)|^2\leq\left(\int_{-1}^1 \rho_\ff(\rho,w)^{2\alpha}|u(\rho,w,y)|^2\,dw\right) \int_{-1}^1 \rho_\ff(\rho,w)^{-2\alpha}\,dw;
\end{equation}
  now,
  \[
    \int_{-1}^1\rho_\ff^{-2\alpha}\,dw=\rho^{1-2\alpha}\int_{-1/\rho}^{1/\rho}\frac{dz}{(1+|z|^2)^\alpha}\lesssim\rho^{1-2\alpha}
  \]
  for $\alpha>1/2$, therefore, with the $v$ integral considered on a
  fixed interval, say $|v|<2$ (notice that the right hand side in \eqref{eq:half-gain-CS}
  is independent of $v$!),
  \[
    \iiint \rho^{2\alpha-1}|\wt u(\rho,v,y)|^2\,\frac{d\rho}{\rho}\,dv\,dy \lesssim \iiint \rho_\ff^{2\alpha}|u(\rho,w,y)|^2\,\frac{d\rho}{\rho}\,dw\,dy,
  \]
  proving the claim for $k=0$. Now, $\rho\pa_\rho$ and $\pa_y$ just commute with this calculation, so the corresponding derivatives are certainly well-behaved. On the other hand, $\pa_v\wt u=u$, so the estimates involving at least one $v$-derivative are just those for $u$ itself.
\end{proof}

\begin{cor}
\label{CorMinkPower}
  Let $k,p\in\N$ be such that $k>n/2$, $p\geq 2$. Let $l\in\R$, $u\in \Hb^{0,l,k}$. Then $u^p\in\Hb^{-1,lp-(p-1)/2+1/2-\delta,k}$ with $\delta=0$ if $p\geq 3$ and $\delta>0$ if $p=2$.
\end{cor}
\begin{proof}
  This follows from $u^p\in\rho_\ff^{-(p-1)/2-\delta}\Hb^{0,lp,k}$ and the previous lemma, using that $(p-1)/2+\delta>1/2$ with $\delta$ as stated.
\end{proof}
In other words, we gain the decay $\rho^{1/2-\delta}$ if we give up one derivative.

\subsection{A class of semilinear equations}

We are now set to discuss solutions to non-linear wave equations on an asymptotically Minkowski space. Under the assumptions of Theorem~\ref{thm:asymp-Mink-lin}, we obtain a forward solution operator $S\colon\Hb^{m-1,l,k}(\Omega)^\bullet\to\Hb^{m,l,k}(\Omega)^\bullet$ of $P=\rho^{-(n-2)/2}\rho^{-2}\Box_g\rho^{(n-2)/2}$ provided $|l|<1,m+l<1/2$ and $k\geq 0$.

Undoing the conjugation, we obtain a forward solution operator
\begin{align*}
  \wt S&=\rho^{(n-2)/2}S\rho^{-2}\rho^{-(n-2)/2}, \\
  \wt S&\colon \Hb^{m-1,l+(n-2)/2+2,k}(\Omega)^\bullet\to \Hb^{m,l+(n-2)/2,k}(\Omega)^\bullet
\end{align*}
of $\Box_g$.

Since $g$ is a Lorentzian scattering metric, the natural vector fields to appear in a non-linear equation are scattering vector fields; more generally, since the analysis is carried out on b-spaces, we indeed allow b-vector fields in the following statement:

\begin{thm}
\label{ThmMinkQu}
  Let
  \[
    q\colon \Hb^{m,l+(n-2)/2,k}(\Omega)^\bullet\times \Hb^{m-1,l+(n-2)/2,k}(\Omega;\Tb^*\Omega)^\bullet\to \Hb^{m-1,l+(n-2)/2+2,k}(\Omega)^\bullet
  \]
  be a continuous function with $q(0,0)=0$ such that there exists a continuous non-decreasing function $L\colon\R_{\geq 0}\to\R$ satisfying
  \[
    \|q(u,\bdiff u)-q(v,\bdiff v)\|\leq L(R)\|u-v\|,\quad \|u\|,\|v\|\leq R.
  \]
  Then there is a constant $C_L>0$ so that the following holds: If $L(0)<C_L$, then for small $R>0$, there exists $C>0$ such that for all $f\in \Hb^{m-1,l+(n-2)/2+2,k}(\Omega)^\bullet$ with norm $\leq C$, the equation
  \[
    \Box_g u=f+q(u,\bdiff u)
  \]
  has a unique solution $u\in \Hb^{m,l+(n-2)/2,k}(\Omega)^\bullet$, with norm $\leq R$, that depends continuously on $f$.
\end{thm}
\begin{proof}
  Use the Banach fixed point theorem as in the proof of Theorem~\ref{ThmDSQu}.
\end{proof}

\begin{rmk}
  Here, just as in Theorem~\ref{ThmGlobalDSQu}, we can also allow $q$ to depend on $\Box_g u$ as well.
\end{rmk}

\subsection{Semilinear equations with polynomial non-linearity}

Next, we want to find a forward solution of the semilinear PDE
\[
  \Box_g u=f+cu^p X(u),
\]
where $c\in\CI(M)$, $p\in\N_0$, and $X(u)=\prod_{j=1}^q \rho V_j(u)$ is a $q$-fold product of derivatives of $u$ along scattering vector fields; here, $V_j$ are b-vector fields. Let us assume $p+q\geq 2$ in order for the equation to be genuinely non-linear. We rewrite the PDE as
\begin{align*}
  L(\rho^{-(n-2)/2}u)&=\rho^{-(n-2)/2-2}f+c\rho^{-2}\rho^{(p-1)(n-2)/2}(\rho^{-(n-2)/2}u)^p \\
  &\qquad\times\prod_{j=1}^q \rho V_j(\rho^{(n-2)/2}\rho^{-(n-2)/2}u).
\end{align*}
Introducing $\wt u=\rho^{-(n-2)/2}u$ and $\wt f=\rho^{-(n-2)/2-2}f$ yields the equation
\begin{align}
  L\wt u&=\wt f+c\rho^{(p-1)(n-2)/2-2}\wt u^p\prod_{j=1}^q\rho^{n/2}(f_j\wt u+V_j\wt u) \nonumber\\
\label{eq-diffeq} &=\wt f+c\rho^{(p-1)(n-2)/2+qn/2-2}\wt u^p\prod_{j=1}^q(f_j\wt u+V_j\wt u),
\end{align}
where the $f_j$ are smooth functions. Now suppose that $\wt u\in \Hb^{m,l,k}(\Omega)^\bullet$ with $m+l<1/2,m\geq 1,k>n/2$ (so that $\Hb^{m-1,-\infty,k}(\Omega)^\bullet$ is an algebra), then the second summand of the right hand side of \eqref{eq-diffeq} lies in $\Hb^{m-1,\ell,k}(\Omega)^\bullet$, where
\[
  \ell=(p-1)(n-2)/2+qn/2-2+pl-(p-1)/2+ql-(q-1)/2-1/2.
\]
For this space to lie in $\Hb^{m-1,l,k}(\Omega)^\bullet$ (which we
want in order to be able to apply the solution operator $S$ and land
in $\Hb^{m,l,k}(\Omega)^\bullet$ so that a fixed point argument as in \S\ref{SecStaticDeSitter} can be applied), we thus need $\ell\geq l$, which can be rewritten as
\begin{equation}
\label{eq-pq-condition}(p-1)(l+(n-3)/2)+q(l+(n-1)/2)\geq 2.
\end{equation}
For $m=1$ and $l<1/2-m$ less than, but close to $-1/2$, we thus get the condition
\[
  (p-1)(n-4)+q(n-2)>4.
\]
If there are only non-linearities involving derivatives of $u$, i.e.\ $p=0$, we get the condition $q>1+2/(n-2)$, i.e.\ quadratic non-linearities are fine for $n\geq 5$, cubic ones for $n\geq 4$.

Note that if $q=0$, we can actually choose $m=0$ and $l<1/2$ close to $1/2$, and we have Corollary~\ref{CorMinkPower} at hand. Thus we can improve \eqref{eq-pq-condition} to $(p-1)(1/2+(n-3)/2)>2-1/2$, i.e.\ $p>1+3/(n-2)$, hence quadratic non-linearities can be dealt with if $n\geq 6$, whereas cubic non-linearities are fine as long as $n\geq 4$. Observe that this condition on $p$ always implies $p>1$, which is a natural condition, since $p=1$ would amount to changing $\Box_g$ into $\Box_g-m^2$ (if one chooses the sign appropriately). But the Klein-Gordon operator naturally fits into a scattering framework, as mentioned in the Introduction, i.e.\ requires a different analysis; we will not pursue this further in this paper.

To summarize the general case, note that $\wt u\in \Hb^{m,l,k}(\Omega)^\bullet$ is equivalent to $u\in \Hb^{m,l+(n-2)/2,k}(\Omega)^\bullet$, and $\wt f\in \Hb^{m-1,l,k}(\Omega)^\bullet$ to $f\in \Hb^{m-1,l+(n-2)/2+2,k}(\Omega)^\bullet$; thus:

\begin{thm}
\label{ThmMink}
  Let $|l|<1,m+l<1/2,k>n/2$, and assume that $p,q\in\N_0$, $p+q\geq 2$, satisfy condition \eqref{eq-pq-condition} or the weaker conditions given above in the cases where $p=0$ or $q=0$; let $m\geq 0$ if $q=0$, otherwise let $m\geq 1$. Moreover, let $c\in\CI(M)$ and $X(u)=\prod_{j=1}^q X_ju$, where $X_j$ is a scattering vector field on $M$. Then for small enough $R>0$, there exists a constant $C>0$ such that for all $f\in \Hb^{m-1,l+(n-2)/2+2,k}(\Omega)^\bullet$ with norm $\leq C$, the equation
  \[
    \Box_g u=f+cu^p X(u)
  \]
  has a unique solution $u\in \Hb^{m,l+(n-2)/2,k}(\Omega)^\bullet$, with norm $\leq R$, that depends continuously on $f$.

  The same conclusion holds if the non-linearity is a finite sum of terms of the form $cu^p X(u)$, provided each such term separately satisfies \eqref{eq-pq-condition}.
\end{thm}
\begin{proof}
  Reformulating the PDE in terms of $\wt u$ and $\wt f$ as above, this follows from an application of the Banach fixed point theorem to the map
  \[
    \Hb^{m,l,k}(\Omega)^\bullet\ni\wt u\mapsto S\biggl(\wt f+c\rho^{(p-1)(n-2)/2+qn/2-2}\wt u^p\prod_{j=1}^q(f_j\wt u+V_j\wt u)\biggr)\in \Hb^{m,l,k}(\Omega)^\bullet
  \]
  with $m,l,k$ as in the statement of the theorem. Here, $p+q\geq 2$ and the smallness of $R$ ensure that this map is a contraction on the ball of radius $R$ in $\Hb^{m,l,k}(\Omega)^\bullet$.
\end{proof}

\begin{rmk}
 If the derivatives in the non-linearity only involve module derivatives, we get a slightly better result since we can work with $\wt u\in \Hb^{0,l,k}(\Omega)^\bullet$: Indeed, a module derivative falling on $\wt u$ gives an element of $\Hb^{0,l,k-1}(\Omega)^\bullet$, applied to which the forward solution operator produces an element of $\Hb^{1,l,k-1}(\Omega)^\bullet\subset \Hb^{0,l,k}(\Omega)^\bullet$.

 The numerology works out as follows: In condition \eqref{eq-pq-condition}, we now take $l<1/2$ close to $1/2$, thus obtaining
 \[
   (p-1)(n-2)+qn>4.
 \]
 Thus, in the case that there are only derivatives in the non-linearity, i.e.\ $p=0$, we get $q>1+2/n$, which allows for quadratic non-linearities provided $n\geq 3$.
\end{rmk}

\begin{rmk}
  Observe that we can improve \eqref{eq-pq-condition} in the case $p\geq 1$, $q\geq 1$, $m\geq 1$ by using the $\delta$-improvement from Corollary~\ref{CorDeltaImprovement}, namely, the right hand side of \eqref{eq-diffeq} actually lies in $\Hb^{m-1,\ell,k}(\Omega)^\bullet$, where now
\[
  \ell=(p-1)(n-2)/2+qn/2-2+pl-(p-1)/2+(p-1)\delta+ql-(q-1)/2-1/2+\delta,
\]
which satisfies $\ell\geq l$ if
\[
 (p-1)(l+(n-3)/2+\delta)+q(l+(n-1)/2)+\delta\geq 2,
\]
which for $l<-1/2$ close to $-1/2$ means: $(p-1)(n-4+2\delta)+q(n-2)+2\delta>4$, where $0<\delta<1/n$.
\end{rmk}

\begin{rmk}\label{rmk:Christodoulou}
 Let us compare the above result with Christodoulou's \cite{Ch86}. A special case of his theorem states that the Cauchy problem for the wave equation on Minkowski space with small initial data in\footnote{Note that $n$ is the dimension of Minkowski space here, whereas Christodoulou uses $n+1$.} $H_{k,k-1}(\R^{n-1})$ admits a global solution $u\in H^k_\loc(\R^n)$ with decay $|u(x)|\lesssim (1+(v/\rho)^2)^{-(n-2)/2}$; here, $k=n/2+2$, and $n$ is assumed to $\geq 4$ and even; in case $n=4$, the non-linearity is moreover assumed to satisfy the null condition. The only polynomial non-linearity that we cannot deal with using the above argument is thus the null-form non-linearity in $4$ dimensions.
 
 To make a further comparison possible, we express $H_{k,\delta}(\R^{n-1})$ as a b-Sobolev space on the radial compactification of $\R^{n-1}$: Note that $u\in H_{k,\delta}(\R^{n-1})$ is equivalent to $(\la x\ra D_x)^\alpha u\in\la x\ra^{-\delta}L^2(\R^{n-1})$, $|\alpha|\leq k$. In terms of the boundary defining function $\rho$ of $\partial\overline{\R^{n-1}}$ and the standard measure $d\omega$ on the unit sphere $\Sphere^{n-2}\subset\R^{n-1}$, we have $L^2(\R^{n-1})=L^2(\tfrac{d\rho}{\rho^2}\,\tfrac{dy}{\rho^{n-2}})=\rho^{(n-1)/2}L^2(\tfrac{d\rho}{\rho}\,dy)$, and thus $H_{k,\delta}(\R^{n-1})=\rho^{(n-1)/2+\delta}H^k_\bl(\wt\frakt=0)$. Therefore, converting the Cauchy problem into a forward problem, the forcing lies in $\Hb^{k,(n-1)/2+k-1,0}(\Omega)^\bullet=\Hb^{n/2+2,n+1/2,0}(\Omega)^\bullet$. Comparing this with the space $H^{0,l+(n-2)/2+2,n/2+1}_\bl$ (with $l<1/2$) needed for our argument, we see that Christodoulou's result applies to a regime of fast decay which is disjoint from our slow decay (or even mild growth) regime.
\end{rmk}

\begin{rmk}\label{rmk:Chrusciel}
  In the case of non-linearities $u^p$, the result of Christodoulou
  \cite{Ch86} implies the existence of global solutions to $\Box_g
  u=f+u^p$ if the spacetime dimension $n$ is \emph{even} and $n\geq 4$
  if $p\geq 3$; in even dimensions $n\geq 6$, $p\geq 2$ suffices; the above result extends this to all dimensions satisfying the respective inequalities. In a somewhat similar context, see the work of Chru{\'s}ciel and {\L}{\c{e}}ski \cite{Ch06}, it has been proved that $p\geq 2$ in fact works in all dimensions $n\geq 5$.
\end{rmk}


\subsection{Semilinear equations with null condition}
\label{SecMinkNullform}

With $g$ the Lorentzian scattering metric on an asymptotically Minkowski space satisfying the assumptions of Theorem~\ref{thm:asymp-Mink-lin} as before, define the null form $Q(\scdiff u,\scdiff v)=g^{jk}\partial_j u\partial_k v$, and write $Q(\scdiff u)$ for $Q(\scdiff u,\scdiff u)$. We are interested in solving the PDE
\[
  \Box_g u=Q(\scdiff u)+f.
\]
The previous discussion solves this for $n\geq 5$; thus, let us from now on assume $n=4$. To make the computations more transparent, we will keep the $n$ in the notation and only substitute $n=4$ when needed. Rewriting the PDE in terms of the operator $L=\rho^{-2}\rho^{-(n-2)/2}\Box_g\rho^{(n-2)/2}$ as above, we get
\[
  L\wt u=\wt f+\rho^{-(n-2)/2-2}Q(\scdiff(\rho^{(n-2)/2}\wt u)),
\]
where $\wt u=\rho^{-(n-2)/2}u$ and $\wt f=\rho^{-(n-2)/2-2}f$. We can write $Q(\scdiff u)=\tfrac{1}{2}\Box_g(u^2)-u\Box_g u$, thus the PDE becomes
\begin{align*}
  L\wt u&=\wt f+\rho^{-(n-2)/2-2}\bigl(\tfrac{1}{2}\Box_g(\rho^{n-2}\wt u^2)-\rho^{(n-2)/2}\wt u\Box_g(\rho^{(n-2)/2}\wt u)\bigr) \\
    &=\wt f+\tfrac{1}{2}L(\rho^{(n-2)/2}\wt u^2)-\rho^{(n-2)/2}\wt u L\wt u.
\end{align*}
Since the results of \S\ref{SecMinkAlgebra} give small improvements on the decay of products of $\Hb^{1,*,*}$ functions with $\Hb^{m,*,*}$ functions ($m\geq 0$), one wants to solve this PDE on a function space that keeps track of these small improvements.

\begin{definition}
  For $l\in\R,k\in\N_0$ and $\alpha\geq 0$, define the space $\cX^{l,k,\alpha}:=\{v\in \Hb^{1,l+\alpha,k}(\Omega)^\bullet\colon Lv\in \Hb^{0,l,k}(\Omega)^\bullet\}$ with norm
  \begin{equation}
   \label{EqNullformSpaceNorm} \|v\|_{\cX^{l,k,\alpha}}=\|v\|_{\Hb^{1,l+\alpha,k}(\Omega)^\bullet}+\|Lv\|_{\Hb^{0,l,k}(\Omega)^\bullet}.
  \end{equation}
\end{definition}

By an argument similar to the one used in the proof of Theorem~\ref{ThmDSQu}, we see that $\cX^{l,k,\alpha}$ is a Banach space.

On $\cX^{l,k,\alpha}$, which $\alpha>0$ chosen below, we want to run an iteration argument: Start by defining the operator $T\colon\cX^{l,k,\alpha}\to \Hb^{1,-\infty,k}(\Omega)^\bullet$ by
\[
  T\colon\wt u\mapsto S\bigl(\wt f-\rho^{(n-2)/2}\wt u L\wt u\bigr)+\tfrac{1}{2}\rho^{(n-2)/2}\wt u^2.
\]
Note that $\wt u\in\cX^{l,k,\alpha}$ implies, using Corollary~\ref{CorDeltaImprovement} with $\delta<1/n$,
\begin{align}
    \rho^{(n-2)/2}\wt u^2 & \in \rho^{(n-2)/2}\Hb^{1,2(l+\alpha)-1/2+\delta,k}(\Omega)^\bullet=\Hb^{1,2l+\alpha+(n-3)/2+\delta+\alpha,k}(\Omega)^\bullet, \nonumber\\
  \label{EqTermInclusions} \rho^{(n-2)/2}\wt u L\wt u & \in \Hb^{0,2l+\alpha+(n-3)/2+\delta,k}(\Omega)^\bullet, \\
    S(\rho^{(n-2)/2}\wt u L\wt u) & \in \Hb^{1,2l+\alpha+(n-3)/2+\delta,k}(\Omega)^\bullet, \nonumber
\end{align}
where in the last inclusion, we need to require $1+(2l+\alpha+(n-3)/2+\delta)<1/2$, which for $n=4$ means
\begin{equation}
\label{EqCond1} l<-1/2-(\alpha+\delta)/2;
\end{equation}
let us assume from now on that this condition holds. Furthermore, \eqref{EqTermInclusions} implies $T\wt u\in \Hb^{1,2l+\alpha+(n-3)/2+\delta,k}(\Omega)^\bullet$. Finally, we analyze
\[
  L(T\wt u)\in \Hb^{0,2l+\alpha+(n-3)/2+\delta,k}(\Omega)^\bullet+\frac{1}{2}L(\rho^{(n-2)/2}\wt u^2).
\]
Using that $L$ is a second-order b-differential operator, we have
\begin{align*}
  \rho^{(n-2)/2} & L(\wt u^2) \in 2\rho^{(n-2)/2}\wt u L\wt u+\rho^{(n-2)/2}\Hb^{0,l+\alpha,k}(\Omega)^\bullet \Hb^{0,l+\alpha,k}(\Omega)^\bullet \\
   &\subset \Hb^{0,2l+\alpha+(n-3)/2+\delta,k}(\Omega)^\bullet+\Hb^{0,2(l+\alpha)+(n-3)/2,k}(\Omega)^\bullet \\
   &=\Hb^{0,2l+\alpha+(n-3)/2+\min\{\alpha,\delta\},k}(\Omega)^\bullet,
\end{align*}
which gives
\begin{align*}
  L(&\rho^{(n-2)/2}\wt u^2)\in L(\rho^{(n-2)/2})\wt u^2+\rho^{(n-2)/2}L(\wt u^2) \\
    &\hspace{18ex}+\rho^{(n-2)/2}\Hb^{1,l+\alpha,k}(\Omega)^\bullet \Hb^{0,l+\alpha,k}(\Omega)^\bullet \\
    &\subset \Hb^{1,2l+\alpha+(n-3)/2+\delta+\alpha,k}(\Omega)^\bullet+\Hb^{0,2l+\alpha+(n-3)/2+\min\{\alpha,\delta\},k}(\Omega)^\bullet \\
	&\hspace{12ex}+\Hb^{0,2l+\alpha+(n-3)/2+\delta+\alpha}(\Omega)^\bullet \\
	&= \Hb^{0,2l+\alpha+(n-3)/2+\min\{\alpha,\delta\},k}(\Omega)^\bullet.
\end{align*}
Hence, putting everything together,
\[
  L(T\wt u)\in \Hb^{0,2l+\alpha+(n-3)/2+\min\{\alpha,\delta\},k}(\Omega)^\bullet.
\]
Therefore, we have $T\wt u\in\cX^{l,k,\alpha}$ provided
\begin{align*}
  2l&+\alpha+(n-3)/2+\delta \geq l+\alpha \\
  2l&+\alpha+(n-3)/2+\min\{\alpha,\delta\} \geq l,
\end{align*}
which for $0<\alpha<\delta$ and $n=4$ is equivalent to
\begin{equation}
\label{EqCond2}
  l\geq -1/2-\delta, l\geq -1/2-2\alpha.
\end{equation}
This is consistent with condition \eqref{EqCond1} if $-1/2-(\alpha+\delta)/2>-1/2-2\alpha$, i.e.\ if $\alpha>\delta/3$.

Finally, for the map $T$ to be well-defined, we need $S\wt f\in\cX^{l,k,\alpha}$, hence $\wt f\in\Ran_{\cX^{l,k,\alpha}}L$, which is in particular satisfied if $\wt f\in \Hb^{0,l+\alpha,k}(\Omega)^\bullet$. Indeed, since $1+l+\alpha<1-1/2-(\delta-\alpha)/2<1/2$ by condition \eqref{EqCond1}, the element $S\wt f\in \Hb^{1,l+\alpha,k}(\Omega)^\bullet$ is well-defined.

We have proved:

\begin{thm}
\label{ThmNullMink}
 Let $c\in\C$, $0<\delta<1/4$, $\delta/3<\alpha<\delta$, and let $-1/2-2\alpha\leq l<-1/2-(\alpha+\delta)/2$. Then for small enough $R>0$, there exists a constant $C>0$ such that for all $f\in \Hb^{0,l+3+\alpha,k}(\Omega)^\bullet$ with norm $\leq C$, the equation
 \[
   \Box_g u=f+cQ(\scdiff u)
 \]
 has a unique solution $u\in\cX^{l+1,k,\alpha}$, with norm $\leq R$, that depends continuously on $f$.
\end{thm}

\def\cprime{$'$}

\end{document}